\documentclass[11pt]{article}%
\usepackage{amssymb}
\usepackage{amsfonts}
\usepackage{amsmath}
\usepackage{graphicx}%
\providecommand{\U}[1]{\protect\rule{.1in}{.1in}}
\newtheorem{theorem}{Theorem}[section]

\newtheorem{corollary}[theorem]{Corollary}

\newtheorem{definition}[theorem]{Definition}

\newtheorem{lemma}[theorem]{Lemma}

\newtheorem{proposition}[theorem]{Proposition}
\newtheorem{remark}[theorem]{Remark}

\newenvironment{proof}[1][Proof]{\textbf{#1.} }{\hfill\rule{0.5em}{0.5em}}
{\catcode`\@=11\global\let\AddToReset=\@addtoreset
\AddToReset{equation}{section}

\AddToReset{theorem}{section}

\begin{document}

\title{Stability properties for quasilinear parabolic equations with measure data and applications}
\author{Marie-Fran\c{c}oise BIDAUT-VERON\thanks{Laboratoire de Math\'{e}matiques et
Physique Th\'{e}orique, CNRS UMR 7350, Facult\'{e} des Sciences, 37200 Tours
France. E-mail: veronmf@univ-tours.fr}
\and Hung NGUYEN\ QUOC\thanks{Laboratoire de Math\'{e}matiques et Physique
Th\'{e}orique, CNRS UMR 7350, Facult\'{e} des Sciences, 37200 Tours France.
E-mail: Hung.Nguyen-Quoc@lmpt.univ-tours.fr}}
\date{.}
\maketitle

\begin{abstract}
Let $\Omega$ be a bounded domain of $\mathbb{R}^{N}$, and $Q=\Omega
\times(0,T).$ We first study problems of the model type
\[
\left\{
\begin{array}
[c]{l}%
{u_{t}}-{\Delta_{p}}u=\mu\qquad\text{in }Q,\\
{u}=0\qquad\text{on }\partial\Omega\times(0,T),\\
u(0)=u_{0}\qquad\text{in }\Omega,
\end{array}
\right.
\]
where $p>1$, $\mu\in\mathcal{M}_{b}(Q)$ and $u_{0}\in L^{1}(\Omega).$ Our main
result is a \textit{stability theorem }extending the results of Dal Maso,
Murat, Orsina, Prignet, for the elliptic case, valid for quasilinear operators
$u\longmapsto\mathcal{A}(u)=$div$(A(x,t,\nabla u))$\textit{. }

As an application, we consider perturbed problems\textit{ of type}

\[
\left\{
\begin{array}
[c]{l}%
{u_{t}}-{\Delta_{p}}u+\mathcal{G}(u)=\mu\qquad\text{in }Q,\\
{u}=0\qquad\text{on }\partial\Omega\times(0,T),\\
u(0)=u_{0}\qquad\text{in }\Omega,
\end{array}
\right.
\]
where $\mathcal{G}(u)$ may be an absorption or a source term$.$ In the model
case $\mathcal{G}(u)=\pm\left\vert u\right\vert ^{q-1}u$ $(q>p-1),$ or
$\mathcal{G}$ has an exponential type. We give existence results when $q$ is
subcritical, or when the measure $\mu$ is good in time and satisfies suitable
capacity conditions.

\end{abstract}

\pagebreak\medskip

\section{Introduction}

Let $\Omega$ be a bounded domain of $\mathbb{R}^{N}$, and $Q=\Omega
\times(0,T),$ $T>0.$ We denote by $\mathcal{M}_{b}(\Omega)$ and $\mathcal{M}%
_{b}(Q)$ the sets of bounded Radon measures on $\Omega$ and $Q$ respectively.
We are concerned with the problem
\begin{equation}
\left\{
\begin{array}
[c]{l}%
{u_{t}}-\text{div}(A(x,t,\nabla u))=\mu\qquad\text{in }Q,\\
{u}=0\qquad\qquad\qquad\qquad\text{on }\partial\Omega\times(0,T),\\
u(0)=u_{0}\qquad\qquad\qquad\text{in }\Omega,
\end{array}
\right.  \label{pmu}%
\end{equation}
where $\mu\in\mathcal{M}_{b}(Q)$, $u_{0}\in L^{1}(\Omega)$ and $A$ is a
Caratheodory function on $Q\times\mathbb{R}^{N}$, such that for $a.e.$
$(x,t)\in Q,$ and any $\xi,\zeta\in\mathbb{R}^{N},$
\begin{equation}
A(x,t,\xi).\xi\geqq c_{1}\left\vert \xi\right\vert ^{p},\qquad\left\vert
A(x,t,\xi)\right\vert \leqq a(x,t)+c_{2}\left\vert \xi\right\vert
^{p-1},\qquad c_{1},c_{2}>0,a\in L^{p^{\prime}}(Q),\label{condi1}%
\end{equation}%
\begin{equation}
(A(x,t,\xi)-A(x,t,\zeta)).\left(  \xi-\zeta\right)  >0\qquad\text{ if }\xi
\neq\zeta.\label{condi2}%
\end{equation}
This includes the model problem
\begin{equation}
\left\{
\begin{array}
[c]{l}%
{u_{t}}-\Delta_{p}u=\mu\qquad\text{in }Q,\\
{u}=0\qquad\qquad\qquad\text{on }\partial\Omega\times(0,T),\\
u(0)=u_{0}\qquad\qquad\text{in }\Omega,
\end{array}
\right.  \label{pmu1}%
\end{equation}
where $\Delta_{p}$ is the $p$-Laplacian defined by $\Delta_{p}u=\text{div}%
(|\nabla u|^{p-2}\nabla u)$ with $p>1$.\newline As an application, we consider
problems with a nonlinear term of order $0$:
\begin{equation}
\left\{
\begin{array}
[c]{l}%
{u_{t}}-\text{div}(A(x,\nabla u))+\mathcal{G}(u)=\mu\qquad\text{in }Q,\\
{u}=0\qquad\qquad\qquad\qquad\qquad\text{on }\partial\Omega\times(0,T),\\
u(0)=u_{0}\qquad\qquad\qquad\qquad\text{in }\Omega,
\end{array}
\right.  \label{pga}%
\end{equation}
where $A$ is a Caratheodory function on $\Omega\times\mathbb{R}^{N}$, such
that, for $a.e.$ $x\in\Omega,$ and any $\xi,\zeta\in\mathbb{R}^{N},$
\begin{equation}
A(x,\xi).\xi\geqq c_{1}\left\vert \xi\right\vert ^{p},\qquad\left\vert
A(x,\xi)\right\vert \leqq c_{2}\left\vert \xi\right\vert ^{p-1},\qquad
c_{3},c_{4}>0,\label{condi3}%
\end{equation}%
\begin{equation}
(A(x,\xi)-A(x,\zeta)).\left(  \xi-\zeta\right)  >0\text{ if }\xi\neq
\zeta,\label{condi4}%
\end{equation}
and $\mathcal{G}(u)$ may be an absorption or a source term, and possibly
depends on $(x,t)\in Q.$ The model problem is the case where $\mathcal{G}$ has
a power-type $\mathcal{G}(u)=\pm\left\vert u\right\vert ^{q-1}u$ $(q>p-1),$ or
an exponential type.\medskip\newline First make a brief survey of the elliptic
associated problem:
\[
\left\{
\begin{array}
[c]{c}%
-\text{div}(A(x,\nabla u))=\mu\qquad\text{in }\Omega,\\
u=0\qquad\qquad\qquad\text{on }\partial\Omega,
\end{array}
\right.
\]
with $\mu\in\mathcal{M}_{b}(\Omega)$ and assumptions (\ref{condi3}),
(\ref{condi4}). When $p=2,$ $A(x,\nabla u)=\nabla u$ existence and uniqueness
are proved for general elliptic operators by duality methods in \cite{St}. For
$p>2-1/N,$ the existence of solutions in the sense of distributions is
obtained in \cite{BoGa89} and \cite{BoGa92}. The condition on $p$ ensures that
the gradient $\nabla u$ is well defined in $(L^{1}\left(  \Omega\right)
)^{N}.$ For general $p>1,$ new classes of solutions are introduced, first when
$\mu\in L^{1}(\Omega),$ such as \textit{entropy solutions}, and
\textit{renormalized solutions}, see \cite{BBGGPV}, and also \cite{Ra1}, and
existence and uniqueness is obtained. For any $\mu\in\mathcal{M}_{b}(\Omega)$
the main work is done in \cite[Theorems 3.1, 3.2]{DMOP}, where not only
existence is proved, but also a stability result, fundamental for
applications. Uniqueness is still an open problem.$\medskip$

Next we make a brief survey about problem (\ref{pmu}).\medskip

The first studiess concern the case $\mu\in L^{p^{\prime}}(Q)$ and $u_{0}\in
L^{2}(\Omega)$, where existence and uniqueness is obtained by variational
methods, see \cite{Li}. In the general case $\mu\in\mathcal{M}_{b}(Q)$ and
$u_{0}\in\mathcal{M}_{b}(\Omega),$ the pionner results come from
\cite{BoGa89}, proving the existence of solutions in the sense of
distributions for%
\begin{equation}
p>p_{1}=2-\frac{1}{N+1}, \label{rangep}%
\end{equation}
see also \cite{Ra}, \cite{Ra0}, and \cite{BDGO97}. The approximated solutions
of (\ref{pmu}) lie in Marcinkiewicz spaces $u\in L^{p_{c},\infty}\left(
Q\right)  $ and $\left\vert \nabla u\right\vert \in L^{m_{c},\infty}\left(
Q\right)  ,$ where
\begin{equation}
p_{c}=p-1+\frac{p}{N},\qquad m_{c}=p-\frac{N}{N+1}. \label{crit}%
\end{equation}
This condition (\ref{rangep}) ensures that $u$ and $\left\vert \nabla
u\right\vert $ belong to $L^{1}\left(  Q\right)  $, since $m_{c}>1$ means
$p>p_{1}$ and $p_{c}>1$ means $p>2N/(N+1).$ Uniqueness follows in the case
$p=2$, $A(x,t,\nabla u)=\nabla u,$ by duality methods, see \cite{Pe07}.

For $\mu\in L^{1}(Q)$, uniqueness is obtained in new classes of solutions:
\textit{entropy solutions}, and \textit{renormalized solutions}, see
\cite{BlMu}, \cite{Pr97}, see also \cite{An} for a semi-group approach.

Then a class of \textit{regular }measures is studied in \cite{DrPoPr}, where a
notion of parabolic capacity $c_{p}^{Q}$ is introduced, defined by
\[
c_{p}^{Q}(E)=\inf(\inf_{E\subset U\text{ open}\subset Q}\{||u||_{W}:u\in
W,u\geqq\chi_{U}\quad a.e.\text{ in }Q\}),
\]
for any Borel set $E\subset Q,$ where
\begin{align*}
X  &  ={{L^{p}}(0,T;W_{0}^{1,p}(\Omega)\cap{L^{2}}(\Omega)),}\\
W  &  =\left\{  {z:z\in}X{,\quad{z_{t}}\in X}^{\prime}\right\}  ,\text{
embedded with the norm }||u||_{W}=||u||_{{X}}+||u_{t}||_{{X}^{\prime}}.
\end{align*}
Let $\mathcal{M}_{0}(Q)$ be the set of Radon measures $\mu$ on $Q$ that do not
charge the sets of zero $c_{p}^{Q}$-capacity:
\[
\forall E\text{ Borel set }\subset Q,\quad c_{p}^{Q}(E)=0\Longrightarrow
\left\vert \mu(E)\right\vert =0.
\]
Then existence and uniqueness of renormalized solutions holds for any measure
$\mu\in\mathcal{M}_{b}(\Omega)\cap\mathcal{M}_{0}(Q),$ called regular (or
diffuse) and $u_{0}\in L^{1}(\Omega)$, and $p>1$. The equivalence with the
notion of entropy solutions is shown in \cite{DrPr}; see also \cite{BlPeRe}
for more general equations.\medskip

Next consider \textit{any} measure $\mu\in\mathcal{M}_{b}(Q).$ Let
$\mathcal{M}_{s}(Q)$ be the set of all bounded Radon measures on $Q$ with
support on a set of zero $c_{p}^{Q}$capacity, also called \textit{singular}.
Let $\mathcal{M}_{b}^{+}(Q),\mathcal{M}_{0}^{+}(Q),\mathcal{M}_{s}^{+}(Q)$ be
the positive cones of $\mathcal{M}_{b}(Q),\mathcal{M}_{0}(Q),\mathcal{M}%
_{s}(Q).$ From \cite{DrPoPr}, $\mu$ can be written (in a unique way) under the
form%
\begin{equation}
\mu=\mu_{0}+\mu_{s},\qquad\mu_{0}\in\mathcal{M}_{0}(Q),\quad\mu_{s}=\mu
_{s}^{+}-\mu_{s}^{-},\qquad\mu_{s}^{+},\mu_{s}^{-}\in\mathcal{M}_{s}^{+}(Q),
\label{deo}%
\end{equation}
and $\mu_{0}\in$ $\mathcal{M}_{0}(Q)$ admits (at least) a decomposition under
the form%
\begin{equation}
\mu_{0}=f-\operatorname{div}g+h_{t},\qquad f\in L^{1}(Q),\quad g\in
(L^{p^{\prime}}(Q))^{N},\quad h\in{X}, \label{dec}%
\end{equation}
and we write $\mu_{0}=(f,g,h).$ The solutions of (\ref{pmu}) are searched in a
renormalized sense linked to this decomposition, introduced in \cite{BlMu}%
,\cite{Pe08}. \ In the range (\ref{rangep}) the existence of a renormalized
solution relative to the decomposition (\ref{dec}) is proved in \cite{Pe08},
using suitable approximations of $\mu_{0}$ and $\mu_{s}$. Uniqueness is still
open, as well as in the elliptic case. $\medskip$

Next consider the problem (\ref{pga}). First we consider the case of
an\textit{ absorption term}: $\mathcal{G}(u)u\geqq0.\medskip$

\noindent Let us recall the case $p=2,$ $A(x,\nabla u)=\nabla u$ and
$\mathcal{G}(u)=|u|^{q-1}u$ $(q>1)$. The first results concern the case
$\mu=0$ and $u_{0}$ is a Dirac mass in $\Omega$, see \cite{BrFr}: existence
holds if and only if $q<(N+2)/N.$ Then optimal results are given in
\cite{BaPi2}, for any $\mu\in\mathcal{M}_{b}({Q})$ and $u_{0}\in
\mathcal{M}_{b}(\Omega)$. Here two capacities are involved: the elliptic
Bessel capacity $C_{\alpha,k}$, ($\alpha,k>1)$ defined, for any Borel set
$E\subset\mathbb{R}^{N},$ by
\[
C_{\alpha,k}(E)=\inf\{||\varphi||_{L^{k}(\mathbb{R}^{N})}:\varphi\in
L^{k}(\mathbb{R}^{N}),\quad G_{\alpha}\ast\varphi\geqq\chi_{E}\},
\]
where $G_{\alpha}$ is the Bessel kernel of order $\alpha$; and a capacity
$c_{\mathbf{G},k}$ $(k>1)$ adapted to the operator of the heat equation of
kernel $\mathbf{G}(x,t)=\chi_{\left(  0,\infty\right)  }(4\pi t)^{-N/2}%
e^{-\left\vert x\right\vert ^{2}/4t}:$ for any Borel set $E\subset
\mathbb{R}^{N+1},$
\[
c_{\mathbf{G},k}(E)=\inf\{||\varphi||_{L^{k}(\mathbb{R}^{N+1})}:\varphi\in
L^{k}(\mathbb{R}^{N+1}),\quad\mathbf{G}\ast\varphi\geqq\chi_{E}\}.
\]
From \cite{BaPi2}, there exists a solution if and only if $\mu$ does not
charge the sets of $c_{\mathbf{G},q^{\prime}}(E)$ capacity zero and $u_{0}$
does not charge the sets of $C_{2/q,q^{\prime}}$ capacity zero. Observe that
one can reduce to a zero initial data, by considering the measure $\mu
+u_{0}\otimes\delta_{0}^{t}$ in $\Omega\times\left(  -T,T\right)  ,$ where
$\otimes$ is the tensorial product and $\delta_{0}^{t}$ is the Dirac mass in
time at $0$.\medskip

For $p\neq2$ such a linear parabolic capacity cannot be used. Most of the
contributions are relative to the case $\mu=0$ with $\Omega$ bounded, or
$\Omega=\mathbb{R}^{N}$. The case where $u_{0}$ is a Dirac mass in $\Omega$ is
studied in \cite{Gm}, \cite{KaVa} when $p>2$, and \cite{ChQiWa} when $p<2$.
Existence and uniqueness hold in the subcritical case $q<p_{c}.$ If $q$ $\geqq
p_{c}$ and $q>1$, there is no solution with an isolated singularity at $t=0$.
For $q<p_{c},$ and $u_{0}\in\mathcal{M}_{b}^{+}(\Omega),$ the existence is
obtained in the sense of distributions in \cite{Zh}, and for any $u_{0}%
\in\mathcal{M}_{b}(\Omega)$ in \cite{BiChVe}. The case $\mu\in$ $L^{1}(Q),$
$u_{0}=0$ is treated in \cite{DAOr}, and $\mu\in$ $L^{1}(Q),$ $u_{0}%
=L^{1}(\Omega)$ in \cite{AndSbWi} where $\mathcal{G}$ can be multivalued. The
case $\mu\in$ $\mathcal{M}_{0}(Q)$ is studied in \cite{PePoPor}, with a new
formulation of the solutions, and existence and uniqueness is obtained for any
function $\mathcal{G}\in C(\mathbb{R)}$ such that $\mathcal{G}(u)u\geqq0.$ Up
to our knowledge, up to now no existence results have been obtained for a
general measure $\mu\in$ $\mathcal{M}_{b}(Q).$\medskip

The case of a source term $\mathcal{G}(u)=-u^{q}$ with $u\geqq0$ has beeen
treated in \cite{BaPi1} for $p=2,$ where optimal conditions are given for
existence. As in the absorption case the arguments of proofs cannot be
extended to general $p.$

\section{Main results}

In \textit{all the sequel} we suppose that $p$ satisfies (\ref{rangep}). Then
\[
X={{L^{p}}(0,T;W_{0}^{1,p}(\Omega)),\qquad X}^{\prime}={{L^{p^{\prime}}%
}(0,T;W^{-1,p^{\prime}}(\Omega)).}%
\]

We first study problem (\ref{pmu}). In Section \ref{prox} we give some
approximations of $\mu\in\mathcal{M}_{b}(Q),$ useful for the applications. In
Section \ref{defsol} we recall the definition of renormalized solutions, that
we call R-solutions of (\ref{pmu}), relative to the decomposition (\ref{dec})
of $\mu_{0}$, and study some of their properties.

Our main result is a \textit{stability theorem} for problem (\ref{pmu}),
proved in Section \ref{cv}, extending to the parabolic case the stability
result of \cite[Theorem 3.4]{DMOP}, and improving the result of \cite{Pe08}:

\begin{theorem}
\label{sta} Let $A:Q\times\mathbb{R}^{N}\longmapsto\mathbb{R}^{N}$ satisfying
(\ref{condi1}),(\ref{condi2}). Let $u_{0}\in L^{1}(\Omega)$, and
\[
\mu=f-\operatorname{div}g+h_{t}+\mu_{s}^{+}-\mu_{s}^{-}\in\mathcal{M}_{b}%
({Q}),
\]
with $f\in L^{1}(Q),g\in(L^{p^{\prime}}(Q))^{N},$ $h\in X$ and $\mu_{s}%
^{+},\mu_{s}^{-}\in\mathcal{M}_{s}^{+}(Q).$ Let $u_{0,n}\in L^{1}(\Omega),$
\[
\mu_{n}=f_{n}-\operatorname{div}g_{n}+(h_{n})_{t}+\rho_{n}-\eta_{n}%
\in\mathcal{M}_{b}({Q}),
\]
with \ $f_{n}\in L^{1}(Q),g_{n}\in(L^{p^{\prime}}(Q))^{N},h_{n}\in X,$ and
$\rho_{n},\eta_{n}\in\mathcal{M}_{b}^{+}({Q}),$ such that
\[
\rho_{n}=\rho_{n}^{1}-\operatorname{div}\rho_{n}^{2}+\rho_{n,s},\qquad\eta
_{n}=\eta_{n}^{1}-\mathrm{\operatorname{div}}\eta_{n}^{2}+\eta_{n,s},
\]
with $\rho_{n}^{1},\eta_{n}^{1}\in L^{1}(Q),\rho_{n}^{2},\eta_{n}^{2}%
\in(L^{p^{\prime}}(Q))^{N}$ and $\rho_{n,s},\eta_{n,s}\in\mathcal{M}_{s}%
^{+}(Q).$ Assume that
\[
\sup_{n}\left\vert {{\mu_{n}}}\right\vert ({Q})<\infty,
\]
and $\left\{  u_{0,n}\right\}  $ converges to $u_{0}$ strongly in
$L^{1}(\Omega),$ $\left\{  f_{n}\right\}  $ converges to $f$ weakly in
$L^{1}(Q),$ $\left\{  g_{n}\right\}  $ converges to $g$ strongly in
$(L^{p^{\prime}}(Q))^{N}$, $\left\{  h_{n}\right\}  $ converges to $h$
strongly in $X$, $\left\{  \rho_{n}\right\}  $ converges to $\mu_{s}^{+}$ and
$\left\{  \eta_{n}\right\}  $ converges to $\mu_{s}^{-}$ in the narrow
topology of measures; and $\left\{  \rho_{n}^{1}\right\}  ,\left\{  \eta
_{n}^{1}\right\}  $ are bounded in $L^{1}(Q)$, and $\left\{  \rho_{n}%
^{2}\right\}  ,\left\{  \eta_{n}^{2}\right\}  $ bounded in $(L^{p^{\prime}%
}(Q))^{N}$. Let $\left\{  u_{n}\right\}  $ be a sequence of R-solutions of
\begin{equation}
\left\{
\begin{array}
[c]{l}%
{u_{n,t}}-\text{div}(A(x,t,\nabla u_{n}))=\mu_{n}\qquad\text{in }Q,\\
{u}_{n}=0\qquad\text{on }\partial\Omega\times(0,T),\\
u_{n}(0)=u_{0,n}\qquad\text{in }\Omega.
\end{array}
\right.  \label{pmun}%
\end{equation}
relative to the decomposition $(f_{n}+\rho_{n}^{1}-\eta_{n}^{1},g_{n}+\rho
_{n}^{2}-\eta_{n}^{2},h_{n})$ of $\mu_{n,0}.$ Let $v_{n}=u_{n}-h_{n}.$ Then up
to a subsequence, $\left\{  u_{n}\right\}  $ converges $a.e.$ in $Q$ to a
R-solution $u$ of (\ref{pmu}), and $\left\{  v_{n}\right\}  $ converges $a.e.$
in $Q$ to $v=u-h.$ Moreover, $\left\{  \nabla u_{n}\right\}  ,\left\{  \nabla
v_{n}\right\}  $ converge respectively to $\nabla u,\nabla v$ $a.e.$ in $Q,$
and $\left\{  T_{k}(u_{n})\right\}  ,\left\{  T_{k}(v_{n})\right\}  $ converge
to $T_{k}(u),$ $T_{k}(v)$ strongly in $X$ for any $k>0$.\bigskip
\end{theorem}

In Section \ref{appli} we give applications to problems of type (\ref{pga}%
).\medskip

We first give an existence result of subcritical type, valid for any measure
$\mu\in\mathcal{M}_{b}(Q):$

\begin{theorem}
\label{new} Let $A:Q\times\mathbb{R}^{N}\rightarrow\mathbb{R}^{N}$ satisfying
(\ref{condi1}), (\ref{condi2}) with $a\equiv0$. Let $(x,t,r)\mapsto
\mathcal{G}(x,t,r)$ be a Caratheodory function on $Q\times\mathbb{R}$ and
$G\in C(\mathbb{R}^{+})$ be a nondecreasing function with values in
$\mathbb{R}^{+},$ such that
\begin{equation}
\left\vert \mathcal{G}(x,t,r)\right\vert \leqq G(|r|)\quad\text{for
}a.e.\text{ }(x,t)\in Q\text{ and any }~r\in\mathbb{R}, \label{isp}%
\end{equation}%
\begin{equation}
\int_{1}^{\infty}G(s)s^{-1-p_{c}}ds<\infty. \label{asg}%
\end{equation}

\noindent(i) Suppose that $\mathcal{G}(x,t,r)r\geqq0,$ for $a.e.$ $(x,t)$ in
$Q$ and any $r\in\mathbb{R}$. Then, for any $\mu\in\mathcal{M}_{b}(Q)$ and
$u_{0}\in L^{1}(\Omega),$ there exists a R-solution u of problem
\begin{equation}
\left\{
\begin{array}
[c]{l}%
{u_{t}}-\text{div}(A(x,t,\nabla u))+\mathcal{G}(u)=\mu\qquad\text{in }Q,\\
{u}=0\qquad\text{in }\partial\Omega\times(0,T),\\
u(0)=u_{0}\qquad\text{in }\Omega.
\end{array}
\right.  \label{pro0}%
\end{equation}
(ii) Suppose that $\mathcal{G}(x,t,r)r\leqq0,$ for $a.e.$ $(x,t)\in Q$ and any
$r\in\mathbb{R},$ and $u_{0}\geqq0,\mu\geqq0$. There exists $\varepsilon>0$
such that for any $\lambda>0$, any $\mu\in\mathcal{M}_{b}(Q)$ and $u_{0}\in
L^{1}(\Omega)$ with $\lambda+|\mu|(Q)+||u_{0}||_{L^{1}(\Omega)}\leqq
\varepsilon$, problem
\begin{equation}
\left\{
\begin{array}
[c]{l}%
{u_{t}}-\text{div}(A(x,t,\nabla u ))+\lambda\mathcal{G}(u)=\mu\qquad\text{in
}Q,\\
{u}=0\qquad\text{in }\partial\Omega\times(0,T),\\
u(0)=u_{0}\qquad\text{in }\Omega,
\end{array}
\right.  \label{pro1}%
\end{equation}
admits a nonnegative R-solution.$\medskip$
\end{theorem}

In particular for any $0<q<p_{c},$ if $\mathcal{G}(u)=\left\vert u\right\vert
^{q-1}u,$ existence holds for any measure $\mu\in\mathcal{M}_{b}(Q)$; if
$\mathcal{G}(u)=-\left\vert u\right\vert ^{q-1}u,$ existence holds for $\mu$
small enough. In the supercritical case $q\geqq p_{c},$ the class of
"admissible" measures, for which there exist solutions, is not known.
$\medskip$

Next we give new results relative to \textit{measures that have a good
behaviour in }$t,$ based on recent results of \cite{BiNQVe} relative to the
elliptic case. We recall the notions of (truncated) W\"{o}lf potential for any
nonnegative measure $\omega\in\mathcal{M}^{+}(\mathbb{R}^{N})$ any $R>0,$
$x_{0}\in\mathbb{R}^{N},$
\[
\mathbf{W}_{1,p}^{R}[\omega]\left(  x_{0}\right)  =\int_{0}^{R}\left(
t^{p-N}\omega(B(x_{0},t))\right)  ^{\frac{1}{p-1}}\frac{dt}{t}.
\]
Any measure $\omega\in\mathcal{M}_{b}(\Omega)$ is identified with its
extension by $0$ to $\mathbb{R}^{N}.$ In case of absorption, we obtain the following:

\begin{theorem}
\label{main1}Let $A:\Omega\times\mathbb{R}^{N}\rightarrow\mathbb{R}^{N}$
satisfying (\ref{condi3}),(\ref{condi4}). Let $p<N$, $q>p-1,$ $\mu
\in\mathcal{M}_{b}(Q)$, $f\in L^{1}(Q)$ and $u_{0}\in L^{1}(\Omega)$. Assume
that%
\[
\left\vert \mu\right\vert \leqq\omega\otimes F,\text{ \quad with }\omega
\in\mathcal{M}_{b}^{+}(\Omega),F\in L^{1}((0,T)),F\geqq0,
\]
and $\omega$ does not charge the sets of $C_{p,\frac{q}{q+1-p}}$-capacity
zero. Then there exists a R- solution $u$ of problem
\begin{equation}
\left\{
\begin{array}
[c]{l}%
u_{t}-\text{div}(A(x,\nabla u))+|u|^{q-1}u=f+\mu\qquad\text{in }Q,\\
{u}=0\qquad\text{on }\partial\Omega\times(0,T),\\
u(0)=u_{0}\qquad\text{in }\Omega.
\end{array}
\right.  \label{mainprob1}%
\end{equation}

\end{theorem}

We show that some of these measures may not lie in $\mathcal{M}_{0}(Q),$ which
improves the existence results of \cite{PePoPor}, see Proposition \ref{mzero}
and Remark \ref{pari}. Otherwise our result can be extended to a more general
function $\mathcal{G},$ see Remark \ref{exten}. We also consider a source term:

\begin{theorem}
\label{120410} Let $A:\Omega\times\mathbb{R}^{N}\rightarrow\mathbb{R}^{N}$
satisfying (\ref{condi3}), (\ref{condi4}). Let $p<N$, $q>p-1$. Let $\mu
\in\mathcal{M}_{b}^{+}(Q),$ and $u_{0}\in L^{\infty}(\Omega),u_{0}\geqq0$.
Assume that%
\[
\mu\leqq\omega\otimes\chi_{(0,T)},\text{ \quad with }\omega\in\mathcal{M}%
_{b}^{+}(\Omega).
\]
Then there exist $\lambda_{0}=\lambda_{0}(N,p,q,c_{3},c_{4},\text{diam}%
\Omega)$ and $b_{0}=b_{0}(N,p,q,c_{3},c_{4},\mathrm{diam}\Omega)$ such that,
if
\begin{equation}
\omega(E)\leqq\lambda_{0}C_{p,\frac{q}{q-p+1}}(E),\quad\forall E\text{
compact}\subset\mathbb{R}^{N},\qquad||u_{0}||_{\infty,\Omega}\leqq
b_{0},\label{051120132}%
\end{equation}
there exists a nonnegative R-solution $u$ of problem
\begin{equation}
\left\{
\begin{array}
[c]{l}%
u_{t}-\text{div}(A(x,\nabla u))=u^{q}+\mu\qquad\text{in }Q,\\
u=0\qquad\text{on }\partial\Omega\times(0,T),\\
u(0)=u_{0}\qquad\text{in }\Omega,
\end{array}
\right.  \label{pro3}%
\end{equation}
which satisfies, $a.e.$ in $Q,$
\begin{equation}
{u(x,t)}\leqq C\mathbf{W}_{1,p}^{2\mathrm{diam}_{\Omega}}[\omega
](x)+2||u_{0}||_{L^{\infty}},\label{maw}%
\end{equation}
where $C=C(N,p,c_{3},c_{4})$.\bigskip
\end{theorem}

Corresponding results in case where $\mathcal{G}$ has exponential type are
given at Theorems \ref{expo} and \ref{MTH1}.

\section{Approximations of measures\label{prox}}

For any open set $\varpi$ of $\mathbb{R}^{m}$ and $F\in(L^{k}(\varpi))^{\nu},$
$k\in\left[  1,\infty\right]  ,m,\nu\in\mathbb{N}^{\ast},$ we set $\left\Vert
F\right\Vert _{k,\varpi}=\left\Vert F\right\Vert _{(L^{k}(\varpi))^{\nu}}%
.$\bigskip

First we give approximations of nonnegative measures in $\mathcal{M}_{0}(Q).$
We recall that any measure $\mu\in\mathcal{M}_{0}(Q)\cap\mathcal{M}_{b}(Q)$
admits a decomposition under the form $\mu=(f,g,h)$ given by (\ref{dec}).
Conversely, any measure of this form, \textit{such that} $h\in L^{\infty}(Q),$
lies in $\mathcal{M}_{0}(Q),$ see \cite[Proposition 3.1]{PePoPor}.

\begin{lemma}
\label{att}Let $\mu\in\mathcal{M}_{0}(Q)\cap\mathcal{M}_{b}^{+}(Q)$ and
$\varepsilon>0$.

\noindent(i) Then, we can find a decomposition $\mu=(f,g,h)$ with $f\in
L^{1}(Q),g\in(L^{p^{\prime}}(Q))^{N},h\in X$ such that
\begin{equation}
||f||_{1,Q}+\left\Vert g\right\Vert _{p^{\prime},Q}+||h||_{X}\leqq
(1+\varepsilon)\mu(Q),\qquad\left\Vert g\right\Vert _{p^{\prime},Q}%
+||h||_{X}\leqq\varepsilon. \label{deco}%
\end{equation}
(ii) Furthermore, there exists a sequence of measures $\mu_{n}=(f_{n}%
,g_{n},h_{n}),$ such that $f_{n},g_{n},h_{n}\in C_{c}^{\infty}(Q)$ and
strongly converge to $f,g,h$ in $L^{1}(Q),(L^{p^{\prime}}(Q))^{N}$ and $X$
respectively, and $\mu_{n}$ converges to $\mu$ in the narrow topology, and
satisfying
\begin{equation}
||f_{n}||_{1,Q}+\left\Vert g_{n}\right\Vert _{p^{\prime},Q}+||h_{n}||_{X}%
\leqq(1+2\varepsilon)\mu(Q),\qquad\left\Vert g_{n}\right\Vert _{p^{\prime}%
,Q}+||h_{n}||_{X}\leqq2\varepsilon. \label{decn}%
\end{equation}

\end{lemma}

\begin{proof}
(i) Step 1. Case where $\mu$ has a compact support in $Q.$ By \cite{DrPoPr},
we can find a decomposition $\mu=(f,g,h)$ with $f,g,h$ have a compact support
in $Q.$ Let $\left\{  \varphi_{n}\right\}  $ be sequence of mollifiers in
$\mathbb{R}^{N+1}$. Then $\mu_{n}=\varphi_{n}\ast\mu\in C_{c}^{\infty}(Q)$ for
$n$ large enough. We see that $\mu_{n}(Q)=\mu(Q)$ and $\mu_{n}$ admits the
decomposition $\mu_{n}=(f_{n},g_{n},h_{n})=(\varphi_{n}\ast f,\varphi_{n}\ast
g,\varphi_{n}\ast h)$. Since $\left\{  f_{n}\right\}  ,\left\{  g_{n}\right\}
,\left\{  h_{n}\right\}  $ strongly converge to $f,g,h$ in $L^{1}%
(Q),(L^{p^{\prime}}(Q))^{N}$ and $X$ respectively, we have for $n_{0}$ large
enough,
\[
||f-f_{n_{0}}||_{1,Q}+||g-g_{n_{0}}||_{p^{\prime},Q}+||h-h_{n_{0}}||_{X}%
\leqq\varepsilon\min\{\mu(Q),1\}.
\]
Then we obtain a decomposition $\mu=(\hat{f},\hat{g},\hat{h})=(\mu_{n_{0}%
}+f-f_{n_{0}},g-g_{n_{0}},h-h_{n_{0}}),$ such that
\begin{equation}
||\hat{f}||_{1,Q}+||\hat{g}||_{p^{\prime},Q}+||\hat{h}||_{X}\leqq
(1+\varepsilon)\mu(Q),\qquad\left\Vert \hat{g}\right\Vert _{p^{\prime}%
,Q}+||\hat{h}||_{X}\leqq\varepsilon. \label{comp}%
\end{equation}

Step 2. General case. Let $\{\theta_{n}\}$ be a nonnegative, nondecreasing
sequence in $C_{c}^{\infty}(Q)$ which converges to $1,$ $a.e.$ in $Q$. Set
${\tilde{\mu}_{0}}={\theta_{0}\mu,}$ and ${\tilde{\mu}_{n}}=(\theta_{n}%
-\theta_{n-1})\mu,$ for any $n\geqq1$. Since $\tilde{\mu}_{n}\in
\mathcal{M}_{0}(Q)\cap\mathcal{M}_{b}^{+}(Q)$ has compact support, by Step 1,
we can find a decomposition $\tilde{\mu}_{n}=($ $\tilde{f}_{n},\tilde{g}%
_{n},\tilde{h}_{n})$ such that
\[
||\tilde{f}_{n}||_{1,Q}+\left\Vert \tilde{g}_{n}\right\Vert _{p^{\prime}%
,Q}+||\tilde{h}_{n}||_{X}\leqq(1+\varepsilon)\mu_{n}(Q),\qquad||\tilde{g}%
_{n}||_{p^{\prime},Q}+||\tilde{h}_{n}||_{X}\leqq2^{-n-1}\varepsilon.
\]
\newline Let $\overline{f}_{n}=\sum\limits_{k=0}^{n}{\tilde{f}}_{k}$,
$\overline{g}_{n}=\sum\limits_{k=0}^{n}\tilde{g}_{k}$ and $\bar{h}_{n}%
=\sum\limits_{k=0}^{n}\tilde{h}_{k}$. Clearly, $\theta_{n}\mu=(\overline
{f}_{n},\overline{g}_{n},\bar{h}_{n}),$ and $\left\{  \overline{f}%
_{n}\right\}  $,$\left\{  \overline{g}_{n}\right\}  ,\left\{  \bar{h}%
_{n}\right\}  $ converge strongly to some $f,g,h,$ respectively in $L^{1}(Q)$,
$(L^{p^{\prime}}(Q))^{N}$, $X$, with
\[
||\overline{f}_{n}||_{1,Q}+||\overline{g}_{n}||_{p^{\prime},Q}+||\bar{h}%
_{n}||_{X}\leqq(1+\varepsilon)\mu(Q),\qquad||\overline{g}_{n}||_{p^{\prime}%
,Q}+||\bar{h}_{n}||_{X}\leqq\varepsilon.
\]
Therefore, $\mu=(f,g,h)$ and (\ref{deco}) holds.\medskip

(ii) We take a sequence $\{m_{n}\}$ in $\mathbb{N}$ such that $f_{n}%
=\varphi_{m_{n}}\ast\overline{f}_{n},g_{n}=\varphi_{m_{n}}\ast\overline{g}%
_{n},h_{n}=\varphi_{m_{n}}\ast\bar{h}_{n}\in C_{c}^{\infty}(Q)$ and
\[
||f_{n}-\overline{f}_{n}||_{1,Q}+||g_{n}-\overline{g}_{n}||_{p^{\prime}%
,Q}+||h_{n}-\bar{h}_{n}||_{X}\leqq\frac{\varepsilon}{n+1}\min\{\mu(Q),1\}.
\]
Let $\mu_{n}=\varphi_{m_{n}}\ast(\theta_{n}\mu)=(f_{n},g_{n},h_{n}).$
Therefore, $\left\{  f_{n}\right\}  ,\left\{  g_{n}\right\}  ,\left\{
h_{n}\right\}  $ strongly converge to $f,g,h$ in $L^{1}(Q),(L^{p^{\prime}%
}(Q))^{N}$ and $X$ respectively. And (\ref{decn}) holds. Furthermore,
$\left\{  \mu_{n}\right\}  $ converges weak-* to $\mu,$ and $\mu_{n}%
(Q)=\int_{Q}\theta_{n}d\mu$ converges to $\mu(Q)$, thus $\left\{  \mu
_{n}\right\}  $ converges in the narrow topology.\medskip
\end{proof}

As a consequence, we get an approximation property for any measure $\mu
\in\mathcal{M}_{b}^{+}(Q):$

\begin{proposition}
\label{P5} Let $\mu\in\mathcal{M}_{b}^{+}(Q)$ and $\varepsilon>0$. Let
$\left\{  \mu_{n}\right\}  $ be a nondecreasing sequence in $\mathcal{M}%
_{b}^{+}(Q)$ converging to $\mu$ in $\mathcal{M}_{b}(Q)$. Then, there exist
$f_{n},f\in L^{1}(Q)$, $g_{n},g\in(L^{p^{\prime}}(Q))^{N}$ and $h_{n},h\in X,$
$\mu_{n,s},\mu_{s}\in\mathcal{M}_{s}^{+}(Q)$ such that
\[
\mu=f-\operatorname{div}g+h_{t}+\mu_{s},\qquad\mu_{n}=f_{n}-\operatorname{div}%
g_{n}+(h_{n})_{t}+\mu_{n,s},
\]
and $\left\{  f_{n}\right\}  ,\left\{  g_{n}\right\}  ,\left\{  h_{n}\right\}
$ strongly converge to $f,g,h$ in $L^{1}(Q),(L^{p^{\prime}}(Q))^{N}$ and $X$
respectively, and $\left\{  \mu_{n,s}\right\}  $ converges to $\mu_{s}$
(strongly) in $\mathcal{M}_{b}(Q)$ and
\begin{equation}
||f_{n}||_{1,Q}+||g_{n}||_{p^{\prime},Q}+||h_{n}||_{X}+\mu_{n,s}(\Omega
)\leqq(1+\varepsilon)\mu(Q),\qquad\text{and }||g_{n}||_{p^{\prime},Q}%
+||h_{n}||_{X}\leqq\varepsilon. \label{2504}%
\end{equation}

\end{proposition}

\begin{proof}
Since $\left\{  \mu_{n}\right\}  $ is nondecreasing, then $\left\{  \mu
_{n,0}\right\}  $, $\left\{  \mu_{n,s}\right\}  $ are too. Clearly,
$\left\Vert {\mu-{\mu_{n}}}\right\Vert _{\mathcal{M}_{b}(Q)}=\left\Vert
{{\mu_{0}}-{\mu_{n,0}}}\right\Vert _{\mathcal{M}_{b}(Q)}+\left\Vert {{\mu_{s}%
}-{\mu_{n,s}}}\right\Vert _{\mathcal{M}_{b}(Q)}$. Hence, $\left\{  \mu
_{n,s}\right\}  $ converge to $\mu_{s}$ and $\left\{  \mu_{n,0}\right\}  $
converge to ${{\mu_{0}}}$ (strongly) in $\mathcal{M}_{b}(Q)$. Set
${\widetilde{\mu}_{0,0}}={\mu_{0,0},}$ and ${\widetilde{\mu}_{n,0}}={\mu
_{n,0}}-{\mu_{n-1,0}}$ for any $n\geqq1$. By Lemma \ref{att}, (i), we can find
$\tilde{f}_{n}\in L^{1}(Q)$, $\tilde{g}_{n}\in(L^{p^{\prime}}(Q))^{N}$ and
$\tilde{h}_{n}\in X$ such that $\tilde{\mu}_{n,0}=(\tilde{f}_{n},\tilde{g}%
_{n},\tilde{h}_{n})$ and
\[
||\tilde{f}_{n}||_{1,Q}+||\tilde{g}_{n}||_{p^{\prime},Q}+||\tilde{h}_{n}%
||_{X}\leqq(1+\varepsilon)\tilde{\mu}_{n,0}(Q),\qquad||\tilde{g}%
_{n}||_{p^{\prime},Q}+||\tilde{h}_{n}||_{X}\leqq2^{-n-1}\varepsilon.
\]

\noindent Let $f_{n}=\sum\limits_{k=0}^{n}{\tilde{f}}_{k}$, $G_{n}%
=\sum\limits_{k=0}^{n}\tilde{g}_{k}$ and $h_{n}=\sum\limits_{k=0}^{n}\tilde
{h}_{k}$. Clearly, $\mu_{n,0}=(f_{n},g_{n},h_{n})$ and the convergence
properties hold with (\ref{2504}), since
\[
||f_{n}||_{1,Q}+||g_{n}||_{p^{\prime},Q}+||h_{n}||_{X}\leqq(1+\varepsilon
)\mu_{0}(Q).
\]

\end{proof}

In Section \ref{appli} we consider some measures $\mu\in\mathcal{M}_{b}(Q)$
which satisfy $\left\vert \mu\right\vert $ $\leqq\omega\otimes F,$ with
$\omega\in\mathcal{M}_{b}(\Omega)$ and $F\in L^{1}((0,T)),F\geqq0.$ It is
interesting to compare the properties of $\omega\otimes F$ and $\omega
:\medskip$

Let $c_{p}^{\Omega}$ be the elliptic capacity in $\Omega$ defined by
\[
c_{p}^{\Omega}(K)=\inf\{\int_{\Omega}|\nabla\varphi|^{p}:\varphi\geqq\chi
_{K},\varphi\in C_{c}^{\infty}(\Omega)\},
\]
for any compact set $K\subset\Omega.$ $\medskip$

Let $\mathcal{M}_{0,e}(\Omega)$ be the set of Radon measures $\omega$ on that
do not charge the sets of zero $c_{p}^{\Omega}$-capacity. Then $\mathcal{M}%
_{b}(\Omega)\cap\mathcal{M}_{0,e}(\Omega)$ is characterised as the set of
measures $\omega\in\mathcal{M}_{b}(\Omega)$ which can be written under the
form $\tilde{f}-\operatorname{div}\tilde{g}$ with $\tilde{f}\in L^{1}(\Omega)$
and $\tilde{g}\in(L^{p^{\prime}}(\Omega))^{N},$ see \cite{BoGaOr96}.

\begin{proposition}
\label{mzero}For any $F\in L^{1}((0,T))$ with $\int_{0}^{T}F(t)dt\not = 0,$
and $\omega\in\mathcal{M}_{b}(\Omega)$,
\[
\omega\in\mathcal{M}_{0,e}(\Omega)\Longleftrightarrow\omega\otimes
F\in\mathcal{M}_{0}(Q).
\]

\end{proposition}

\begin{proof}
Assume that $\omega\otimes F\in\mathcal{M}_{0}(Q)$. Then, there exist $f\in
L^{1}(Q)$, $g\in\left(  L^{p^{\prime}}(Q)\right)  ^{N}$ and $h\in X$, such
that
\begin{equation}
\int_{Q}\varphi(x,t)F(t)d\omega(x)dt=\int_{Q}\varphi(x,t)f(x,t)dxdt+\int%
_{Q}g(x,t).\nabla\varphi(x,t)dxdt-\int_{Q}h(x,t)\varphi_{t}(t,x)dxdt,
\label{11061}%
\end{equation}
for all $\varphi\in C_{c}^{\infty}(\Omega\times\lbrack0,T])$, see \cite[Lemma
2.24 and Theorem 2.27]{PePoPor}.\ By choosing $\varphi(x,t)=\varphi(x)\in
C_{c}^{\infty}(\Omega)$ and using Fubini's Theorem, (\ref{11061}) is rewritten
as
\[
\int_{\Omega}\varphi(x)d\omega(x)=\int_{\Omega}\varphi(x)\tilde{f}%
(x)dx+\int_{\Omega}\tilde{g}(x).\nabla\varphi(x)dx,
\]
where $\tilde{f}(x)=\left(  \int_{0}^{T}F(t)dt\right)  ^{-1}\int_{0}%
^{T}f(x,t)dt\in L^{1}(\Omega)$ and $\tilde{g}(x)=\left(  \int_{0}%
^{T}F(t)dt\right)  ^{-1}\int_{0}^{T}g(x,t)dt\in\left(  L^{p^{\prime}}%
(\Omega)\right)  ^{N}$; hence $\omega\in\mathcal{M}_{0,e}(\Omega)$. \newline
Conversely, assume that $\omega=\tilde{f}-\operatorname{div}\tilde{g}%
\in\mathcal{M}_{0,e}(\Omega),$ with $\tilde{f}\in L^{1}(\Omega)$ and
$\tilde{g}\in\left(  L^{p^{\prime}}(\Omega)\right)  ^{N}$. So $\omega\otimes
T_{n}(F)=f_{n}-\operatorname{div}g_{n},$ with $f_{n}=\tilde{f}T_{n}(F)\in
L^{1}(Q)$ and $g_{n}=\tilde{g}T_{n}(F)\in\left(  L^{p^{\prime}}(Q)\right)
^{N}$. Then $\omega\otimes T_{n}(F)$ admits the decomposition $(f_{n}%
,g_{n},h),$ with $h=0\in L^{\infty}(Q),$ thus $\omega\otimes T_{n}%
(F)\in\mathcal{M}_{0}(Q)$. And $\left\{  \omega\otimes T_{n}(F)\right\}  $
converges to $\omega\otimes F$ strongly in $\mathcal{M}_{b}(Q)$, since
$||\omega\otimes(F-T_{n}(F))||_{\mathcal{M}_{b}(Q)}\leqq||\omega
||_{\mathcal{M}_{b}(\Omega)}\left\Vert F-T_{n}(F)\right\Vert _{L^{1}((0,T))}.$
Then $\omega\otimes F\in\mathcal{M}_{0}(Q),$ since $\mathcal{M}_{0}%
(Q)\cap\mathcal{M}_{b}(Q)$ is strongly closed in $\mathcal{M}_{b}(Q).$
\end{proof}

\section{Renormalized solutions of problem (\ref{pmu})\label{defsol}}

\subsection{Notations and Definition}

For any function $f\in L^{1}(Q),$ we write $\int_{Q}f$ instead of $\int%
_{Q}fdxdt,$ and for any measurable set $E\subset${$Q$}${,}$ $\int_{E}f$
instead of $\int_{E}fdxdt.$\medskip

\noindent We set $T_{k}(r)=\max\{\min\{r,k\},-k\},$ for any $k>0$ and
$r\in\mathbb{R}$. We recall that if $u$ is a measurable function defined and
finite $a.e.$ in $Q$, such that $T_{k}(u)\in X$ for any $k>0$, there exists a
measurable function $w$ from $Q$ into $\mathbb{R}^{N}$ such that $\nabla
T_{k}(u)=\chi_{|u|\leqq k}w,$ $a.e.$ in $Q,$ and for any $k>0$. We define the
gradient $\nabla u$ of $u$ by $w=\nabla u$. \medskip

\noindent Let $\mu=\mu_{0}+\mu_{s}\in\mathcal{M}_{b}(${$Q$}$),$ and $(f,g,h)$
be a decomposition of $\mu_{0}$ given by (\ref{dec}), and ${{\widehat{\mu_{0}%
}}}=\mu_{0}-h_{t}=f-\operatorname{div}g$. In the general case ${{\widehat{\mu
_{0}}}}\notin\mathcal{M}(Q),$ but we write, for convenience,%
\[
\int_{Q}{wd{\widehat{\mu_{0}}}}:=\int_{Q}(fw+g.\nabla w),\qquad\forall w\in
X{\cap}L^{\infty}(Q).
\]

\begin{definition}
\label{defin}Let ${u}_{0}\in L^{1}(\Omega),$ $\mu=\mu_{0}+\mu_{s}%
\in\mathcal{M}_{b}(${$Q$}$)$. A measurable function $u$ is a
\textbf{renormalized solution, }called\textbf{\ R-solution} of (\ref{pmu}) if
there exists a decompostion $(f,g,h)$ of $\mu_{0}$ such that
\begin{equation}
v=u-h\in L^{\sigma}(0,T;W_{0}^{1,\sigma}(\Omega)\cap L^{\infty}(0,T;L^{1}%
(\Omega)),\quad\forall\sigma\in\left[  1,m_{c}\right)  ;\qquad T_{k}(v)\in
X,\quad\forall k>0, \label{defv}%
\end{equation}
and:\medskip

(i) for any $S\in W^{2,\infty}(\mathbb{R})$ such that $S^{\prime}$ has compact
support on $\mathbb{R}$, and $S(0)=0$,%
\begin{equation}
-\int_{\Omega}S(u_{0})\varphi(0)dx-\int_{Q}{{\varphi_{t}}S(v)}+\int%
_{Q}{S^{\prime}(v)A(x,t,\nabla u).\nabla\varphi}+\int_{Q}{S^{\prime\prime
}(v)\varphi A(x,t,\nabla u).\nabla v}=\int_{Q}{S^{\prime}(v)\varphi
d{\widehat{\mu_{0}},}} \label{renor}%
\end{equation}
for any $\varphi\in X\cap L^{\infty}(Q)$ such that $\varphi_{t}\in X^{\prime
}+L^{1}(Q)$ and $\varphi(T,.)=0$;\medskip

(ii) for any $\phi\in C(\overline{{Q}}),$%
\begin{equation}
\lim_{m\rightarrow\infty}\frac{1}{m}\int\limits_{\left\{  m\leqq v<2m\right\}
}{\phi A(x,t,\nabla u).\nabla v}=\int_{Q}\phi d\mu_{s}^{+} \label{renor2}%
\end{equation}%
\begin{equation}
\lim_{m\rightarrow\infty}\frac{1}{m}\int\limits_{\left\{  -m\geqq
v>-2m\right\}  }{\phi A(x,t,\nabla u).\nabla v}=\int_{Q}\phi d\mu_{s}^{-}.
\label{renor3}%
\end{equation}

\end{definition}

\begin{remark}
As a consequence, $S(v)\in C([0,T];L^{1}(\Omega))$ and ${S(v)(0,.)=S(u}_{0})$
in $\Omega;$ and $u$ satisfies the equation
\begin{equation}
({S(v))}_{t}-\operatorname{div}({S^{\prime}(v)A(x,t,\nabla u))+S^{\prime
\prime}(v)A(x,t,\nabla u).\nabla v{=f}S^{\prime}(v)-\operatorname{div}%
(gS^{\prime}(v))+S^{\prime\prime}(v)g.\nabla v,}\text{ }\label{dpri}%
\end{equation}
in the sense of distributions in $Q,$ see \cite[Remark 3]{Pe08}. Moreover
\begin{align*}
{\left\Vert {S{{(v)}_{t}}}\right\Vert _{{X}^{\prime}+{L^{1}}({Q})}}%
\leqq{\left\Vert \operatorname{div}({{{S^{^{\prime}}}(v)}A(x,t,\nabla
u)})\right\Vert _{{X}^{\prime}}} &  +{\left\Vert {{S^{^{\prime\prime}}%
}(v)A(x,t,\nabla u).\nabla v}\right\Vert _{1,Q}}\\
&  +{\left\Vert {{S^{^{\prime}}}(v)f}\right\Vert _{1,Q}}+{\left\Vert
{g.{S^{^{\prime\prime}}}(v)\nabla v}\right\Vert _{1,Q}+\left\Vert
\operatorname{div}({{{S^{^{\prime}}}(v)g)}}\right\Vert _{{X}^{\prime}}.}%
\end{align*}
Thus, if $\left[  -M,M\right]  \supset$ supp$S^{\prime},$%
\begin{align*}
& {\left\Vert {{S^{^{\prime\prime}}}(v)A(x,t,\nabla u).\nabla v}\right\Vert
_{1,Q}}\leq\left\Vert S\right\Vert _{{W^{2,\infty}}(\mathbb{R})}(\left\Vert
{A(x,t,\nabla u)\chi_{|v|\leqq M}}\right\Vert _{p^{\prime},Q}^{p^{\prime}%
}+\left\Vert |\nabla T_{M}(v)|\right\Vert _{p,Q}^{p})\\
& \leqq C\left\Vert S\right\Vert _{{W^{2,\infty}}(\mathbb{R})}(\left\Vert
{|\nabla u|^{p}\chi_{|v|\leqq M}}\right\Vert _{1,Q}+||a||_{p^{\prime}%
,Q}^{p^{\prime}}+\left\Vert |\nabla T_{M}(v)|\right\Vert _{p,Q}^{p})
\end{align*}
thus%
\begin{align}
{\left\Vert {S{{(v)}_{t}}}\right\Vert _{{X}^{\prime}+{L^{1}}({Q})}} &  \leqq
C\left\Vert S\right\Vert _{{W^{2,\infty}}(\mathbb{R})}\left(  {}\right.
\left\Vert {\left\vert {\nabla u}\right\vert }^{p}\chi_{|v|\leqq M}\right\Vert
_{1,Q}^{1/p^{\prime}}+\left\Vert {|\nabla u|^{p}\chi_{|v|\leqq M}}\right\Vert
_{1,Q}+\left\Vert |\nabla T_{M}(v)|\right\Vert _{p,Q}^{p}\nonumber\\
&  +\left\Vert a\right\Vert _{p^{\prime},Q}+\left\Vert a\right\Vert
_{p^{\prime},Q}^{p^{\prime}}+\left\Vert f\right\Vert _{1,Q}+\left\Vert
g\right\Vert _{p^{\prime},Q}\left\Vert \left\vert {\nabla u}\right\vert
^{p}\chi_{|v|\leqq M}\right\Vert _{1,Q}^{1/p}+\left\Vert g\right\Vert
_{p^{\prime},Q}\left.  {}\right)  \label{11051}%
\end{align}
We also deduce that, for any $\varphi\in X\cap L^{\infty}(Q),$ such that
$\varphi_{t}{\in X}^{\prime}+L^{1}(Q),$
\begin{align}
\int_{\Omega}S(v(T))\varphi(T)dx-\int_{\Omega}S(u_{0})\varphi(0)dx &
-\int_{Q}{{\varphi_{t}}S(v)}+\int_{Q}{S^{\prime}(v)A(x,t,\nabla u).\nabla
\varphi}\nonumber\\
&  +\int_{Q}{S^{\prime\prime}(v)A(x,t,\nabla u).\nabla v\varphi}=\int%
_{Q}{S^{\prime}(v)\varphi d{\widehat{\mu_{0}}.}}\label{renor4}%
\end{align}

\end{remark}

\begin{remark}
Let $u,v$ satisfying (\ref{defin}). It is easy to see that the condition
(\ref{renor2}) ( resp. (\ref{renor3}) ) is equivalent to
\begin{equation}
\lim_{m\rightarrow\infty}\frac{1}{m}\int\limits_{\left\{  m\leqq v<2m\right\}
}{\phi A(x,t,\nabla u).\nabla u}=\int_{Q}\phi d\mu_{s}^{+} \label{lim1}%
\end{equation}
resp.
\begin{equation}
\lim_{m\rightarrow\infty}\frac{1}{m}\int\limits_{\left\{  m\geqq
v>-2m\right\}  }{\phi A(x,t,\nabla u).\nabla u}=\int_{Q}\phi d\mu_{s}^{-}.
\label{lim2}%
\end{equation}
In particular, for any $\varphi\in L^{p^{\prime}}(Q)$ there holds
\begin{equation}
\lim_{m\rightarrow\infty}\frac{1}{m}\int\limits_{m\leqq|v|<2m}{|\nabla
u|\varphi}=0,\qquad\lim_{m\rightarrow\infty}\frac{1}{m}\int\limits_{m\leqq
|v|<2m}{|\nabla v|\varphi}=0. \label{lim3}%
\end{equation}

\end{remark}

\begin{remark}
(i) Any function $U\in X$ such that $U_{t}\in X^{\prime}+L^{1}(Q)$ admits a
unique $c_{p}^{Q}$-quasi continuous representative, defined $c_{p}^{Q}$-quasi
$a.e.$ in $Q,$ still denoted $U.$ Furthermore, if $U\in L^{\infty}(Q),$ then
for any $\mu_{0}\in\mathcal{M}_{0}(Q),$ there holds $U\in L^{\infty}%
(Q,d\mu_{0}),$ see \cite[Theorem 3 and Corollary 1]{Pe08}.\medskip

(ii) Let $u$ be any R- solution of problem (\ref{pmu}). Then, $v=u-h$ admits a
$c_{p}^{Q}$-quasi continuous functions representative which is finite
$c_{p}^{Q}$-quasi $a.e.$ in $Q,$ and $u$ satisfies definition \ref{defin} for
every decomposition $(\tilde{f},\tilde{g},\tilde{h})$ such that $h-\tilde
{h}\in L^{\infty}(Q)$, see \cite[Proposition 3 and Theorem 4 ]{Pe08}.
\end{remark}

\subsection{Steklov and Landes approximations}

\textit{A main difficulty for proving Theorem \ref{sta} is the choice of
admissible test functions }$(S,\varphi)$\textit{ in (\ref{renor}), valid for
any R-solution}. Because of a lack of regularity of these solutions, we use
two ways of approximation adapted to parabolic equations:

\begin{definition}
\label{ste}Let $\varepsilon\in(0,T)$ and $z\in L_{loc}^{1}(Q)$. For any
$l\in(0,\varepsilon)$ we define the \textbf{Steklov time-averages}
$[z]_{l},[z]_{-l}$ of $z$ by
\[
{[z]_{l}}(x,t)=\frac{1}{l}\int\limits_{t}^{t+l}{z(x,s)ds}\qquad\text{for
}a.e.\;(x,t)\in\Omega\times(0,T-\varepsilon),
\]%
\[
{[z]_{-l}}(x,t)=\frac{1}{l}\int\limits_{t-l}^{t}{z(x,s)ds}\qquad\text{for
}a.e.\;(x,t)\in\Omega\times(\varepsilon,T).
\]

\end{definition}

\noindent The idea to use this approximation for R-solutions can be found in
\cite{BlPo}. Recall some properties, given in  \cite{PePoPor}.  Let
$\varepsilon\in(0,T),$ and $\varphi_{1}\in C_{c}^{\infty}(\overline{\Omega
}\times\lbrack0,T)),\;\varphi_{2}\in C_{c}^{\infty}(\overline{\Omega}%
\times(0,T])$ with $\mathrm{Supp}\varphi{{_{1}}}\subset\overline{\Omega}%
\times\lbrack0,T-\varepsilon],\;\mathrm{Supp}\varphi{{_{2}}}\subset
\overline{\Omega}\times\lbrack\varepsilon,T]$. There holds

\begin{description}
\item[(i)] If $z\in{X}$, then $\varphi_{1}{[z]_{l}}$ and $\varphi_{2}%
{[z]_{-l}}\in{W.}$

\item[(ii)] If $z\in X$ and $z_{t}\in X^{\prime}+L^{1}(Q),$ then, as
$l\rightarrow0,$ $(\varphi_{1}{[z]_{l})}$ and $(\varphi_{2}{[z]_{-l})}$
converge respectively to $\varphi_{1}z$ and $\varphi_{2}z$ in $X,$ and $a.e.$
in $Q;$ and $(\varphi_{1}{[z]_{l})}_{t},(\varphi_{2}{[z]_{-l})}_{t}$ converge
to $(\varphi_{1}z)_{t},(\varphi_{2}z)_{t}$ in $X^{\prime}+L^{1}(Q)$.

\item[(iii)] If moreover $z\in L^{\infty}(Q)$, then from any sequence
$\{l_{n}\}\rightarrow0,$ there exists a subsequence $\{l_{\nu}\}$ such that
$\left\{  [z]_{l_{\nu}}\right\}  ,\left\{  [z]_{-l_{\nu}}\right\}  $ converge
to $z,$ $c_{p}^{Q}$-quasi everywhere in $Q.$
\end{description}

Next we recall the approximation introduced in \cite{La}, used in \cite{DAOr},
\cite{BDGO97}, \cite{BlPo1}:

\begin{definition}
Let $\mu\in\mathcal{M}_{b}(Q)$ and $u_{0}\in L^{1}(\Omega).$ Let $u$ be a
R-solution of (\ref{pmu}), and $v=u-h$ given at (\ref{defv}), and $k>0.$ For
any $\nu\in\mathbb{N},$ the \textbf{Landes-time approximation} $\langle
T_{k}(v)\rangle_{\nu}$ of the truncate function $T_{k}(v)$ is defined in the
following way:\medskip

Let $\left\{  z_{\nu}\right\}  $ be a sequence of functions in $W_{0}%
^{1,p}(\Omega)\cap L^{\infty}(\Omega)$, such that $||z_{\nu}||_{\infty,\Omega
}\leqq k$, $\left\{  z_{\nu}\right\}  $ converges to $T_{k}(u_{0})$ $a.e.$ in
$\Omega$, and $\nu^{-1}||z_{\nu}||_{W_{0}^{1,p}(\Omega)}^{p}$ converges to
$0$. Then, $\langle T_{k}(v)\rangle_{\nu}$ is the unique solution of the
problem
\[
(\langle T_{k}(v)\rangle_{\nu})_{t}=\nu\left(  {{T_{k}}(v)-}\langle
T_{k}(v)\rangle_{\nu}\right)  \quad\text{in the sense of distributions, }%
\quad\langle T_{k}(v)\rangle_{\nu}(0)={z_{\nu}}\text{, }\qquad\text{in }%
\Omega.
\]

\end{definition}

Therefore, $\langle T_{k}(v)\rangle_{\nu}\in X\cap L^{\infty}(Q)$ and
$(T_{k}(v)\rangle_{\nu})_{t}\in X$, see \cite{La}. Furthermore, up to
subsequences, $\left\{  \langle T_{k}(v)\rangle_{\nu}\right\}  $ converges to
$T_{k}(v)$ strongly in $X$ and $a.e.$ in $Q$, and $||\left(  {{T_{k}}%
(v)}\right)  _{\nu}||_{L^{\infty}(Q)}\leqq k$.

\subsection{First properties}

In the sequel we use the following notations: for any function $J\in
{W^{1,\infty}}(\mathbb{R})$, nondecreasing with $J(0)=0,$ we set
\begin{equation}
\overline{J}(r)=\int_{0}^{r}J{(\tau)d\tau,}\text{\qquad}\mathcal{J}%
(r)=\int_{0}^{r}J{^{\prime}(\tau)\tau d\tau.} \label{lam}%
\end{equation}
It is easy to verify that $\mathcal{J}(r)\geqq0,$
\begin{equation}
\mathcal{J}(r)+\overline{J}(r)=J(r)r,\quad\text{and\quad}\mathcal{J}%
(r)-\mathcal{J}(s)\geqq s\left(  J{(}r{)-J(}s{)}\right)  \quad\quad\forall
r,s\in\mathbb{R}. \label{222}%
\end{equation}
In particular we define, for any $k>0,$ and any $r\in\mathbb{R},$
\begin{equation}
\overline{T_{k}}(r)=\int_{0}^{r}T_{k}{(\tau)d\tau,}\text{\qquad}%
\mathcal{T}_{k}(r)=\int_{0}^{r}T_{k}^{\prime}{(\tau)\tau d\tau,} \label{tkp}%
\end{equation}
and we use several times a truncature used in \cite{DMOP}:
\begin{equation}
H{_{m}}(r)={\chi_{\lbrack-m,m]}}(r)+\frac{{2m-|s|}}{m}{\chi_{m<|s|\leqq2m}%
}(r),\qquad\overline{H_{m}}(r)=\int_{0}^{r}H{_{m}(\tau)d\tau.} \label{Hm}%
\end{equation}

The next Lemma allows to extend the range of the test functions in
(\ref{renor}). Its proof, given in the Appendix, is obtained by Steklov
approximation of the solutions.

\begin{lemma}
\label{integ}Let $u$ be a R-solution of problem (\ref{pmu}). Let
$J\in{W^{1,\infty}}(\mathbb{R})$ be nondecreasing with $J(0)=0$, and
$\overline{J}$ defined by (\ref{lam}). Then,%
\begin{align}
&  \int_{Q}{{S^{\prime}(v)A(x,t,\nabla u).\nabla\left(  {\xi J(S(v))}\right)
}}+\int_{Q}{S^{\prime\prime}(v)A(x,t,\nabla u ).\nabla v\xi J(S(v))}%
\nonumber\\
&  {-\int_{\Omega}{{\xi(0)J(S({u_{0}}))S({u_{0}})}}-}\int_{Q}{{{\xi_{t}%
}\overline{J}(S(v))}}\nonumber\\
&  \leqq\int_{Q}{{S^{\prime}(v)\xi J(S(v))d\widehat{\mu_{0}}{,}}}
\label{parti}%
\end{align}
for any $S\in{W^{2,\infty}}(\mathbb{R})$ such that $S^{\prime}$ has compact
support on $\mathbb{R}$ and $S(0)=0,$ and for any $\xi\in C^{1}(Q)\cap
W^{1,\infty}(Q),\xi\geqq0.$
\end{lemma}

Next we give estimates of the gradient, following the first estimates of
\cite{BDGO97}, see also \cite{DrPoPr}, \cite[Proposition 2]{Pe08}, \cite{LePe}.

\begin{proposition}
\label{estsup}If $u$ is a R-solution of problem (\ref{pmu}), then there exists
$c=c(p)$ such that, for any $k\geqq1$ and $\ell$ $\geqq0,$
\begin{equation}
\int\limits_{\ell\leqq|v|\leqq\ell+k}{{{\left\vert {\nabla u}\right\vert }%
^{p}+}}\int\limits_{{\ell}\leqq{|v|}\leqq{\ell+k}}{{{\left\vert {\nabla
v}\right\vert }^{p}}}\leqq ckM \label{alp}%
\end{equation}
and
\begin{equation}
{\left\Vert v\right\Vert _{{L^{\infty}}((0,T);{L^{1}}(\Omega))}}\leqq
c({M}+|\Omega|), \label{gam}%
\end{equation}
where
\[
M={{{\left\Vert {{u_{0}}}\right\Vert }_{1,\Omega}}}+\left\vert {\mu_{s}%
}\right\vert (Q){+}\left\Vert f\right\Vert _{1,Q}+\left\Vert g\right\Vert
_{p^{\prime},Q}^{p^{\prime}}+{\left\Vert h\right\Vert _{{X}}^{p}%
}+||a||_{p^{\prime},Q}^{p^{\prime}}.
\]

\noindent As a consequence, for any $k$ $\geqq1,$
\begin{equation}
\mathrm{meas}\left\{  {|v|>k}\right\}  \leqq{C_{1}M_{1}}{k^{-p_{c}},\qquad
}\mathrm{meas}\left\{  {|\nabla v|>k}\right\}  \leqq{C_{2}M_{2}}{k^{-m_{c}},}
\label{mess}%
\end{equation}%
\begin{equation}
\mathrm{meas}\left\{  {|u|>k}\right\}  \leqq{C_{3}M_{2}}{k^{-p_{c}},\qquad
}\mathrm{meas}\left\{  {|\nabla u|>k}\right\}  \leqq{C_{4}M_{2}}{k^{-m_{c}},}
\label{mess2}%
\end{equation}
where $C_{i}=C_{i}(N,p,c_{1},c_{2}),$ $i=1$-$4,$ and ${M_{1}}={\left(
M{+|\Omega|}\right)  ^{\frac{p}{N}}}M$ and ${M_{2}}=M_{1}+M.$
\end{proposition}

\begin{proof}
Set for any $r\in\mathbb{R}$, and $m,k,\ell>0,$
\[
T_{k,\ell}(r)=\max\{\min\{r-\ell,k\},0\}+\min\{\max\{r+\ell,-k\},0\}.
\]
For $m>k+\ell$, we can choose $(J,S,\xi)=(T_{k,\ell},\overline{H_{m}},\xi)$ as
test functions in (\ref{parti}), where $\overline{H_{m}}$ is defined at
(\ref{Hm}) and $\xi\in C^{1}([0,T])$ with values in $[0,1]$, independent on
$x$. Since $T_{k,\ell}(\overline{H_{m}}(r))=T_{k,\ell}(r)$ for all
$r\in\mathbb{R},$ we obtain
\[%
\begin{array}
[c]{l}%
-\int_{\Omega}{\xi(0){T_{k,\ell}}({u_{0}})\overline{H_{m}}({u_{0}})}-\int%
_{Q}{{\xi_{t}}}\overline{T_{k,\ell}}{(\overline{H_{m}}(v))}\\
\\
+\int\limits_{\left\{  {\ell}\leqq{|v|}<{\ell+k}\right\}  }{\xi A(x,t,\nabla
u).\nabla v}-\frac{k}{m}\int\limits_{\left\{  m\leqq|v|<2m\right\}  }{\xi
A(x,t,\nabla u).\nabla v}\leqq\int_{Q}H_{m}{(v)\xi{T_{k,\ell}}%
(v)d{{\widehat{\mu_{0}}}.}}%
\end{array}
\]
And
\[
\int_{Q}H_{m}{(v)\xi{T_{k,\ell}}(v)d{{\widehat{\mu_{0}}}}\;}{=}\int_{Q}%
H_{m}{(v)\xi{T_{k,\ell}}(v)}f{{{+}}}\int\limits_{\left\{  {\ell}\leqq
{|v|}<{\ell+k}\right\}  }{\xi\nabla v.g-}\frac{k}{m}\int\limits_{\left\{
m\leqq|v|<2m\right\}  }{\xi\nabla v.g.}%
\]
Let $m\rightarrow\infty$; then, for any $k\geqq1,$ since $v\in L^{1}(Q)$ and
from (\ref{renor2}), (\ref{renor3}), and (\ref{lim3}), we find%
\begin{equation}
-\int_{Q}{{\xi_{t}}\overline{T_{k,\ell}}(v)}+\int\limits_{\left\{  {\ell}%
\leqq{|v|}<{\ell+k}\right\}  }{\xi A(x,t,\nabla u).\nabla v\;}{\leqq}%
\int\limits_{\left\{  {\ell}\leqq{|v|}<{\ell+k}\right\}  }{\xi\nabla
v.g}+k({{{\left\Vert {{u_{0}}}\right\Vert }_{1,\Omega}+}}\left\vert {\mu_{s}%
}\right\vert (Q){+}\left\Vert f\right\Vert _{1,Q}). \label{epsa}%
\end{equation}
Next, we take $\xi\equiv1$. We verify that there exists $c=c(p)$ such that
\[
{ A(x,t,\nabla u).\nabla v-\nabla v.g\geqq}\frac{c_{1}}{4}(|\nabla
u|^{p}+|\nabla v|^{p})-c(\left\vert g\right\vert ^{p^{\prime}}+|\nabla
h|^{p}+|a|^{p^{\prime}})
\]
where $c_{1}$ is the constant in \eqref{condi1}. Hence (\ref{alp}) follows.
Thus, from (\ref{epsa}) and the H\"{o}lder inequality, we get, with another
constant $c$, for any $\xi\in C^{1}([0,T])$ with values in $[0,1],$
\[
-\int_{Q}{{\xi_{t}}\overline{T_{k,\ell}}(v)}\leqq ck M
\]
Thus $\int_{\Omega}\overline{T_{k,\ell}}{(v)(t)}\leqq ckM,$ for $a.e.$
$t\in(0,T).$ We deduce (\ref{gam}) by taking $k=1,\ell=0$, since
$\overline{T_{1,0}}(r)=\overline{T_{1}}(r)$ $\geqq|r|-1,$ for any
$r\in\mathbb{R}.$ \medskip

\noindent Next, from the Gagliardo-Nirenberg embedding Theorem, we have
\[
\int_{Q}{{{\left\vert {{T_{k}}(v)}\right\vert }^{\frac{{p(N+1)}}{N}}}}\leqq
C_{1}\left\Vert v\right\Vert _{{L^{\infty}}((0,T);{L^{1}}(\Omega))}^{\frac
{p}{N}}\int_{Q}{{{\left\vert {\nabla{T_{k}}(v)}\right\vert }^{p}},}%
\]
where $C_{1}=C_{1}(N,p).$ Then, from (\ref{alp}) and (\ref{gam}), we get, for
any $k$ $\geqq1,$
\[
\mathrm{meas}\left\{  {|v|>k}\right\}  \leqq{k^{-\frac{{p(N+1)}}{N}}}\int%
_{Q}{{{\left\vert {{T_{k}}(v)}\right\vert }^{\frac{{p(N+1)}}{N}}}}\leqq
C\left\Vert v\right\Vert _{{L^{\infty}}((0,T);{L^{1}}(\Omega))}^{\frac{p}{N}%
}{k^{-\frac{{p(N+1)}}{N}}}\int_{Q}{{{\left\vert {\nabla{T_{k}}(v)}\right\vert
}^{p}}}\leqq C_{2}M_{1}{k^{-p_{c}}},\text{ }%
\]
with $C_{2}=C_{2}(N,p,c_{1},c_{2})$. We obtain%
\begin{align*}
\mathrm{meas}\left\{  {|\nabla v|>k}\right\}   &  \leqq\frac{1}{{{k^{p}}}}%
\int_{0}^{{k^{p}}}{\mathrm{{meas}}}\left(  {\left\{  {|\nabla v{|^{p}}%
>s}\right\}  }\right)  ds\\
&  \leqq\mathrm{{meas}}\left\{  {|v|>{k^{\frac{N}{{N+1}}}}}\right\}  +\frac
{1}{{{k^{p}}}}\int_{0}^{{k^{p}}}{\mathrm{{meas}}}\left(  {\left\{  {|\nabla
v{|^{p}}>s,|v|}\leqq{{k^{\frac{N}{{N+1}}}}}\right\}  }\right)  ds\\
&  \leqq{C_{2}}{M_{1}}{k^{-m_{c}}}+\frac{1}{{{k^{p}}}}\int\limits_{|v|\leqq
{k^{\frac{N}{{N+1}}}}}{{{\left\vert {\nabla v}\right\vert }^{p}}}\leqq{C_{2}%
}{M_{2}}{k^{-m_{c}}},
\end{align*}
with ${C_{3}}=C_{3}(N,p,c_{1},c_{2}).$ Furthermore, for any $k\geqq1,$
\[
\mathrm{meas}\left\{  {|h|>k}\right\}  +\mathrm{meas}\left\{  {|\nabla
h|>k}\right\}  \leqq C_{4}k^{-p}\left\Vert h\right\Vert _{X}^{p},
\]
where $C_{4}=C_{4}(N,p,c_{1},c_{2})$. Therefore, we easily get (\ref{mess2}).
\end{proof}

\begin{remark}
\label{h2} If $\mu\in L^{1}(Q)$ and $a\equiv0$ in (\ref{condi1}), then
(\ref{alp}) holds for all $k>0$ and the term $|\Omega|$ in inequality
(\ref{gam}) can be removed where $M=||u_{0}||_{1,\Omega}+|\mu|(Q)$.
Furthermore, (\ref{mess2}) is stated as follows:
\begin{equation}
\mathrm{meas}\left\{  {|u|>k}\right\}  \leqq C_{3}M^{\frac{p+N}{N}}{k^{-p_{c}%
},\qquad}\mathrm{meas}\left\{  {|\nabla u|>k}\right\}  \leqq C_{4}%
M^{\frac{N+2}{N+1}}k^{-m_{c}},\forall k>0. \label{1810131}%
\end{equation}
To see last inequality, we do in the following way:
\begin{align*}
\mathrm{meas}\left\{  {|\nabla v|>k}\right\}   &  \leqq\frac{1}{{{k^{p}}}}%
\int_{0}^{{k^{p}}}{\mathrm{{meas}}}\left(  {\left\{  {|\nabla v{|^{p}}%
>s}\right\}  }\right)  ds\\
&  \leqq\mathrm{{meas}}\left\{  {|v|>M^{\frac{1}{N+1}}k^{\frac{N}{{N+1}}}%
}\right\}  +\frac{1}{{{k^{p}}}}\int_{0}^{{k^{p}}}{\mathrm{{meas}}\left\{
|\nabla v{|^{p}}>s,|v|\leqq M^{\frac{1}{N+1}}k^{\frac{N}{{N+1}}}\right\}
}ds\\
&  \leqq C_{4}M^{\frac{N+2}{N+1}}k^{-m_{c}}.
\end{align*}

\end{remark}

\begin{proposition}
\label{mun} Let $\{\mu_{n}\}$ $\subset$ $\mathcal{M}_{b}(${$Q$}$),$ and
$\{u_{0,n}\}\subset L^{1}(\Omega),$ with
\[
\sup_{n}\left\vert {{\mu_{n}}}\right\vert ({Q})<\infty,\text{ and }\sup
_{n}||{{u_{0,n}}}||_{1,\Omega}<\infty.
\]
Let $u_{n}$ be a R-solution of (\ref{pmu}) with data $\mu_{n}=\mu_{n,0}%
+\mu_{n,s}$ and $u_{0,n},$ relative to a decomposition $(f_{n},g_{n},h_{n})$
of $\mu_{n,0}$, and $v_{n}=u_{n}-h_{n}.$ Assume that $\{f_{n}\}$ is bounded in
$L^{1}(Q)$, $\{g_{n}\}$ bounded in $(L^{p^{\prime}}(Q))^{N}$ and $\{h_{n}\}$
bounded in $X$. \medskip

\noindent Then, up to a subsequence, $\{v_{n}\}$ converges $a.e.$ to a
function $v,$ such that $T_{k}(v)\in X$ and $v\in L^{\sigma}((0,T);W_{0}%
^{1,\sigma}(\Omega))\cap{L^{\infty}}((0,T);{L^{1}}(\Omega))$ for any
$\sigma\in\lbrack1,m_{c}).$ And\medskip

{(i)} $\left\{  v_{n}\right\}  $ converges to $v$ strongly in $L^{\sigma}(Q)$
for any $\sigma\in\lbrack1,m_{c}),$ and $\sup{\left\Vert {{v_{n}}}\right\Vert
_{{L^{\infty}}((0,T);{L^{1}}(\Omega))}}<\infty,$\medskip

{(ii)} $\sup_{k>0}\sup_{n}\frac{1}{k+1}\int_{Q}|\nabla T_{k}(v_{n}%
)|^{p}<\infty$,\medskip

{(iii)} $\left\{  T_{k}(v_{n})\right\}  $ converges to $T_{k}(v)$ wealkly in
$X,$ for any $k>0$,\medskip

{(iv)} $\left\{  A\left(  x,t,\nabla\left(  T_{k}(v_{n})+h_{n}\right)
\right)  \right\}  $ converges to some $F_{k}$ weakly in $(L^{p^{\prime}%
}(Q))^{N}$. \medskip
\end{proposition}

\begin{proof}
Take $S\in W^{2,\infty}(\mathbb{R})$ such that $S^{\prime}$ has compact
support on $\mathbb{R}$ and $S(0)=0$. We combine (\ref{11051}) with
(\ref{alp}), and deduce that $\{S(v_{n})_{t}\}$ is bounded in ${X}^{\prime
}+{L^{1}}(${$Q$}$)$ and $\{S(v_{n})\}$ bounded in $X$. Hence, $\{S(v_{n})\}$
is relatively compact in $L^{1}(Q)$. On the other hand, we choose $S=S_{k}$
such that $S_{k}(z)=z,$ if $\left\vert z\right\vert <k$ and $S(z)=2k\;$%
sign$z,$ if $|z|>2k.$ Thanks to (\ref{gam}), we obtain%
\begin{align}
\mathrm{meas}\left\{  {\left\vert {{v_{n}}-{v_{m}}}\right\vert >\sigma
}\right\}   &  \leqq\mathrm{meas}\left\{  {\left\vert {{v_{n}}}\right\vert
>k}\right\}  +\mathrm{meas}\left\{  {\left\vert {{v_{m}}}\right\vert
>k}\right\}  +\mathrm{meas}\left\{  {\left\vert {{S_{k}}({v_{n}})-{S_{k}%
}({v_{m}})}\right\vert >\sigma}\right\} \nonumber\\
&  \leqq\frac{1}{k}({{{\left\Vert {{v_{n}}}\right\Vert }_{1,Q}}+{{\left\Vert
{{v_{m}}}\right\Vert }_{1,Q})}}+\mathrm{meas}\left\{  {\left\vert {{S_{k}%
}({v_{n}})-{S_{k}}({v_{m}})}\right\vert >\sigma}\right\} \nonumber\\
&  \leqq\frac{C}{k}+\mathrm{meas}\left\{  {\left\vert {{S_{k}}({v_{n}}%
)-{S_{k}}({v_{m}})}\right\vert >\sigma}\right\}  . \label{april244}%
\end{align}
Thus, up to a subsequence $\{u_{n}\}$ is a Cauchy sequence in measure, and
converges $a.e.$ in $Q$ to a function $u$. Thus, $\left\{  T_{k}%
(v_{n})\right\}  $ converges to $T_{k}(v)$ weakly in $X$, since $\sup
_{n}{\left\Vert {{T_{k}}({v_{n}})}\right\Vert _{X}}<\infty$ for any $k>0$. And
$\left\{  |\nabla\left(  T_{k}(v_{n})+h_{n}\right)  |^{p-2}\nabla\left(
T_{k}(v_{n})+h_{n}\right)  \right\}  $ converges to some $F_{k}$ weakly in
$(L^{p^{\prime}}(Q))^{N}$. Furthermore, from (\ref{mess}), $\left\{
v_{n}\right\}  $ converges to $v$ strongly in $L^{\sigma}(Q),$ for any
$\sigma<p_{c}.$
\end{proof}

\section{The convergence theorem\label{cv}}

We first recall some properties of the measures, see \cite[Lemma 5]{Pe08},
\cite{DMOP}.

\begin{proposition}
\label{04041} Let $\mu_{s}=\mu_{s}^{+}-\mu_{s}^{-}\in\mathcal{M}_{b}(Q),$
where $\mu_{s}^{+}$ and $\mu_{s}^{-}$ are concentrated, respectively, on two
disjoint sets $E^{+}$ and $E^{-}$ of zero $c_{p}^{Q}$-capacity. Then, for any
$\delta>0$, there exist two compact sets $K_{\delta}^{+}\subseteq E^{+}$ and
$K_{\delta}^{-}\subseteq E^{-}$ such that
\[
\mu_{s}^{+}(E^{+}\backslash K_{\delta}^{+})\leqq\delta,\text{\qquad}\mu
_{s}^{-}(E^{-}\backslash K_{\delta}^{-})\leqq\delta,
\]
and there exist $\psi_{\delta}^{+},\psi_{\delta}^{-}\in C_{c}^{1}(Q)$ with
values in $\left[  0,1\right]  ,$ such that $\psi_{\delta}^{+},\psi_{\delta
}^{-}=1$ respectively on $K_{\delta}^{+},K_{\delta}^{-},$ and $\text{supp}%
(\psi_{\delta}^{+})\cap\text{supp}(\psi_{\delta}^{-})=\emptyset$, and
\[
||\psi_{\delta}^{+}||_{X}+||(\psi_{\delta}^{+})_{t}||_{X^{\prime}+L^{1}%
(Q)}\leqq\delta,\qquad||\psi_{\delta}^{-}||_{X}+||(\psi_{\delta}^{-}%
)_{t}||_{X^{\prime}+L^{1}(Q)}\leqq\delta.
\]
There exist decompositions $(\psi_{\delta}^{+})_{t}={\left(  {\psi_{\delta
}^{+}}\right)  _{t}^{1}+\left(  {\psi_{\delta}^{+}}\right)  _{t}^{2}}$ and
$(\psi_{\delta}^{-})_{t}={\left(  {\psi_{\delta}^{-}}\right)  _{t}^{1}+\left(
{\psi_{\delta}^{-}}\right)  _{t}^{2}}$ in $X^{\prime}+L^{1}(Q),$ such that
\begin{equation}
{\left\Vert {\left(  {\psi_{\delta}^{+}}\right)  _{t}^{1}}\right\Vert
_{{X}^{\prime}}}\leqq\frac{\delta}{3},\qquad{\left\Vert {\left(  {\psi
_{\delta}^{+}}\right)  _{t}^{2}}\right\Vert _{1,Q}}\leqq\frac{\delta}%
{3},\qquad{\left\Vert {\left(  {\psi_{\delta}^{-}}\right)  _{t}^{1}%
}\right\Vert _{{X}^{\prime}}}\leqq\frac{\delta}{3},\qquad{\left\Vert {\left(
{\psi_{\delta}^{-}}\right)  _{t}^{2}}\right\Vert _{1,Q}}\leqq\frac{\delta}{3}.
\label{41}%
\end{equation}
Both $\left\{  \psi_{\delta}^{+}\right\}  $ and $\left\{  \psi_{\delta}%
^{-}\right\}  $ converge to $0$, $^{\ast}$-weakly in $L^{\infty}(Q)$, and
strongly in $L^{1}(Q)$ and up to subsequences, $a.e.$ in $Q,$ as $\delta$
tends to $0$.

Moreover if $\rho_{n}$ and $\eta_{n}$ are as in Theorem \ref{sta}, we have,
for any $\delta,\delta_{1},\delta_{2}>0,$
\begin{equation}
\int_{Q}{\psi_{\delta}^{-}}d\rho_{n}+\int_{Q}{\psi_{\delta}^{+}}d\eta
_{n}=\omega(n,\delta),\qquad\int_{Q}{\psi_{\delta}^{-}}d\mu_{s}^{+}\leqq
\delta,\qquad\int_{Q}{\psi_{\delta}^{+}}d\mu_{s}^{-}\leqq\delta, \label{12054}%
\end{equation}%
\begin{equation}
\int_{Q}{(1-\psi_{\delta_{1}}^{+}\psi_{\delta_{2}}^{+})}d\rho_{n}%
=\omega(n,\delta_{1},\delta_{2}),\qquad\int_{Q}{(1-\psi_{\delta_{1}}^{+}%
\psi_{\delta_{2}}^{+})}d\mu_{s}^{+}\leqq\delta_{1}+\delta_{2}, \label{12056}%
\end{equation}%
\begin{equation}
\int_{Q}{(1-\psi_{\delta_{1}}^{-}\psi_{\delta_{2}}^{-})}d\eta_{n}%
=\omega(n,\delta_{1},\delta_{2}),\qquad\int_{Q}{(1-\psi_{\delta_{1}}^{-}%
\psi_{\delta_{2}}^{-})}d\mu_{s}^{-}\leqq\delta_{1}+\delta_{2}. \label{12057}%
\end{equation}

\end{proposition}

Hereafter, if $n,\varepsilon,...,\nu$ are real numbers, and a function $\phi$
depends on $n,\varepsilon,...,\nu$ and eventual other parameters $\alpha
,\beta,..,\gamma$, and $n\rightarrow n_{0},\varepsilon\rightarrow
\varepsilon_{0},..,$ $\nu\rightarrow\nu_{0}$, we write $\phi=\omega
(n,\varepsilon,..,\nu)$, then this means\textbf{ }$\overline{\lim}%
_{\nu\rightarrow{\nu_{0}}}..\overline{\lim}_{\varepsilon\rightarrow
{\varepsilon_{0}}}\overline{\lim}_{n\rightarrow{n_{0}}}\left\vert
\phi\right\vert =0,$\textbf{ }when the parameters\textbf{ }$\alpha
,\beta,..,\gamma$\textbf{ }are fixed\textbf{.} In the same way, $\phi
\leqq\omega(n,\varepsilon,\delta,...,\nu)$ means $\overline{\lim}%
_{\nu\rightarrow{\nu_{0}}}..\overline{\lim}_{\varepsilon\rightarrow
{\varepsilon_{0}}}\overline{\lim}_{n\rightarrow{n_{0}}}\phi\leqq0$, and $\phi$
$\geqq\omega(n,\varepsilon,..,\nu)$ means $-\phi\leqq\omega(n,\varepsilon
,..,\nu).$

\begin{remark}
\label{05041}In the sequel we use a convergence property, consequence of the
Dunford-Pettis theorem, still used in \cite{DMOP}: If $\left\{  a_{n}\right\}
$ is a sequence in $L^{1}(Q)$ converging to $a$ weakly in $L^{1}(Q)$ and
$\left\{  b_{n}\right\}  $ a bounded sequence in $L^{\infty}(Q)$ converging to
$b,$ $a.e.$ in $Q,$ then $\lim_{n\rightarrow\infty}\int_{Q}{{a_{n}}{b_{n}}%
}=\int_{Q}{ab.}$
\end{remark}

Next we prove Thorem \ref{sta}.\bigskip

\begin{proof}
[Scheme of the proof]Let $\{\mu_{n}\},\left\{  u_{0,n}\right\}  $ and
$\left\{  u_{n}\right\}  $ satisfying the assumptions of Theorem \ref{sta}.
Then we can apply Proposition \ref{mun}. Setting $v_{n}=u_{n}-h_{n},$ up to
subsequences, $\left\{  u_{n}\right\}  $ converges $a.e.$ in $Q$ to some
function $u,$ and $\left\{  v_{n}\right\}  $ converges $a.e.$ to $v=u-h,$ such
that $T_{k}(v)\in X$ and $v\in L^{\sigma}((0,T);W_{0}^{1,\sigma}(\Omega
))\cap{L^{\infty}}((0,T);{L^{1}}(\Omega))$ for every $\sigma\in\left[
1,m_{c}\right)  $. And $\{v_{n}\}$ satisfies the conclusions (i) to (iv) of
Proposition \ref{mun}. We have
\begin{align*}
\mu_{n}  &  =(f_{n}-\operatorname{div}g_{n}+(h_{n})_{t})+(\rho_{n}%
^{1}-\operatorname{div}\rho_{n}^{2})-(\eta_{n}^{1}-\operatorname{div}\eta
_{n}^{2})+\rho_{n,s}-\eta_{n,s}\\
&  =\mu_{n,0}+(\rho_{n,s}-\eta_{n,s})^{+}-(\rho_{n,s}-\eta_{n,s})^{-},
\end{align*}
where
\begin{equation}
\mu_{n,0}=\lambda_{n,0}+\rho_{n,0}-\eta_{n,0},\text{ \quad with }\lambda
_{n,0}=f_{n}-\operatorname{div}g_{n}+(h_{n})_{t},\quad\rho_{n,0}=\rho_{n}%
^{1}-\operatorname{div}\rho_{n}^{2},\quad\eta_{n,0}=\eta_{n}^{1}%
-\operatorname{div}\eta_{n}^{2}. \label{muni}%
\end{equation}
Hence
\begin{equation}
\rho_{n,0},\eta_{n,0}\in\mathcal{M}_{b}^{+}(Q)\cap\mathcal{M}_{0}%
(Q),\text{\quad and\quad}\rho_{n}\geqq\rho_{n,0},\quad\eta_{n}\geqq\eta_{n,0}.
\label{muno}%
\end{equation}

\noindent Let $E^{+},E^{-}$ be the sets where, respectively, $\mu_{s}^{+}$ and
$\mu_{s}^{-}$ are concentrated. For any $\delta_{1},\delta_{2}>0$, let
$\psi_{\delta_{1}}^{+},\psi_{\delta_{2}}^{+}$ and $\psi_{\delta_{1}}^{-}%
,\psi_{\delta_{2}}^{-}$ as in Proposition \ref{04041} and set
\[
\Phi_{\delta_{1},\delta_{2}}=\psi_{\delta_{1}}^{+}\psi_{\delta_{2}}^{+}%
+\psi_{\delta_{1}}^{-}\psi_{\delta_{2}}^{-}.
\]
\textit{Suppose that we can prove the two estimates, near }$E$
\begin{equation}
I_{1}:=\int\limits_{\left\{  |v_{n}|\leqq k\right\}  }{\Phi_{\delta_{1}%
,\delta_{2}}A(x,t,\nabla u_{n}).\nabla\left(  {{v_{n}}-}\langle T_{k}%
(v)\rangle_{\nu}\right)  }\leqq\omega(n,\nu,\delta_{1},\delta_{2}),
\label{12059}%
\end{equation}
\textbf{ }\textit{and far from} $E,$%
\begin{equation}
I_{2}:=\int\limits_{\left\{  |v_{n}|\leqq k\right\}  }{(1-\Phi_{\delta
_{1},\delta_{2}})A(x,t,\nabla u_{n}).\nabla({{v_{n}}-}\langle T_{k}%
(v)\rangle_{\nu})}\leqq\omega(n,\nu,\delta_{1},\delta_{2}). \label{120510}%
\end{equation}
Then it follows that
\begin{equation}
\overline{\lim}_{n,\nu}\int\limits_{\left\{  |v_{n}|\leqq k\right\}
}{A(x,t,\nabla u_{n}).\nabla\left(  {{v_{n}}-}\langle{{{{{T_{k}}(v)\rangle}%
}_{\nu}}}\right)  }\leqq0, \label{12052}%
\end{equation}
which implies%
\begin{equation}
\overline{\lim}_{n\rightarrow\infty}\int\limits_{\left\{  |v_{n}|\leqq
k\right\}  }{A(x,t,\nabla u_{n}).\nabla\left(  {{v_{n}}-{T_{k}}(v)}\right)
}\leqq0, \label{12061}%
\end{equation}
since $\left\{  \langle{{{{{T_{k}}(v)\rangle}}_{\nu}}}\right\}  $ converges to
$T_{k}(v)$ in $X.$ On the other hand, from the weak convergence of $\left\{
T_{k}(v_{n})\right\}  $ to $T_{k}(v)$ in $X,$ we verify that%
\[
\int\limits_{\left\{  |v_{n}|\leqq k\right\}  }{A(x,t,\nabla(T_{k}%
(v)+h_{n})).\nabla\left(  {{T_{k}}({v_{n}})-{T_{k}}(v)}\right)  }=\omega(n).
\]
Thus we get
\[
\int\limits_{\left\{  |v_{n}|\leqq k\right\}  }{\left(  {A(x,t,\nabla
u_{n})-A(x,t,\nabla(T_{k}(v)+h_{n}))}\right)  .\nabla\left(  {{u_{n}}-\left(
{{T_{k}}(v)+{h_{n}}}\right)  }\right)  }=\omega(n).
\]
Then, it is easy to show that, up to a subsequence,
\begin{equation}
\left\{  \nabla u_{n}\right\}  \text{ converges to }\nabla u,\qquad\text{
}a.e.\text{ in }Q. \label{pp}%
\end{equation}
Therefore, $\left\{  A(x,t,\nabla u_{n})\right\}  $ converges to $A(x,t,\nabla
u)$ weakly in $(L^{p^{\prime}}(Q))^{N}$ ; and from (\ref{12061}) we find
\[
\overline{\lim}_{n\rightarrow\infty}\int_{Q}A(x,t,\nabla u_{n}).\nabla{T_{k}%
}({v_{n}})\leqq\int_{Q}A(x,t,\nabla u)\nabla T_{k}(v).
\]
Otherwise, $\left\{  {A(x,t,\nabla\left(  {{T_{k}}(v_{n})+{h_{n}}}\right)
)}\right\}  $ converges weakly in $(L^{p^{\prime}}(Q))^{N}$to some $F_{k},$
from Proposition \ref{mun}, and we obtain that $F_{k}={A(x,t,\nabla\left(
{{T_{k}}(v)+{h}}\right)  ).}$ Hence
\begin{align*}
\overline{\lim}_{n\rightarrow\infty}\int_{Q}A(x,t,\nabla(T_{k}(v_{n}%
)+h_{n})).\nabla(T_{k}(v_{n})+h_{n})  &  \leqq\overline{\lim}_{n\rightarrow
\infty}\int_{Q}A(x,t,\nabla u_{n}).\nabla T_{k}(v_{n})\\
&  +\overline{\lim}_{n\rightarrow\infty}\int_{Q}A(x,t,\nabla(T_{k}%
(v_{n})+h_{n})).\nabla h_{n}\\
&  \leqq\int_{Q}A(x,t,\nabla(T_{k}(v)+h)).\nabla(T_{k}(v)+h).
\end{align*}
As a consequence
\begin{equation}
\left\{  T_{k}(v_{n})\right\}  \text{ converges to }T_{k}(v),\text{ strongly
in }X,\qquad\forall k>0. \label{02041}%
\end{equation}
Then \textit{to finish the proof we have to check that }$u$\textit{ is a
solution of} (\ref{pmu}).\medskip\medskip
\end{proof}

In order to prove (\ref{12059}) we need a first Lemma, inspired of \cite[Lemma
6.1]{DMOP}, extending \cite[Lemma 6 and Lemma 7]{Pe08}:

\begin{lemma}
\label{april261}Let $\psi_{1,\delta},\psi_{2,\delta}\in C^{1}(Q)$ be uniformly
bounded in $W^{1,\infty}(Q)$ with values in $[0,1],$ such that$\int_{Q}%
{\psi_{1,\delta}}d\mu_{s}^{-}\leqq\delta$ and $\int_{Q}{\psi_{2,\delta}}%
d\mu_{s}^{+}\leqq\delta$. Then, under the assumptions of Theorem \ref{sta},
\begin{equation}
\frac{1}{m}\int\limits_{\left\{  m\leqq{v_{n}}<2m\right\}  }{{{\left\vert
{\nabla{u_{n}}}\right\vert }^{p}}{\psi_{2,\delta}}}=\omega(n,m,\delta
),\quad\quad\frac{1}{m}\int\limits_{\left\{  m\leqq{v_{n}}<2m\right\}
}{{{\left\vert {\nabla{v_{n}}}\right\vert }^{p}}{\psi_{2,\delta}}}%
=\omega(n,m,\delta), \label{13051}%
\end{equation}%
\begin{equation}
\frac{1}{m}\int\limits_{-2m<{v_{n}}\leqq-m}{{{\left\vert {\nabla{u_{n}}%
}\right\vert }^{p}}{\psi_{1,\delta}}}=\omega(n,m,\delta),\qquad\frac{1}{m}%
\int\limits_{-2m<{v_{n}}\leqq-m}{{{\left\vert {\nabla{v_{n}}}\right\vert }%
^{p}}{\psi_{1,\delta}}}=\omega(n,m,\delta), \label{13052}%
\end{equation}
and for any $k>0,$%
\begin{equation}
\int\limits_{\left\{  m\leqq{v_{n}}<m+k\right\}  }{{{\left\vert {\nabla{u_{n}%
}}\right\vert }^{p}}{\psi_{2,\delta}}}=\omega(n,m,\delta),\qquad
\int\limits_{\left\{  m\leqq{v_{n}}<m+k\right\}  }{{{\left\vert {\nabla{v_{n}%
}}\right\vert }^{p}}{\psi_{2,\delta}}}=\omega(n,m,\delta), \label{13053}%
\end{equation}%
\begin{equation}
\int\limits_{\left\{  -m-k<{v_{n}}\leqq-m\right\}  }{{{\left\vert
{\nabla{u_{n}}}\right\vert }^{p}}{\psi_{1,\delta}}}=\omega(n,m,\delta
),\qquad\int\limits_{\left\{  -m-k<{v_{n}}\leqq-m\right\}  }{{{\left\vert
{\nabla{v_{n}}}\right\vert }^{p}}{\psi_{1,\delta}}}=\omega(n,m,\delta).
\label{13054}%
\end{equation}

\end{lemma}

\begin{proof}
(i) Proof of (\ref{13051}), (\ref{13052}). Set for any $r\in\mathbb{R}$ and
any $m,\ell\geqq1$%
\[
{S_{m,\ell}}(r)=\int_{0}^{r}{\left(  {\frac{{-m+\tau}}{m}{\chi_{\lbrack
m,2m]}}(\tau)+{\chi_{(2m,2m+\ell]}}(\tau)+\frac{{4m+2h-\tau}}{{2m+\ell}}%
{\chi_{(2m+\ell,4m+2h]}}(\tau)}\right)  d\tau,}%
\]%
\[
{S_{m}}(r){=}\int_{0}^{r}{\left(  {\frac{{-m+\tau}}{m}{\chi_{\lbrack m,2m]}%
}(\tau)+{\chi_{(2m,\infty)}}(\tau)}\right)  d\tau}.
\]
Note that ${S}_{m,\ell}^{\prime\prime}{=\chi}_{\left[  m,2m\right]  }/m-{\chi
}_{\left[  2m+\ell,2(2m+\ell)\right]  }/(2m+\ell).$ We choose $(\xi
,J,S)=(\psi_{2,\delta},T_{1},S_{m,\ell})$ as test functions in (\ref{parti})
for $u_{n},$ and observe that, from (\ref{muni}),
\begin{equation}
\widehat{\mu_{n,0}}=\mu_{n,0}-(h_{n})_{t}=\widehat{\lambda_{n,0}}+\rho
_{n,0}-\eta_{n,0}=f_{n}-\operatorname{div}g_{n}+\rho_{n,0}-\eta_{n,0}.
\label{xxx}%
\end{equation}
Thus we can write $%
{\textstyle\sum_{i=1}^{6}}
A_{i}\leqq%
{\textstyle\sum_{i=7}^{12}}
A_{i},$ where
\begin{align*}
A_{1}  &  =-\int\limits_{\Omega}{{\psi_{2,\delta}}(0){T_{1}}({S_{m,\ell}%
}({u_{0,n}})){S_{m,\ell}}({u_{0,n}}),\quad}A_{2}=-\int_{Q}{{{\left(
{{\psi_{2,\delta}}}\right)  }_{t}}\overline{T_{1}}({S_{m,\ell}}({v_{n}})),}\\
A_{3}  &  =\int_{Q}{{{S}_{m,\ell}^{^{\prime}}}({v_{n}}){T_{1}}({S_{m,\ell}%
}({v_{n}}))A(x,t,\nabla u_{n})\nabla{\psi_{2,\delta}},}\\
A_{4}  &  =\int_{Q}{{{{{{S}_{m,\ell}^{^{\prime}}}({v_{n}})}}^{2}\psi
_{2,\delta}}{T_{1}^{^{\prime}}}({S_{m,\ell}}({v_{n}}))A(x,t,\nabla
u_{n})\nabla{v_{n},}}\\
A_{5}  &  =\frac{1}{m}\int\limits_{\left\{  m\leqq{v_{n}}\leqq2m\right\}
}{{\psi_{2,\delta}}{T_{1}}({S_{m,\ell}}({v_{n}}))A(x,t,\nabla u_{n}%
)\nabla{v_{n}},}%
\end{align*}%
\begin{align*}
A_{6}  &  =-\frac{1}{{2m+\ell}}\int\limits_{\left\{  2m+\ell\leqq{v_{n}%
}<2(2m+\ell)\right\}  }{\psi_{2,\delta}A(x,t,\nabla u_{n})\nabla{v_{n},}}\\
A_{7}  &  =\int_{Q}{{{S^{\prime}}_{m,\ell}}({v_{n}}){T_{1}}({S_{m,\ell}%
}({v_{n}})){\psi_{2,\delta}}{f_{n},\quad\quad}}A_{8}=\int_{Q}{{{S}_{m,\ell
}^{^{\prime}}}({v_{n}}){T_{1}}({S_{m,\ell}}({v_{n}})){g_{n}.}\nabla
{\psi_{2,\delta},}}\\
A_{9}  &  =\int_{Q}{{{\left(  {{{S}_{m,\ell}^{^{\prime}}}({v_{n}})}\right)
}^{2}}{T_{1}^{^{\prime}}}({S_{m,\ell}}({v_{n}})){\psi_{2,\delta}g_{n}.}%
\nabla{v_{n},\quad\quad}}A_{10}=\frac{1}{m}\int\limits_{m\leqq{v_{n}}\leqq
2m}{{T_{1}}({S_{m,\ell}}({v_{n}})){\psi_{2,\delta}g_{n}.}\nabla{v_{n},}}\\
A_{11}  &  =-\frac{1}{{2m+\ell}}\int\limits_{\left\{  2m+\ell\leqq{v_{n}%
}<2(2m+\ell)\right\}  }{{\psi_{2,\delta}g_{n}.}\nabla{v_{n}},\quad}A_{12}%
=\int_{Q}{{{S}_{m,\ell}^{^{\prime}}}({v_{n}}){T_{1}}({S_{m,\ell}}({v_{n}%
})){\psi_{2,\delta}}d\left(  {{\rho_{n,0}}-{\eta_{n,0}}}\right)  .}%
\end{align*}
Since $||S_{m,\ell}(u_{0,n})||_{1,\Omega}\leqq\int\limits_{\left\{  m\leqq
u_{0,n}\right\}  }u_{0,n}dx$, we find $A_{1}=\omega(\ell,n,m)$. Otherwise
\[
|A_{2}|\leqq{\left\Vert {{\psi_{2,\delta}}}\right\Vert _{{W^{1,\infty}}({Q})}%
}\int\limits_{\left\{  m\leqq v_{n}\right\}  }{{v_{n}}},\qquad|A_{3}%
|\leqq{\left\Vert {{\psi_{2,\delta}}}\right\Vert _{{W^{1,\infty}}({Q})}}%
\int\limits_{\left\{  m\leqq v_{n}\right\}  }\left(  |a|+c_{2}{\left\vert
{\nabla{u_{n}}}\right\vert }^{p-1}\right)  ,
\]
which implies $A_{2}=\omega(\ell,n,m)$ and $A_{3}=\omega(\ell,n,m).$ Using
(\ref{renor2}) for $u_{n}$, we have
\[
A_{6}=-\int_{Q}{{\psi_{2,\delta}}d{{\left(  {{\rho_{n,s}}-{\eta_{n,s}}%
}\right)  }^{+}}}+\omega(\ell)=\omega(\ell,n,m,\delta).
\]
Hence $A_{6}=\omega(\ell,n,m,\delta),$ since ${{{\left(  {{\rho_{n,s}}%
-{\eta_{n,s}}}\right)  }^{+}}}$ converges to $\mu_{s}^{+}$ as $n\rightarrow
\infty$ in the narrow topology, and $\int_{Q}{\psi_{2,\delta}}d\mu_{s}%
^{+}\leqq\delta.$ We also obtain $A_{11}=\omega(\ell)$ from (\ref{lim3}).

\noindent Now $\left\{  S_{m,\ell}^{^{\prime}}(v_{n})T_{1}(S_{m,\ell}%
(v_{n}))\right\}  _{\ell}$ converges to $S_{m}^{^{\prime}}(v_{n})T_{1}%
(S_{m}(v_{n}))$, $\left\{  S_{m}^{^{\prime}}(v_{n})T_{1}(S_{m}(v_{n}%
))\right\}  _{n}$ converges to $S_{m}^{^{\prime}}(v)$ $T_{1}(S_{m}(v))$,
$\left\{  S_{m}^{^{\prime}}(v)T_{1}(S_{m}(v))\right\}  _{m}$ converges to $0$,
$\ast$-weakly in $L^{\infty}(Q),$ and $\left\{  f_{n}\right\}  $ converges to
$f$ weakly in $L^{1}(Q)$, $\left\{  g_{n}\right\}  $ converges to $g$ strongly
in $(L^{p^{\prime}}(Q))^{N}$. From Remark \ref{05041}, we obtain%
\begin{align*}
A_{7}  &  =\int_{Q}{{{S^{\prime}}_{m}}({v_{n}}){T_{1}}({S_{m}}({v_{n}}%
)){\psi_{2,\delta}}{f_{n}}}+\omega(\ell)=\int_{Q}{{{S^{\prime}}_{m}}(v){T_{1}%
}({S_{m}}(v)){\psi_{2,\delta}}f}+\omega(\ell,n)=\omega(\ell,n,m),\\
A_{8}  &  =\int_{Q}{{{S^{\prime}}_{m}}({v_{n}}){T_{1}}({S_{m}}({v_{n}}%
)){g_{n}.}\nabla{\psi_{2,\delta}}}+\omega(\ell)=\int_{Q}{{{S^{\prime}}_{m}%
}(v){T_{1}}({S_{m}}(v))g\nabla{\psi_{2,\delta}}}+\omega(\ell,n)=\omega
(\ell,n,m).
\end{align*}
\newline Otherwise, $A_{12}\leqq\int_{Q}{{\psi_{2,\delta}}d{\rho_{n}}}$, and
$\left\{  \int_{Q}{{\psi_{2,\delta}}d{\rho_{n}}}\right\}  $ converges to
$\int_{Q}{\psi_{2,\delta}}d\mu_{s}^{+},$ thus $A_{12}\leqq\omega
(\ell,n,m,\delta)$.

\noindent Using Holder inequality and the condition \eqref{condi1} we have
\[
g_{n}.\nabla v_{n}-A(x,t,\nabla u_{n})\nabla{v_{n}}\leq C_{1}\left(
|g_{n}|^{p^{\prime}}+|\nabla h_{n}|^{p}+|a|^{p^{\prime}}\right)
\]
with $C_{1}=C_{1}(p,c_{2}),$ which implies
\[
A_{9}-A_{4}\leqq C_{1}\int_{Q}{{{\left(  {{{S^{\prime}}_{m,\ell}}({v_{n}}%
)}\right)  }^{2}T}}_{1}^{\prime}{({S_{m,\ell}}({v_{n}})){\psi_{2,\delta}%
}\left(  {|{g_{n}}{|^{p^{\prime}}}+|{h_{n}}{|^{p}}}+|a|^{p^{\prime}}\right)
=\;}\omega(\ell,n,m).
\]
Similarly we also show that $A_{10}-A_{5}/2\leqq\omega(\ell,n,m)$. Combining
the estimates, we get $A_{5}/2\leqq\omega(\ell,n,m,\delta)$. Using Holder
inequality we have
\[
A(x,t,\nabla u_{n})\nabla v_{n}\geq\frac{c_{1}}{2}|\nabla u_{n}|^{p}%
-C_{2}(|a|^{p^{\prime}}+|\nabla h_{n}|^{p}).
\]
with $C_{2}=C_{2}(p,c_{1},c_{2}),$ which implies
\[
\frac{1}{m}\int\limits_{\left\{  m\leqq{v_{n}}<2m\right\}  }{{{\left\vert
{\nabla{u_{n}}}\right\vert }^{p}}{\psi_{2,\delta}}{T_{1}}({S_{m,\ell}}({v_{n}%
}))=\;}\omega(\ell,n,m,\delta).
\]
Note that for all $m$ $>4$, $S_{m,\ell}(r)\geqq1$ for any $r\in\lbrack\frac
{3}{2}m,2m];$ hence $T_{1}(S_{m,\ell}(r)=1.$ So,
\[
\frac{1}{m}\int\limits_{\left\{  \frac{3}{2}m\leqq{v_{n}}<2m\right\}
}{{{\left\vert {\nabla{u_{n}}}\right\vert }^{p}}{\psi_{2,\delta}}}=\omega
(\ell,n,m,\delta).
\]
Since ${\left\vert {\nabla{v_{n}}}\right\vert ^{p}}\leqq{2^{p-1}}{\left\vert
{\nabla{u_{n}}}\right\vert ^{p}}+{2^{p-1}}{\left\vert {\nabla{h_{n}}%
}\right\vert ^{p}}$, there also holds
\[
\frac{1}{m}\int\limits_{\left\{  \frac{3}{2}m\leqq{v_{n}}<2m\right\}
}{{{\left\vert {\nabla{v_{n}}}\right\vert }^{p}}{\psi_{2,\delta}}}=\omega
(\ell,n,m,\delta).
\]
We deduce (\ref{13051}) by summing on each set $\left\{  (\frac{4}{3})^{\nu
}m\leqq{v_{n}}\leqq(\frac{4}{3})^{\nu+1}m\right\}  $ for $\nu=0,1,2.$
Similarly, we can choose $(\xi,\psi,S)=(\psi_{1,\delta},T_{1},\tilde
{S}_{m,\ell})$ as test functions in (\ref{parti}) for $u_{n},$ where
$\tilde{S}_{m,\ell}(r)=$ ${S_{m,\ell}}(-r),$ and we obtain (\ref{13052}%
).\medskip\ 

(ii) Proof of (\ref{13053}), (\ref{13054}). We set, for any $k,m,\ell\geqq1,$%
\[
{S_{k,m,\ell}}(r)=\int_{0}^{r}{\left(  {{T_{k}}(\tau-{T_{m}}(\tau
)){\chi_{\lbrack m,k+m+\ell]}}+k\frac{{2(k+\ell+m)-\tau}}{{k+m+\ell}}%
{\chi_{(k+m+\ell,2(k+m+\ell)]}}}\right)  d\tau}%
\]%
\[
{S_{k,m}}(r)=\int\limits_{0}^{s}{{T_{k}}(\tau-{T_{m}}(\tau)){\chi_{\lbrack
m,\infty)}}d\tau.}%
\]
We choose $(\xi,\psi,S)=(\psi_{2,\delta},T_{1},S_{k,m,\ell})$ as test
functions in (\ref{parti}) for $u_{n}$. In the same way we also obtain
\[
\int\limits_{\left\{  m\leqq{v_{n}}<m+k\right\}  }{{{\left\vert {\nabla{u_{n}%
}}\right\vert }^{p}}{\psi_{2,\delta}}{T_{1}}({S_{k,m,\ell}}({v_{n}}))}%
=\omega(\ell,n,m,\delta).
\]
Note that $T_{1}(S_{k,m,\ell}(r))$ $=1$ for any $r$ $\geqq m+1$, thus
$\int\limits_{\left\{  m+1\leqq{v_{n}}<m+k\right\}  }{{{\left\vert
{\nabla{u_{n}}}\right\vert }^{p}}{\psi_{2,\delta}}}=\omega(n,m,\delta),$ which
implies (\ref{13053}) by changing $m$ into $m-1$. Similarly, we obtain
(\ref{13054}).\medskip
\end{proof}

Next we look at the behaviour near $E.$

\begin{lemma}
\label{near} Estimate (\ref{12059}) holds.
\end{lemma}

\begin{proof}
There holds%
\[
I_{1}=\int_{Q}{\Phi_{\delta_{1},\delta_{2}}A(x,t,\nabla u_{n}).\nabla{T_{k}%
}({v_{n}})-}\int\limits_{\left\{  |v_{n}|\leqq k\right\}  }{\Phi_{\delta
_{1},\delta_{2}}A(x,t,\nabla u_{n}).\nabla{\langle T_{k}(v)\rangle}_{\nu}{.}}%
\]
From Proposition \ref{mun}, (iv), $\left\{  A(x,t,\nabla\left(  T_{k}%
(v_{n})+h_{n}\right)  ).\nabla\langle T_{k}(v)\rangle_{\nu}\right\}  $
converges weakly in $L^{1}(Q)$ to $F_{k}\nabla\langle T_{k}(v)\rangle_{\nu}$
$.$ And $\left\{  \chi_{\left\{  |v_{n}|\leqq k\right\}  }\right\}  $
converges to $\chi_{|v|\leqq k},$ $a.e.$ in $Q$ , and $\Phi_{\delta_{1}%
,\delta_{2}}$ converges to $0$ $a.e.$ in $Q$ as $\delta_{1}\rightarrow0,$ and
$\Phi_{\delta_{1},\delta_{2}}$ takes its values in $\left[  0,1\right]  $.
Thanks to Remark \ref{05041}, we have
\begin{align*}
&  \int\limits_{\left\{  |v_{n}|\leqq k\right\}  }{{\Phi_{{\delta_{1}}%
,{\delta_{2}}}}A(x,t,\nabla u_{n}).\nabla{\langle T_{k}(v)\rangle}_{\nu}}\\
&  =\int_{Q}{\chi_{\left\{  |v_{n}|\leqq k\right\}  }{\Phi_{{\delta_{1}%
},{\delta_{2}}}}A(x,t,\nabla\left(  T_{k}(v_{n})+h_{n}\right)  ).\nabla\langle
T_{k}(v)\rangle}_{\nu}\\
&  =\int_{Q}{\chi_{|v|\leqq k}{\Phi_{{\delta_{1}},{\delta_{2}}}}F_{k}%
.\nabla{\langle T_{k}(v)\rangle}_{\nu}}+\omega(n)=\omega(n,\nu,{\delta_{1}}).
\end{align*}
Therefore, if we prove that
\begin{equation}
\int_{Q}{\Phi_{\delta_{1},\delta_{2}}A(x,t,\nabla u_{n}).\nabla{T_{k}}({v_{n}%
})}\leqq\omega(n,\delta_{1},\delta_{2}), \label{120511}%
\end{equation}
then we deduce (\ref{12059}). As noticed in \cite{DMOP}, \cite{Pe08}, it is
precisely for this estimate that we need the double cut ${\psi_{{\delta_{1}}%
}^{+}\psi_{{\delta_{2}}}^{+}.}$ To do this, we set, for any $m>k>0,$ and any
$r\in\mathbb{R},$%
\[
{\hat{S}_{k,m}}(r)=\int_{0}^{r}{\left(  {k-{T_{k}}(\tau)}\right)  H{_{m}}%
(\tau)d\tau,}%
\]
where $H{_{m}}$ is defined at (\ref{Hm}). Hence supp${\hat{S}_{k,m}\subset
}\left[  -2m,k\right]  ;$ and ${\hat{S}}^{\prime\prime}{_{k,m}=-\chi}_{\left[
-k,k\right]  }+\frac{2k}{m}{\chi}_{\left[  -2m,-m\right]  }.$ We choose
$(\varphi,S)=({\psi_{{\delta_{1}}}^{+}\psi_{{\delta_{2}}}^{+}},{\hat{S}_{k,m}%
})$ as test functions in (\ref{renor}). From (\ref{xxx}), we can write
\[
{A}_{1}+{A}_{2}-{A}_{3}+{A}_{4}+{A}_{5}+{A}_{6}=0,
\]
where
\begin{align*}
&  {A}_{1}=-\int_{Q}({{{{\psi_{{\delta_{1}}}^{+}\psi_{{\delta_{2}}}^{+}}})}%
}_{t}{\hat{S}_{k,m}({v_{n}}),\quad A}_{2}=\int_{Q}{(k-{T_{k}}({v_{n}}))H{_{m}%
}({v_{n}})A(x,t,\nabla u_{n}).\nabla({\psi_{{\delta_{1}}}^{+}\psi_{{\delta
_{2}}}^{+})},}\\
{A}_{3}  &  =\int_{Q}{\psi_{{\delta_{1}}}^{+}\psi_{{\delta_{2}}}%
^{+}A(x,t,\nabla u_{n}).\nabla{T_{k}}({v_{n}}),\quad A}_{4}=\frac{{2k}}{m}%
\int\limits_{\left\{  -2m<{v_{n}}\leqq-m\right\}  }{\psi_{{\delta_{1}}}%
^{+}\psi_{{\delta_{2}}}^{+}A(x,t,\nabla u_{n}).\nabla{v_{n}},}\\
{A}_{5}  &  =-\int_{Q}{(k-{T_{k}}({v_{n}}))H{_{m}}({v_{n}})\psi_{{\delta_{1}}%
}^{+}\psi_{{\delta_{2}}}^{+}}d\widehat{\lambda_{n,0}},\quad{A}_{6}=\int%
_{Q}{(k-{T_{k}}({v_{n}})){H_{m}}({v_{n}})\psi_{{\delta_{1}}}^{+}\psi
_{{\delta_{2}}}^{+}d\left(  {{\eta_{n,0}-\rho_{n,0}}}\right)  ;}%
\end{align*}
and we estimate ${A}_{3}.$ As in \cite[p.585]{Pe08}, since $\left\{  {{\hat
{S}_{k,m}}({v_{n}})}\right\}  $ converges to {{$\hat{S}_{k,m}$}}${({v})}$
weakly in $X,$ and {{$\hat{S}_{k,m}$}}${({v})\in L}^{\infty}(Q),$ and from
(\ref{41}), there holds
\[
{A}_{1}=-\int_{Q}({{{{\psi_{{\delta_{1}}}^{+})}}}}_{t}{{{{\psi_{{\delta_{2}}%
}^{+}}}\hat{S}_{k,m}}({v})-}\int_{Q}{{{{\psi_{{\delta_{1}}}^{+}}}}}%
({{{{\psi_{{\delta_{2}}}^{+})}}}}_{t}{{\hat{S}_{k,m}}({v})+\omega(n)=\omega
}(n,\delta_{1}).
\]

Next consider ${A}_{2}.$ Notice that ${{v_{n}=\;}}T_{2m}(v_{n})$ on
supp$({H{_{m}}({v_{n}})})$. From Proposition \ref{mun}, (iv), the sequence
$\left\{  A(x,t,\nabla\left(  T_{2m}(v_{n})+h_{n}\right)  ).\nabla
(\psi_{\delta_{1}}^{+}\psi_{\delta_{2}}^{+})\right\}  $ converges to
$F_{2m}.\nabla(\psi_{\delta_{1}}^{+}\psi_{\delta_{2}}^{+})$ weakly in
$L^{1}(Q)$. Thanks to Remark \ref{05041} and the convergence of $\psi
_{\delta_{1}}^{+}\psi_{\delta_{2}}^{+}$ in $X$ to $0$ as $\delta_{1}$ tends to
$0$, we find
\[
{A}_{2}=\int_{Q}{(k-{T_{k}}(v)){H_{m}}(v)F_{2m}.\nabla({\psi_{{\delta_{1}}%
}^{+}\psi_{{\delta_{2}}}^{+})}}+\omega(n)=\omega(n,{\delta_{1}}).
\]

Then consider ${A}_{4}.$ Then for some $C=C(p,c_{2}),$
\[
\left\vert {A}_{4}\right\vert \leqq C\frac{{2k}}{m}\int\limits_{\left\{
-2m<{v_{n}}\leqq-m\right\}  }\left(  |\nabla{u_{n}}|^{p}+|\nabla{v_{n}}%
|^{p}+|a|^{p^{\prime}}\right)  \psi_{{\delta_{1}}}^{+}\psi_{{\delta_{2}}}%
^{+}.
\]
Since ${\psi_{{\delta_{1}}}^{+}}$ takes its values in $\left[  0,1\right]  ,$
from Lemma \ref{april261}, we get in particular ${A}_{4}=\omega(n,\delta
_{1},m,\delta_{2})$.

Now estimate $A_{5}.$ The sequence $\left\{  (k-T_{k}(v_{n})){H{_{m}}({v_{n}%
})\psi_{{\delta_{1}}}^{+}\psi_{{\delta_{2}}}^{+}}\right\}  $ converges weakly
in $X$ to $(k-T_{k}(v)){H{_{m}}({v})\psi_{{\delta_{1}}}^{+}\psi_{{\delta_{2}}%
}^{+},}$ and $\left\{  (k-T_{k}(v_{n}))H_{m}(v_{n})\right\}  $ converges
$^{\ast}$-weakly in $L^{\infty}(Q)$ and $a.e.$ in $Q$ to $(k-T_{k}%
(v))H_{m}(v).$ Otherwise $\left\{  f_{n}\right\}  $ converges to $f$ weakly in
$L^{1}\left(  Q\right)  $ and $\left\{  g_{n}\right\}  $ converges to $g$
strongly in $(L^{p^{\prime}}\left(  Q\right)  )^{N}.$ Thanks to Remark
\ref{05041} and the convergence of ${\psi_{{\delta_{1}}}^{+}\psi_{{\delta_{2}%
}}^{+}}$ to $0$ in $X$ and $a.e.$ in $Q$ as $\delta_{1}\rightarrow0$, we
deduce that
\[
{A}_{5}=-\int_{Q}{(k-{T_{k}}({v_{n}}))H{_{m}}({v})\psi_{{\delta_{1}}}^{+}%
\psi_{{\delta_{2}}}^{+}}d\widehat{\nu_{0}}+\omega(n)=\omega(n,\delta_{1}),
\]
where $\widehat{\nu_{0}}=f-\operatorname{div}g.$

Finally ${A}_{6}\leqq2k\int_{Q}{\psi_{{\delta_{1}}}^{+}\psi_{{\delta_{2}}}%
^{+}d{\eta_{n}}}$; using (\ref{12054}) we also find ${A}_{6}$ $\leqq
\omega(n,\delta_{1},m,\delta_{2}).$ By addition, since ${A}_{3}$ does not
depend on $m,$ we obtain
\[
{A}_{3}=\int_{Q}{\psi_{{\delta_{1}}}^{+}\psi_{{\delta_{2}}}^{+}A(x,t,\nabla
u_{n})\nabla{T_{k}}({v_{n}})}\leqq\omega(n,{\delta_{1}},{\delta_{2}}).
\]
Reasoning as before with $({\psi_{{\delta_{1}}}^{-}\psi_{{\delta_{2}}}^{-}%
},{\check{S}_{k,m}})$ as test function in (\ref{renor}), where ${\check
{S}_{k,m}(r)=-}${$\hat{S}_{k,m}$}$(-r),$ we get in the same way
\[
\int_{Q}{\psi_{{\delta_{1}}}^{-}\psi_{{\delta_{2}}}^{-}A(x,t,\nabla
u_{n})\nabla{T_{k}}({v_{n}})}\leqq\omega(n,{\delta_{1}},{\delta_{2}}).
\]
Then, (\ref{120511}) holds.\medskip
\end{proof}

Next we look at the behaviour far from $E.$

\begin{lemma}
\label{far}. Estimate (\ref{120510}) holds.
\end{lemma}

\begin{proof}
Here we estimate $I_{2};$ we can write
\[
I_{2}=\int\limits_{\left\{  |v_{n}|\leqq k\right\}  }{(1-\Phi_{\delta
_{1},\delta_{2}})A(x,t,\nabla u_{n})\nabla\left(  {{T_{k}}({v_{n}})-}\langle
T_{k}(v)\rangle_{\nu}\right)  .}%
\]

\noindent Following the ideas of \cite{Por99}, used also in \cite{Pe08}, we
define, for any $r\in\mathbb{R}$ and $\ell>2k>0$,
\[
{R_{n,\nu,\ell}}={T_{\ell+k}}\left(  {{v_{n}}-}\langle T_{k}(v)\rangle_{\nu
}\right)  -{T_{\ell-k}}\left(  {{v_{n}}-{T_{k}}\left(  {{v_{n}}}\right)
}\right)  .
\]
Recall that $\left\Vert \langle T_{k}(v)\rangle_{\nu}\right\Vert _{\infty
,Q}\leqq k,$ and observe that
\begin{equation}
{R_{n,\nu,\ell}}=2k\;\mathrm{sign}({v_{n}})\quad\text{in}\;\left\{
{\left\vert {{v_{n}}}\right\vert }\geqq{\ell+2k}\right\}  ,\quad\text{
}|R_{n,\nu,\ell}|\leqq4k,\quad R_{n,\nu,\ell}=\omega(n,\nu,\ell)\text{
}a.e.\text{ in }Q, \label{13057}%
\end{equation}%
\begin{equation}
\lim_{n\rightarrow\infty}R_{n,\nu,\ell}={T_{\ell+k}}\left(  {{v}-\langle
T_{k}(v)\rangle}_{\nu}\right)  -{T_{\ell-k}}\left(  {{v}-{T_{k}}\left(  {{v}%
}\right)  }\right)  ,\qquad a.e.\;\text{in}\;{Q},\text{ and weakly in }X.
\label{13058}%
\end{equation}
Next consider $\xi_{1,n_{1}}\in C_{c}^{\infty}([0,T)),\xi_{2,n_{2}}\in
C_{c}^{\infty}((0,T])$ with values in $[0,1],$ such that $(\xi_{1,n_{1}}%
)_{t}\leqq0$ and $(\xi_{2,n_{2}})_{t}$ $\geqq0$; and $\left\{  \xi_{1,n_{1}%
}(t)\right\}  $ (resp. $\left\{  \xi_{1,n_{2}}(t)\right\}  )$ converges to
$1,\,$for any $t\in\lbrack0,T)$ (resp. $t\in(0,T]$ ); and moreover, for any
$a\in C([0,T];L^{1}(\Omega))$, $\left\{  \int_{Q}a{{{\left(  \xi_{1,n_{1}%
}\right)  }_{t}}}\right\}  $ and $\int_{Q}a{{{\left(  \xi_{2,n_{2}}\right)
}_{t}}}$ converge respectively to $-\int\limits_{\Omega}{a(T,.)}$ and
$\int\limits_{\Omega}{a(0,.).}$ We set
\[
\varphi={\varphi_{n,n_{1},n_{2},{l_{1}},{l_{2},\ell}}}=\xi_{1,n_{1}}%
(1-{\Phi_{\delta_{1},\delta_{2}}}){\left[  {{T_{\ell+k}}\left(  {{v_{n}}%
-}\langle T_{k}(v)\rangle_{\nu}\right)  }\right]  _{{l_{1}}}}-\xi_{2,n_{2}%
}(1-{\Phi_{\delta_{1},\delta_{2}}}){\left[  {{T_{\ell-k}}\left(
{v_{n}-{{{T_{k}}(v_{n})}}}\right)  }\right]  _{{-l_{2}}}.}%
\]
We can see that
\begin{equation}
{\varphi-(1-{\Phi_{\delta_{1},\delta_{2}}}){R_{n,\nu,\ell}}}=\omega
(l_{1},l_{2},n_{1},n_{2})\;\quad\text{ in norm in }X\text{ and }%
a.e.\;\text{in}\;{Q}. \label{13056}%
\end{equation}
We can choose $(\varphi,S)=({\varphi_{n,n_{1},n_{2},{l_{1}},{l_{2},\ell}}%
},\overline{H_{m}})$ as test functions in (\ref{renor4}) for $u_{n}$, where
$\overline{H_{m}}$ is defined at (\ref{Hm}), with $m>\ell+2k.$ We obtain
\[
A_{1}+A_{2}+A_{3}+A_{4}+A_{5}=A_{6}+A_{7},
\]
with
\begin{align*}
A_{1}  &  =\int\limits_{\Omega}{\varphi(T){\overline{H_{m}}}({v_{n}%
(T)})dx,\quad\quad}A_{2}=-\int\limits_{\Omega}{\varphi(0){\overline{H_{m}}%
}({u_{0,n}})dx,}\\
A_{3}  &  =-\int_{Q}{\varphi{_{t}\overline{H_{m}}}({v_{n}}),\quad\quad}%
A_{4}=\int_{Q}H_{m}{({v_{n}})A(x,t,\nabla u_{n}).\nabla\varphi},\\
A_{5}  &  =\int_{Q}{{\varphi}}H^{\prime}{{_{m}}({v_{n}})A(x,t,\nabla
u_{n}).\nabla v_{n}{,\quad\quad}}A_{6}=\int_{Q}H_{m}{({v_{n}}){\varphi}%
d}\widehat{\lambda_{n,0}}{{,}}\\
A_{7}  &  =\int_{Q}H_{m}{({v_{n}}){\varphi}d\left(  {{\rho_{n,0}}-{\eta_{n,0}%
}}\right)  .}%
\end{align*}
\newline\textbf{Estimate} of $A_{4}$. This term allows to study $I_{2}.$
Indeed, $\left\{  H_{m}(v_{n})\right\}  $ converges to $1,$ $a.e.$ in $Q$;
thanks to (\ref{13056}), (\ref{13057}) (\ref{13058}), we have
\begin{align*}
A_{4}  &  =\int_{Q}{(1-{\Phi_{\delta_{1},\delta_{2}}})A(x,t,\nabla
u_{n}).\nabla{R_{n,\nu,\ell}}}-\int_{Q}{{R_{n,\nu,\ell}}A(x,t,\nabla
u_{n}).\nabla{\Phi_{\delta_{1},\delta_{2}}+}}\omega(l_{1},l_{2},n_{1}%
,n_{2},m)\\
&  =\int_{Q}{(1-{\Phi_{\delta_{1},\delta_{2}}})A(x,t,\nabla u_{n}%
).\nabla{R_{n,\nu,\ell}+}}\omega(l_{1},l_{2},n_{1},n_{2},m,n,\nu,\ell)\\
&  =I_{2}+\int\limits_{\left\{  \left\vert {{v_{n}}}\right\vert >k\right\}
}{(1-{\Phi_{\delta_{1},\delta_{2}}})A(x,t,\nabla u_{n}).\nabla{R_{n,\nu,\ell
}+}}\omega(l_{1},l_{2},n_{1},n_{2},m,n,\nu,\ell)\\
&  =I_{2}+B_{1}+B_{2}+\omega(l_{1},l_{2},n_{1},n_{2},m,n,\nu,\ell),
\end{align*}
where
\begin{align*}
B_{1}  &  =\int\limits_{\left\{  \left\vert {{v_{n}}}\right\vert >k\right\}
}{(1-{\Phi_{\delta,\eta}})({{\chi_{\left\vert {{v_{n}}-\langle T_{k}%
(v)\rangle}_{\nu}\right\vert \leqq\ell+k}}-{\chi_{\left\vert \left\vert
{{v_{n}}}\right\vert {-k}\right\vert \leqq\ell-k})}}A(x,t,\nabla u_{n}).\nabla
v_{n},}\\
B_{2}  &  =-\int\limits_{\left\{  \left\vert {{v_{n}}}\right\vert >k\right\}
}(1-{\Phi_{\delta_{1},\delta_{2}}}){\chi_{\left\vert {{v_{n}}-\langle
T_{k}(v)\rangle}_{\nu}\right\vert \leqq\ell+k}}A(x,t,\nabla u_{n}%
).\nabla\langle{{{{{T_{k}}(v)\rangle}}_{\nu}.}}%
\end{align*}
Now $\left\{  A(x,t,\nabla\left(  {{T_{\ell+2k}}({v_{n}})+{h_{n}}}\right)
).\nabla{\langle T_{k}(v)\rangle}_{\nu}\right\}  $ converges to $F_{\ell
+2k}\nabla{\langle T_{k}(v)\rangle}_{\nu},$ weakly in $L^{1}(Q).$ Otherwise
$\left\{  \chi_{|v_{n}|>k}{\chi_{\left\vert {{v_{n}}-\langle T_{k}(v)\rangle
}_{\nu}\right\vert \leqq\ell+k}}\right\}  $ converges to $\chi_{|v|>k}%
{\chi_{\left\vert {{v}-\langle T_{k}(v)\rangle}_{\nu}\right\vert \leqq\ell
+k},}$ $a.e.$ in $Q$. And $\left\{  \langle T_{k}(v)\rangle_{\nu}\right\}  $
converges to $T_{k}(v)$ strongly in $X$. Thanks to Remark \ref{05041} we get%
\begin{align*}
B_{2}  &  =-\int_{Q}{(1-{\Phi_{{\delta_{1}},{\delta_{2}}}})\;{\chi_{|v|>k}%
\;}{\chi_{\left\vert {v-\langle T_{k}(v)\rangle}_{\nu}\right\vert \leqq\ell
+k}}F_{\ell+2k}.\nabla{\langle T_{k}(v)\rangle}_{\nu}}+\omega(n)\\
&  =-\int_{Q}{(1-{\Phi_{{\delta_{1}},{\delta_{2}}}}){{{\;}}}{\chi_{|v|>k}%
\;}{\chi_{\left\vert {v-{{{{T_{k}}(v)}}}}\right\vert \leqq\ell+k}}F_{\ell
+2k}.\nabla{{{{T_{k}}(v)}}}}+\omega(n,\nu)=\omega(n,\nu),
\end{align*}
since ${\nabla{{{{T_{k}}(v)\;}}}{\chi_{|v|>k}=0.}}$ Besides, we see that, for
some $C=C(p,c_{2}),$%
\[
\left\vert B_{1}\right\vert \leqq C\int\limits_{\left\{  \ell-2k\leqq
\left\vert {{v_{n}}}\right\vert <\ell+2k\right\}  }{(1-{\Phi_{\delta
_{1},\delta_{2}}})\left(  |\nabla u_{n}|^{p}+|\nabla v_{n}|^{p}+|a|^{p^{\prime
}}\right)  }.
\]
Using (\ref{12056}) and (\ref{12057}) and applying (\ref{13053}) and
(\ref{13054}) to ${1-{\Phi_{\delta_{1},\delta_{2}}}}$, we obtain, for $k>0$
\begin{equation}
\int\limits_{\left\{  m\leqq|{v_{n}}|<m+4k\right\}  }({{{\left\vert
{\nabla{u_{n}}}\right\vert }^{p}+{\left\vert {\nabla{v_{n}}}\right\vert }%
^{p})}(1-{\Phi_{\delta_{1},\delta_{2}}})}=\omega(n,m,\delta_{1},\delta_{2}).
\label{04043}%
\end{equation}
Thus, $B_{1}=\omega(n,\nu,\ell,\delta_{1},\delta_{2}),$ hence $B_{1}%
+B_{2}=\omega(n,\nu,\ell,\delta_{1},\delta_{2}).$ Then
\begin{equation}
A_{4}=I_{2}+\omega(l_{1},l_{2},n_{1},n_{2},m,n,\nu,\ell,\delta_{1},\delta
_{2}). \label{a4}%
\end{equation}
\textbf{Estimate} of $A_{5}$. For $m>\ell+2k$, since $|\varphi|\leqq2\ell,$
and (\ref{13056}) holds, we get, from the dominated convergence Theorem,
\begin{align*}
A_{5}  &  =\int_{Q}(1-{\Phi_{\delta_{1},\delta_{2}}})R_{n,\nu,\ell}%
{{H_{m}^{^{\prime}}}}(v_{n})A(x,t,\nabla u_{n}).\nabla v_{n}+\omega
(l_{1},l_{2},n_{1},n_{2})\\
&  =-\frac{{2k}}{m}\int\limits_{\left\{  m\leqq\left\vert {{v_{n}}}\right\vert
<2m\right\}  }{(1-{\Phi_{\delta_{1},\delta_{2}}})A(x,t,\nabla u_{n}).\nabla
v_{n}+}\omega(l_{1},l_{2},n_{1},n_{2});
\end{align*}
here, the final equality followed from the relation, since $m>\ell+2k,$
\begin{equation}
R_{n,\nu,\ell}{{H_{m}^{^{\prime}}}}(v_{n})=-\frac{2k}{m}\chi_{m\leqq
|v_{n}|\leqq2m},\quad a.e.\text{ in }Q. \label{relt}%
\end{equation}
Next we go to the limit in $m,$ by using (\ref{renor2}), (\ref{renor3}) for
$u_{n}$, with ${\phi=(1-{\Phi_{\delta_{1},\delta_{2}}})}$. There holds
\[
A_{5}=-2k\int_{Q}{(1-{\Phi_{\delta_{1},\delta_{2}}})d\left(  (\rho_{n,s}%
-\eta_{n,s})^{+}+(\rho_{n,s}-\eta_{n,s})^{-}\right)  +}\omega(l_{1}%
,l_{2},n_{1},n_{2},m).
\]
Then, from (\ref{12056}) and (\ref{12057}), we get $A_{5}=\omega(l_{1}%
,l_{2},n_{1},n_{2},m,n,\nu,\ell,\delta_{1},\delta_{2}).$ \medskip

\noindent\textbf{Estimate} of $A_{6}$. Again, from (\ref{13056}),
\begin{align*}
A_{6}  &  =\int_{Q}H_{m}{({v_{n}}){\varphi f}}_{n}+\int_{Q}g_{n}.\nabla
(H_{m}{({v_{n}}){\varphi)}}\\
&  =\int_{Q}H_{m}{({v_{n}})(1-{\Phi_{\delta_{1},\delta_{2}}}){R_{n,\nu,\ell}%
f}}_{n}+\int_{Q}g_{n}.\nabla(H_{m}{({v_{n}})(1-{\Phi_{\delta_{1},\delta_{2}}%
}){R_{n,\nu,\ell})+}}\omega(l_{1},l_{2},n_{1},n_{2}).
\end{align*}
Thus we can write ${A_{6}}={D}_{1}+{D}_{2}+{D}_{3}+{D}_{4}+\omega(l_{1}%
,l_{2},n_{1},n_{2}),$ where
\begin{align*}
{D}_{1}  &  =\int_{Q}H{{_{m}}({v_{n}})(1-{\Phi_{\delta_{1},\delta_{2}}%
}){R_{n,\nu,\ell}}{f_{n},\qquad}D}_{2}=\int_{Q}{(1-{\Phi_{\delta_{1}%
,\delta_{2}}}){R_{n,\nu,\ell}H}}_{m}^{\prime}{({v_{n}}){g_{n}.}\nabla{v_{n},}%
}\\
&  {D}_{3}=\int_{Q}H{{_{m}}({v_{n}})(1-{\Phi_{\delta_{1},\delta_{2}}}){g_{n}%
}.\nabla{R_{n,\nu,\ell},\qquad}D}_{4}=-\int_{Q}H{{_{m}}({v_{n}}){R_{n,\nu
,\ell}g_{n}}.\nabla}{\Phi_{\delta_{1},\delta_{2}}.}%
\end{align*}
Since $\left\{  f_{n}\right\}  $ converges to $f$ weakly in $L^{1}(Q)$, and
(\ref{13057})-(\ref{13058}) hold, we get from Remark \ref{05041},
\[
{D}_{1}=\int_{Q}{(1-{\Phi_{\delta_{1},\delta_{2}}})\left(  {{T_{\ell+k}%
}\left(  {v-}\langle{{{{{T_{k}}(v)\rangle}}_{\nu}}}\right)  -{T_{\ell-k}%
}\left(  {v-{T_{k}}\left(  v\right)  }\right)  }\right)  f+}\omega
(m,n)=\omega(m,n,\nu,\ell).
\]
We deduce from (\ref{lim3}) that ${D}_{2}=\omega(m)$. Next consider $D_{3}.$
Note that $H_{m}{({v_{n}})=1+\omega(m),}$ and (\ref{13058}) holds, and
$\left\{  g_{n}\right\}  $ converges to $g$ strongly in ($L^{p^{\prime}%
}(Q))^{N},$ and $\langle T_{k}(v)\rangle_{\nu}$ converges to $T_{k}(v)$
strongly in $X.$ Then we obtain successively that
\begin{align*}
{D}_{3}  &  =\int_{Q}{(1-{\Phi_{\delta_{1},\delta_{2}}})g.\nabla\left(
{{T_{\ell+k}}\left(  {v-\langle T_{k}(v)\rangle}_{\nu}\right)  -{T_{\ell-k}%
}\left(  {v-{T_{k}}\left(  v\right)  }\right)  }\right)  +}\omega(m,n)\\
&  =\int_{Q}{(1-{\Phi_{\delta_{1},\delta_{2}}})g.\nabla\left(  {{T_{\ell+k}%
}\left(  {v-{T_{k}}(v)}\right)  -{T_{\ell-k}}\left(  {v-{T_{k}}\left(
v\right)  }\right)  }\right)  +}\omega(m,n,\nu)\\
&  =\omega(m,n,\nu,\ell).
\end{align*}
Similarly we also get $D_{4}=\omega(m,n,\nu,\ell)$. Thus ${A_{6}}=\omega
(l_{1},l_{2},n_{1},n_{2},m,n,\nu,\ell,\delta_{1},\delta_{2}).\medskip$

\noindent\textbf{Estimate} of $A_{7}$. We have
\begin{align*}
\left\vert A_{7}\right\vert  &  =\left\vert \int_{Q}{{{{S^{\prime}}_{m}%
}({v_{n}})\left(  {1-{\Phi_{{\delta_{1}},{\delta_{2}}}}}\right)
{R_{n,\nu,\ell}}d\left(  {{\rho_{n,0}}-{\eta_{n,0}}}\right)  }}\right\vert
+\omega({l_{1}},{l_{2}},{n_{1}},{n_{2}})\\
&  \leqq4k\int_{Q}{\left(  {1-{\Phi_{{\delta_{1}},{\delta_{2}}}}}\right)
d\left(  \rho_{n}+\eta_{n}\right)  }+\omega({l_{1}},{l_{2}},{n_{1}},{n_{2}}).
\end{align*}
From (\ref{12056}) and (\ref{12057}) we get $A_{7}=\omega(l_{1},l_{2}%
,n_{1},n_{2},m,n,\nu,\ell,\delta_{1},\delta_{2}).\medskip$

\noindent\textbf{Estimate} of $A_{1}+A_{2}+A_{3}$. We set
\[
J(r)={T_{\ell-k}}\left(  r{-{T_{k}}\left(  r\right)  }\right)  ,\qquad\forall
r\in\mathbb{R},
\]
and use the notations $\overline{J}{\ }${and}$\mathcal{J}$ of (\ref{lam}).
From the definitions of $\xi_{1,n_{1}},\xi_{1,n_{2}},$ we can see that
\begin{align}
A_{1}+A_{2}  &  =-\int_{\Omega}J(v_{n}(T)){{\overline{H_{m}}}({v_{n}(T)}%
)}-\int_{\Omega}T_{\ell+k}(u_{0,n}-z_{\nu}){{\overline{H_{m}}}(}%
u_{0,n})+\omega(l_{1},l_{2},n_{1},n_{2})\nonumber\\
&  =-\int_{\Omega}J(v_{n}(T))v_{n}(T)-\int_{\Omega}T_{\ell+k}(u_{0,n}-z_{\nu
})u_{0,n}+\omega(l_{1},l_{2},n_{1},n_{2},m), \label{a1a2}%
\end{align}
where $z_{\nu}=\langle T_{k}(v)\rangle_{\nu}(0).$ We can write $A_{3}%
=F_{1}+F_{2},$ where
\begin{align*}
\text{ }F_{1}  &  =-\int_{Q}{{{\left(  {{\xi_{n_{1}}}(1-{\Phi_{\delta
_{1},\delta_{2}}}){{\left[  {{T_{\ell+k}}\left(  {v_{n}-\langle T_{k}%
(v)\rangle}_{\nu}\right)  }\right]  }_{{l_{1}}}}}\right)  }_{t}\overline
{H_{m}}}({v_{n}}),}\\
F_{2}  &  =\int_{Q}{{{\left(  {{\xi_{n_{2}}}(1-{\Phi_{\delta_{1},\delta_{2}}%
}){{\left[  {{T_{\ell-k}}\left(  {v_{n}-{T_{k}}\left(  {v_{n})}\right)
}\right)  }\right]  }_{{-l_{2}}}}}\right)  }_{t}\overline{H_{m}}}({v_{n}}).}%
\end{align*}
\textbf{Estimate} of $F_{2}$. We write $F_{2}=G_{1}+G_{2}+G_{3},$ with
\begin{align*}
G_{1}  &  =-\int_{Q}{{{\left(  {{\Phi_{\delta_{1},\delta_{2}}}}\right)  }_{t}%
}{\xi_{n_{2}}}{{\left[  {{T_{\ell-k}}\left(  {v_{n}-{T_{k}}\left(
v_{n}\right)  }\right)  }\right]  }_{{-l_{2}}}\overline{H_{m}}}({v_{n}}),}\\
G_{2}  &  =\int_{Q}{(1-{\Phi_{\delta_{1},\delta_{2}}}){{\left(  {{\xi_{n_{2}}%
}}\right)  }_{t}}{{\left[  {{T_{\ell-k}}\left(  {v_{n}-{T_{k}}\left(
v_{n}\right)  }\right)  }\right]  }_{{-l_{2}}}}\overline{H_{m}}(v_{n}),}\\
G_{3}  &  =\int_{Q}{{\xi_{n_{2}}}(1-{\Phi_{\delta_{1},\delta_{2}}}){{\left(
{{{\left[  {{T_{\ell-k}}\left(  {v_{n}-{T_{k}}\left(  {v_{n}}\right)
}\right)  }\right]  }_{{-l_{2}}}}}\right)  }_{t}}\overline{H_{m}}(v_{n}).}%
\end{align*}
We find easily
\[
{G}_{1}=-\int_{Q}{{{\left(  {{\Phi_{\delta_{1},\delta_{2}}}}\right)  }_{t}%
J}(v_{n})v_{n}+}\omega(l_{1},l_{2},n_{1},n_{2},m),
\]%
\[
{G}_{2}=\int_{Q}{(1-{\Phi_{\delta_{1},\delta_{2}}}){{\left(  {{\xi_{n_{2}}}%
}\right)  }_{t}}}J(v_{n}){{\overline{H_{m}}}({v_{n}})+}\omega({l_{1},l_{2}%
})=\int\limits_{\Omega}J{(u_{0,n})u_{0,n}+}\omega(l_{1},l_{2},n_{1},n_{2},m).
\]
Next consider $G_{3}.$ Setting $b={{\overline{H_{m}}}({v_{n}}){,}}$ there
holds from (\ref{tkp}) and (\ref{222}),
\[
((\left[  J{{(b)}}\right]  _{-{l_{2}}})_{t}b)(.,t)=\frac{{{b(.,t)}}}{l_{2}%
}(J{{(b)(.,t)-}}J{{{{(b)(.,t-l}}_{2}{{)}}).}}%
\]
Hence
\[
{\left(  {{{\left[  {{T_{\ell-k}}\left(  {{v_{n}}-{T_{k}}\left(  {{v_{n}}%
}\right)  }\right)  }\right]  }_{-{l_{2}}}}}\right)  _{t}\overline{H_{m}}%
}({v_{n}})\geqq{\left(  {{{\left[  \mathcal{J}{({\overline{H_{m}}}({v_{n}}%
))}\right]  }_{-{l_{2}}}}}\right)  _{t}=\left(  {{{\left[  \mathcal{J}%
{({v_{n}})}\right]  }_{-{l_{2}}}}}\right)  _{t},}%
\]
since $\mathcal{J}$ is constant in $\left\{  \left\vert r\right\vert \geqq
m+\ell+2k\right\}  .$ Integrating by parts in $G_{3},$ we find
\begin{align*}
G_{3}  &  \geqq\int_{Q}{{\xi_{2,n_{2}}}(1-{\Phi_{\delta_{1},\delta_{2}}%
}){{\left(  {{{\left[  \mathcal{J}{({v_{n}})}\right]  }_{{-l_{2}}}}}\right)
}_{t}}}\\
&  =-\int_{Q}{{{\left(  {{\xi_{2,n_{2}}}(1-{\Phi_{\delta_{1},\delta_{2}}}%
)}\right)  }_{t}}{{\left[  \mathcal{J}{({v_{n}})}\right]  }_{{-l_{2}}}}}%
+\int\limits_{\Omega}{{\xi_{2,n_{2}}}}(T){{{\left[  \mathcal{J}{({v_{n}}%
)}\right]  }_{{-l_{2}}}}(T)}\\
&  =-\int_{Q}{{{\left(  {{\xi_{2,n_{2}}}}\right)  }_{t}}(1-{\Phi_{\delta
_{1},\delta_{2}}})}\mathcal{J}{({v_{n}})}\\
&  +\int_{Q}{{\xi_{2,n_{2}}{\left(  {{\Phi_{\delta_{1},\delta_{2}}}}\right)
}_{t}}}\mathcal{J}{({v_{n}})}+\int\limits_{\Omega}{{\xi_{2,n_{2}}}%
}(T)\mathcal{J}{({v_{n}}(T))+}\omega({l_{1},l_{2}})\\
&  =-\int\limits_{\Omega}\mathcal{J}{({u_{0,n}})+\int_{Q}{{{\left(
{{\Phi_{\delta_{1},\delta_{2}}}}\right)  }_{t}}\mathcal{J}{({v_{n}})}+}%
\int\limits_{\Omega}\mathcal{J}{({v_{n}}(T))}+}\omega(l_{1},l_{2},n_{1}%
,n_{2}).
\end{align*}
Therefore, since $\mathcal{J}({{v_{n}}})-J({{v_{n}}}){{v_{n}}}=-{\overline
{J}({v_{n}})}$ and ${\overline{J}(u_{0,n})=}J{(u_{0,n})u_{0,n}-}%
\mathcal{J}{(u_{0,n}),}$ we obtain
\begin{equation}
{F}_{2}\geqq\int\limits_{\Omega}{\overline{J}(u_{0,n})}\text{ }-\int%
_{Q}{{{\left(  {{\Phi_{\delta_{1},\delta_{2}}}}\right)  }_{t}}{\overline{J}%
}({v_{n}})}+\int\limits_{\Omega}\mathcal{J}{(v_{n}(T))+}\omega(l_{1}%
,l_{2},n_{1},n_{2},m). \label{f2}%
\end{equation}
\textbf{Estimate }of $F_{1}.$ Since $m>\ell+2k,$ there holds ${{T_{\ell+k}%
}\left(  {{v_{n}}-}\langle{{{{{T_{k}}(v)\rangle}}_{\nu}}}\right)  ={T_{\ell
+k}}\left(  {{\overline{H_{m}}}({v_{n}})-}\langle{{{{{T_{k}}({\overline{H_{m}%
}}(v))\rangle}}_{\nu}}}\right)  }$ on supp${{\overline{H_{m}}}({v_{n}}).}$
Hence we can write $F_{1}=L_{1}+L_{2},$ with
\begin{align*}
L_{1}  &  =-\int_{Q}{{{\left(  {{\xi_{1,n_{1}}}(1-{\Phi_{\delta_{1},\delta
_{2}}}){{\left[  {{T_{\ell+k}}\left(  {{\overline{H_{m}}}({v_{n}})-}%
\langle{{{{{T_{k}}({\overline{H_{m}}}(v))\rangle}}_{\nu}}}\right)  }\right]
}_{{l_{1}}}}}\right)  }_{t}}\left(  {{\overline{H_{m}}}({v_{n}})-}%
\langle{{{{{T_{k}}({\overline{H_{m}}}(v)\rangle}}_{\nu}}}\right)  }\\
L_{2}  &  =-\int_{Q}{{{\left(  {{\xi_{1,n_{1}}}(1-{\Phi_{\delta_{1},\delta
_{2}}}){{\left[  {{T_{\ell+k}}\left(  {{\overline{H_{m}}}({v_{n}})-}%
\langle{{{{{T_{k}}({\overline{H_{m}}}(v))\rangle}}_{\nu}}}\right)  }\right]
}_{{l_{1}}}}}\right)  }_{t}}}\langle{{{{{T_{k}}({\overline{H_{m}}}(v))\rangle
}}_{\nu}.}}%
\end{align*}
Integrating by parts we have, by definition of the Landes-time approximation,
\begin{align}
L_{2}  &  =\int_{Q}{{\xi_{1,n_{1}}}(1-{\Phi_{\delta_{1},\delta_{2}}}){{\left[
{{T_{\ell+k}}\left(  {{\overline{H_{m}}}({v_{n}})-}\langle{{{{{T_{k}%
}({\overline{H_{m}}}(v))\rangle}}_{\nu}}}\right)  }\right]  }_{{l_{1}}}%
}{{\left(  \langle{{{{{T_{k}}({\overline{H_{m}}}(v))\rangle}}_{\nu}}}\right)
}_{t}}}\nonumber\\
&  +\int_{\Omega}{{\xi_{1,n_{1}}}}(0){{{\left[  {{T_{\ell+k}}\left(
{{\overline{H_{m}}}({v_{n}})-}\langle{{{{{T_{k}}({\overline{H_{m}}}%
(v))\rangle}}_{\nu}}}\right)  }\right]  }_{{l_{1}}}}(0)\langle{{{{{T_{k}%
}({\overline{H_{m}}}(v))\rangle}}_{\nu}}}(0)}\nonumber\\
&  =\nu\int_{Q}{(1-{\Phi_{\delta_{1},\delta_{2}}}){T_{\ell+k}}\left(  {v_{n}%
-}\langle{{{{{T_{k}}(v)\rangle}}_{\nu}}}\right)  \left(  {{T_{k}}(v)-}%
\langle{{{{{T_{k}}(v)\rangle}}_{\nu}}}\right)  }+\int\limits_{\Omega}%
{{T_{\ell+k}}\left(  {{u_{0,n}}-{z_{\nu}}}\right)  {z_{\nu}+}}\omega
(l_{1},l_{2},n_{1},n_{2}). \label{a2}%
\end{align}
We decompose $L_{1}$ into $L_{1}=K_{1}+K_{2}+K_{3},$ where
\begin{align*}
K_{1}  &  =-\int_{Q}{{{\left(  {{\xi_{1,n_{1}}}}\right)  }_{t}}(1-{\Phi
_{\delta_{1},\delta_{2}}}){{\left[  {{T_{\ell+k}}\left(  {{\overline{H_{m}}%
}({v_{n}})-}\langle{{{{{T_{k}}({\overline{H_{m}}}(v))\rangle}}_{\nu}}}\right)
}\right]  }_{{l_{1}}}}\left(  {{\overline{H_{m}}}({v_{n}})-}\langle{{{{{T_{k}%
}({\overline{H_{m}}}(v))\rangle}}_{\nu}}}\right)  }\\
K_{2}  &  =\int_{Q}{{\xi_{1,n_{1}}{\left(  {{\Phi_{\delta_{1},\delta_{2}}}%
}\right)  }_{t}}{{\left[  {{T_{\ell+k}}\left(  {{\overline{H_{m}}}({v_{n}}%
)-}\langle{{{{{T_{k}}({\overline{H_{m}}}(v))\rangle}}_{\nu}}}\right)
}\right]  }_{{l_{1}}}}\left(  {{\overline{H_{m}}}({v_{n}})-}\langle{{{{{T_{k}%
}({\overline{H_{m}}}(v))\rangle}}_{\nu}}}\right)  }\\
K_{3}  &  =-\int_{Q}{{\xi_{1,n_{1}}}(1-{\Phi_{\delta_{1},\delta_{2}}%
}){{\left(  {{{\left[  {{T_{\ell+k}}\left(  {{\overline{H_{m}}}({v_{n}}%
)-}\langle{{{{{T_{k}}({\overline{H_{m}}}(v))\rangle}}_{\nu}}}\right)
}\right]  }_{{l_{1}}}}}\right)  }_{t}}\left(  {{\overline{H_{m}}}({v_{n}}%
)-}\langle{{{{{T_{k}}({\overline{H_{m}}}(v)\rangle}}_{\nu}}}\right)  .}%
\end{align*}
Then we check easily that
\[
K_{1}=\int\limits_{\Omega}{{T_{\ell+k}}\left(  {{v_{n}}-}\langle{{{{{T_{k}%
}(v)\rangle}}_{\nu}}}\right)  (T)\left(  {{v_{n}}-}\langle{{{{{T_{k}%
}(v)\rangle}}_{\nu}}}\right)  (T)dx+}\omega(l_{1},l_{2},n_{1},n_{2},m),
\]%
\[
K_{2}=\int_{Q}{{{\left(  {{\Phi_{\delta_{1},\delta_{2}}}}\right)  }_{t}%
}{T_{\ell+k}}\left(  {{v_{n}}-}\langle{{{{{T_{k}}(v)\rangle}}_{\nu}}}\right)
\left(  {{v_{n}}-}\langle{{{{{T_{k}}(v)\rangle}}_{\nu}}}\right)  +}%
\omega(l_{1},l_{2},n_{1},n_{2},m).
\]
Next consider $K_{3}.$ Here we use the function $\mathcal{T}_{k}$ defined at
(\ref{tkp}). We set $b={{\overline{H_{m}}}({v_{n}})-}\langle{{{{{T_{k}%
}({\overline{H_{m}}}(v))\rangle}}_{\nu}.}}$ Hence from (\ref{222}),
\begin{align*}
((\left[  {{T_{\ell+k}(b)}}\right]  _{{l_{1}}})_{t}b)(.,t)  &  =\frac
{{{b(.,t)}}}{l_{1}}({{T_{\ell+k}(b)(.,t+l}}_{1}{{)-{{T_{\ell+k}(b)(.,t)}})}}\\
{}  &  {{\leqq}}\frac{1}{l_{1}}(\mathcal{T}_{\ell+k}(b)({{(.,t+l}}_{1}{{))}%
}-\mathcal{T}_{\ell+k}(b)(.,t))=(\left[  \mathcal{T}_{\ell+k}(b)\right]
_{l_{1}})_{t}.
\end{align*}
Thus
\begin{align*}
{\left(  {{{\left[  {{T_{\ell+k}}\left(  {{\overline{H_{m}}}({v_{n}})-}%
\langle{{{{{T_{k}}({\overline{H_{m}}}(v))\rangle}}_{\nu}}}\right)  }\right]
}_{{l_{1}}}}}\right)  _{t}}\left(  {{\overline{H_{m}}}({v_{n}})-}%
\langle{{{{{T_{k}}({\overline{H_{m}}}(v))\rangle}}_{\nu}}}\right)   &
\leqq{\left(  {{{\left[  \mathcal{T}_{\ell+k}{\left(  {{\overline{H_{m}}%
}({v_{n}})-}\langle{{{{{T_{k}}({\overline{H_{m}}}(v))\rangle}}_{\nu}}}\right)
}\right]  }_{{l_{1}}}}}\right)  _{t}}\\
&  ={\left(  {{{\left[  \mathcal{T}_{\ell+k}{{({v_{n}}-}\langle{{{{{T_{k}%
}(v)\rangle}}_{\nu}}}}\right]  }_{{l_{1}}}}}\right)  _{t}.}%
\end{align*}
Then
\begin{align*}
{K}_{3}  &  \geqq-\int_{Q}{{\xi_{1,n_{1}}}(1-{\Phi_{\delta_{1},\delta_{2}}%
}){{\left(  {{{\left[  \mathcal{T}{_{\ell+k}\left(  {{v_{n}}-}\langle
{{{{{T_{k}}(v)\rangle}}_{\nu}}}\right)  }\right]  }_{{l_{1}}}}}\right)  }_{t}%
}}\\
&  =\int_{Q}{{{\left(  {{\xi_{1,n_{1}}}}\right)  }_{t}}(1-{\Phi_{\delta
_{1},\delta_{2}}}){{\left[  \mathcal{T}{_{\ell+k}\left(  {{v_{n}}-}%
\langle{{{{{T_{k}}(v)\rangle}}_{\nu}}}\right)  }\right]  }_{{l_{1}}}}}%
-\int_{Q}{{\xi_{1,n_{1}}{\left(  {{\Phi_{\delta_{1},\delta_{2}}}}\right)
}_{t}}{{\left[  \mathcal{T}{_{\ell+k}\left(  {{v_{n}}-}\langle{{{{{T_{k}%
}(v)\rangle}}_{\nu}}}\right)  }\right]  }_{{l_{1}}}}}\\
&  +\int\limits_{\Omega}{{\xi_{1,n_{1}}}}(0){{{\left[  \mathcal{T}{_{\ell
+k}\left(  {{v_{n}}-}\langle{{{{{T_{k}}(v)\rangle}}_{\nu}}}\right)  }\right]
}_{{l_{1}}}}(0)}\\
&  =-\int\limits_{\Omega}\mathcal{T}{_{\ell+k}\left(  {{v_{n}(T)}%
-\langle{{{{{T_{k}}(v)\rangle}}_{\nu}{{(T)}}}}}\right)  }-\int_{Q}{{{\left(
{{\Phi_{\delta_{1},\delta_{2}}}}\right)  }_{t}}}\mathcal{T}{_{\ell+k}\left(
{{v_{n}}-}\langle{{{{{T_{k}}(v)\rangle}}_{\nu}}}\right)  }\\
&  +\int\limits_{\Omega}\mathcal{T}{{_{\ell+k}}\left(  {{u_{0,n}}-{z_{\nu}}%
}\right)  +}\omega({l_{1}},{l_{2}},{n_{1}},{n_{2}}).
\end{align*}
We find by addition, since $T_{\ell+k}(r)-\mathcal{T}{{_{\ell+k}%
(r)={\overline{T}}_{\ell+k}(r)}}$ for any $r\in\mathbb{R},$
\begin{align}
{L}_{1}  &  \geqq\int\limits_{\Omega}\mathcal{T}{{_{\ell+k}}\left(  {{u_{0,n}%
}-{z_{\nu}}}\right)  }+\int\limits_{\Omega}{{{\overline{T}}_{\ell+k}}\left(
{{v_{n}}(T)-\langle{{{{{T_{k}}(v)\rangle}}_{\nu}}}(T)}\right)  }\nonumber\\
&  +\int_{Q}{{{\left(  {{\Phi_{{\delta_{1}},{\delta_{2}}}}}\right)  }_{t}%
}{{\overline{T}}_{\ell+k}}\left(  {{v_{n}}-}\langle{{{{{T_{k}}(v)\rangle}%
}_{\nu}}}\right)  +}\omega({l_{1}},{l_{2}},{n_{1}},{n_{2}},m). \label{a1}%
\end{align}
We deduce from (\ref{a1}), (\ref{a2}), (\ref{f2}),
\begin{align}
{A}_{3}  &  \geqq\int\limits_{\Omega}\overline{J}{({u_{0,n}})}+\int%
\limits_{\Omega}\mathcal{T}{{_{\ell+k}}\left(  {{u_{0,n}}-{z_{\nu}}}\right)
}+\int\limits_{\Omega}{{T_{\ell+k}}\left(  {{u_{0,n}}-{z_{\nu}}}\right)
{z_{\nu}}}\label{a3}\\
&  +\int\limits_{\Omega}{{{\overline{T}}_{\ell+k}}\left(  {{v_{n}}(T)-}%
\langle{{{{{T_{k}}(v)\rangle}}_{\nu}}(T)}\right)  }+\int\limits_{\Omega
}\mathcal{J}{({v_{n}}(T))}+\int_{Q}{{{\left(  {{\Phi_{{\delta_{1}},{\delta
_{2}}}}}\right)  }_{t}}\left(  {{{\overline{T}}_{\ell+k}}\left(  {{v_{n}}%
-}\langle{{{{{T_{k}}(v)\rangle}}_{\nu}}}\right)  -\overline{J}({v_{n}}%
)}\right)  }\nonumber\\
&  +\nu\int_{Q}{(1-{\Phi_{{\delta_{1}},{\delta_{2}}}}){T_{\ell+k}}\left(
{{v_{n}}-}\langle{{{{{T_{k}}(v)\rangle}}_{\nu}}}\right)  \left(  {{T_{k}}%
(v)-}\langle{{{{{T_{k}}(v)\rangle}}_{\nu}}}\right)  +}\omega({l_{1}},{l_{2}%
},{n_{1}},{n_{2}},m).\nonumber
\end{align}
Next we add (\ref{a1a2}) and (\ref{a3}). Note that $\mathcal{J}{({v_{n}%
}(T))-J({v_{n}}(T)){v_{n}}(T)=-\overline{J}({v_{n}}(T)),}$ and also
$\mathcal{T}{{_{\ell+k}}\left(  {{u_{0,n}}-{z_{\nu}}}\right)  -{T_{\ell+k}%
}\left(  {{u_{0,n}}-{z_{\nu}}}\right)  ({z_{\nu}-{u_{0,n}})=-{\overline{T}%
}_{\ell+k}}\left(  {{u_{0,n}}-{z_{\nu}}}\right)  .}$ Then we find
\begin{align*}
A_{1}+A_{2}+A_{3}  &  \geqq\int\limits_{\Omega}{\left(  \overline{J}%
{({u_{0,n}})-{{\overline{T}}_{\ell+k}}\left(  {{u_{0,n}}-{z_{\nu}}}\right)
}\right)  }+\int\limits_{\Omega}{\left(  {{{\overline{T}}_{\ell+k}}\left(
{{v_{n}}(T)-\langle{{{{{T_{k}}(v)\rangle}}_{\nu}}}(T)}\right)  -\overline
{J}({v_{n}}(T))}\right)  }\\
&  +\int_{Q}{{{\left(  {{\Phi_{{\delta_{1}},{\delta_{2}}}}}\right)  }_{t}%
}\left(  {{{\overline{T}}_{\ell+k}}\left(  {{v_{n}}-}\langle{{{{{T_{k}%
}(v)\rangle}}_{\nu}}}\right)  -\overline{J}({v_{n}})}\right)  }\\
&  +\nu\int_{Q}{(1-{\Phi_{{\delta_{1}},{\delta_{2}}}}){T_{\ell+k}}\left(
{{v_{n}}-}\langle{{{{{T_{k}}(v)\rangle}}_{\nu}}}\right)  \left(  {{T_{k}}%
(v)-}\langle{{{{{T_{k}}(v)\rangle}}_{\nu}}}\right)  +}\omega({l_{1}},{l_{2}%
},{n_{1}},{n_{2}},m).
\end{align*}
Notice that ${{{{\overline{T}}_{\ell+k}}\left(  r{-}s\right)  -\overline
{J}(r)}\geqq0}$ for any $r,s\in\mathbb{R}$ such that $\left\vert s\right\vert
\leqq k;$ thus
\[
\int\limits_{\Omega}{\left(  {{{\overline{T}}_{\ell+k}}\left(  {{v_{n}}%
(T)-}\langle{{{{{T_{k}}(v)\rangle}}_{\nu}}(T)}\right)  -\overline{J}({v_{n}%
}(T))}\right)  \geqq0.}%
\]
And $\left\{  {{u_{0,n}}}\right\}  $ converges to $u_{0}$ in $L^{1}(\Omega)$
and $\left\{  v_{n}\right\}  $ converges to $v$ in $L^{1}(Q)$ from Proposition
\ref{mun}. Thus we obtain%
\[%
\begin{array}
[c]{c}%
A_{1}+A_{2}+A_{3}\geqq\int_{\Omega}{\left(  \overline{J}{({u_{0}}%
)-{{\overline{T}}_{\ell+k}}\left(  {{u_{0}}-{z_{\nu}}}\right)  }\right)
}+\int_{Q}{{{\left(  {{\Phi_{{\delta_{1}},{\delta_{2}}}}}\right)  }_{t}%
}\left(  {{{\overline{T}}_{\ell+k}}\left(  {v-}\langle{{{{{T_{k}}(v)\rangle}%
}_{\nu}}}\right)  -\overline{J}(v)}\right)  }\\
\\
+\nu\int_{Q}{(1-{\Phi_{{\delta_{1}},{\delta_{2}}}}){T_{\ell+k}}\left(
{v-}\langle{{{{{T_{k}}(v)\rangle}}_{\nu}}}\right)  \left(  {{T_{k}}%
(v)-}\langle{{{{{T_{k}}(v)\rangle}}_{\nu}}}\right)  +}\omega({l_{1}},{l_{2}%
},{n_{1}},{n_{2}},m,n).
\end{array}
\]
Moreover ${{T_{\ell+k}}\left(  r{-s}\right)  \left(  {{T_{k}}(r)-s}\right)
\geqq0}$ for any $r,s\in\mathbb{R}$ such that $\left\vert s\right\vert \leqq
k,$ hence
\begin{align*}
A_{1}+A_{2}+A_{3}  &  \geqq\int_{\Omega}{\left(  \overline{J}{({u_{0}%
})-{{\overline{T}}_{\ell+k}}\left(  {{u_{0}}-{z_{\nu}}}\right)  }\right)
}+\int_{Q}{{{\left(  {{\Phi_{{\delta_{1}},{\delta_{2}}}}}\right)  }_{t}%
}\left(  {{{\overline{T}}_{\ell+k}}\left(  {v-}\langle{{{{{T_{k}}(v)\rangle}%
}_{\nu}}}\right)  -\overline{J}(v)}\right)  }\\
& \\
&  {+}\omega({l_{1}},{l_{2}},{n_{1}},{n_{2}},m,n).
\end{align*}
As $\nu\rightarrow\infty,$ $\left\{  z_{\nu}\right\}  $ converges to
$T_{k}(u_{0}),$ $a.e.$ in $\Omega$, thus we get%
\begin{align*}
A_{1}+A_{2}+A_{3}  &  \geqq\int\limits_{\Omega}{\left(  \overline{J}{({u_{0}%
})-{{\overline{T}}_{\ell+k}}\left(  {{u_{0}}-{T_{k}}({u_{0}})}\right)
}\right)  }+\int_{Q}{{{\left(  {{\Phi_{{\delta_{1}},{\delta_{2}}}}}\right)
}_{t}}\left(  {{{\overline{T}}_{\ell+k}}\left(  {v-{T_{k}}(v)}\right)
-\overline{J}(v)}\right)  }\\
&  +\omega({l_{1}},{l_{2}},{n_{1}},{n_{2}},m,n,\nu).
\end{align*}
Finally $\left\vert {{{\overline{T}}_{\ell+k}}\left(  r{-{T_{k}}(r)}\right)
-\overline{J}(r)}\right\vert \leqq2k|r|{\chi_{\left\{  \left\vert r\right\vert
\geqq\ell\right\}  }}$ for any $r\in\mathbb{R},$ thus
\[
A_{1}+A_{2}+A_{3}\geqq\omega({l_{1}},{l_{2}},{n_{1}},{n_{2}},m,n,\nu,\ell).
\]
Combining all the estimates, we obtain $I_{2}\leqq\omega(l_{1},l_{2}%
,n_{1},n_{2},m,n,\nu,\ell,\delta_{1},\delta_{2})$ which implies (\ref{120510}%
), since $I_{2}$ does not depend on $l_{1},l_{2},n_{1},n_{2},m,\ell.$\medskip
\end{proof}

Next we conclude the proof of Theorem \ref{sta}:

\begin{lemma}
\label{concl} The function $u$ is a R-solution of (\ref{pmu}).
\end{lemma}

\begin{proof}
(i) First show that $u$ satisfies (\ref{renor}). Here we proceed as in
\cite{Pe08}. Let $\varphi\in X\cap L^{\infty}(Q)$ such $\varphi_{t}\in
X^{\prime}+L^{1}(Q),$ $\varphi(.,T)=0,$ and $S\in W^{2,\infty}(\mathbb{R})$,
such that $S^{\prime}$ has compact support on $\mathbb{R}$, $S(0)=0$. Let
$M>0$ such that supp$S^{\prime}\subset\lbrack-M,M]$. Taking successively
$(\varphi,S)$ and $(\varphi\psi_{\delta}^{\pm},S)$ as test functions in
(\ref{renor}) applied to $u_{n}$, we can write
\[
A_{1}+A_{2}+A_{3}+A_{4}=A_{5}+A_{6}+A_{7},\qquad A_{2,\delta,\pm}%
+A_{3,\delta,\pm}+A_{4,\delta,\pm}=A_{5,\delta,\pm}+A_{6,\delta,\pm
}+A_{7,\delta,\pm},
\]
where
\begin{align*}
A_{1}  &  =-\int_{\Omega}{\varphi(0)S({u_{0,n}}),\quad}A_{2}=-\int%
_{Q}{{\varphi_{t}}S({v_{n}}),\quad}A_{2,\delta,\pm}=-\int_{Q}(\varphi
\psi_{\delta}^{\pm}){{_{t}}S({v_{n}}),}\\
A_{3}  &  =\int_{Q}{S^{\prime}({v_{n}})A(x,t,\nabla u_{n}).\nabla\varphi
,\quad}A_{3,\delta,\pm}=\int_{Q}{S^{\prime}({v_{n}})A(x,t,\nabla u_{n}%
).\nabla(\varphi\psi_{\delta}^{\pm}),}%
\end{align*}%
\begin{align*}
A_{4}  &  =\int_{Q}{S^{\prime\prime}({v_{n}})\varphi A(x,t,\nabla
u_{n}).\nabla v_{n},\quad}A_{4,\delta,\pm}=\int_{Q}{S^{\prime\prime}({v_{n}%
})\varphi\psi_{\delta}^{\pm}A(x,t,\nabla u_{n}).\nabla v_{n},}\\
A_{5}  &  =\int_{Q}{S^{\prime}(}v{{_{n}})\varphi d{\widehat{\lambda_{n,0}%
},\quad A}}_{6}=\int_{Q}{S^{\prime}({v_{n}})\varphi d{\rho_{n,0},}\quad}%
A_{7}=-\int_{Q}{S^{\prime}(}v{{_{n}})\varphi d{\eta_{n,0},}}\\
A_{5,\delta,\pm}  &  =\int_{Q}{S^{\prime}(}v{{_{n}})\varphi\psi_{\delta}^{\pm
}d{\widehat{\lambda_{n,0}},{\quad}A}}_{6,\delta,\pm}=\int_{Q}{S^{\prime
}({v_{n}})\varphi\psi_{\delta}^{\pm}d{\rho_{n,0},}\quad}A_{7,\delta,\pm}%
=-\int_{Q}{S^{\prime}(}v{{_{n}})\varphi\psi_{\delta}^{\pm}d{\eta_{n,0}.}}%
\end{align*}

\noindent\ Since $\left\{  u_{0,n}\right\}  $ converges to $u_{0}$ in
$L^{1}(\Omega),$ and $\left\{  S({v_{n}})\right\}  $ converges to $S(v)$
strongly in $X$ and weak$^{\ast}$ in $L^{\infty}(Q),$ there holds, from
(\ref{12054}),
\[
A_{1}=-\int\limits_{\Omega}{\varphi(0)S({u_{0}})}+\omega(n),\quad A_{2}%
=-\int_{Q}{{\varphi_{t}}S(v)}+\omega(n),\quad A_{2,\delta,\psi_{\delta}^{\pm}%
}=\omega(n,\delta).
\]
\newline Moreover $T_{M}({v_{n}})$ converges to $T_{M}(v)$, then $T_{M}%
({v_{n}})+h_{n}$ converges to $T_{k}(v)+h$ strongly in $X$, thus%
\begin{align*}
A_{3}  &  =\int_{Q}{S^{\prime}({v_{n}})A(x,t,\nabla\left(  {{T_{M}}\left(
{{v_{n}}}\right)  +{h_{n}}}\right)  ).\nabla\varphi}\\
&  =\int_{Q}{S^{\prime}(v)A(x,t,\nabla\left(  {{T_{M}}\left(  {{v}}\right)
+{h}}\right)  ).\nabla\varphi}+\omega(n)\\
&  =\int_{Q}{S^{\prime}(v)A(x,t,\nabla u).\nabla\varphi}+\omega(n);
\end{align*}
and
\begin{align*}
A_{4}  &  =\int_{Q}{S^{\prime\prime}({v_{n}})\varphi A(x,t,\nabla\left(
{{T_{M}}\left(  {{v_{n}}}\right)  +{h_{n}}}\right)  ).\nabla{T_{M}}\left(
{{v_{n}}}\right)  }\\
&  =\int_{Q}{S^{\prime\prime}(v)\varphi A(x,t,\nabla\left(  {{T_{M}}\left(
{{v}}\right)  +{h}}\right)  ).\nabla{T_{M}}\left(  v\right)  }+\omega(n)\\
&  =\int_{Q}{S^{\prime\prime}(v)\varphi A(x,t,\nabla u).\nabla v}+\omega(n).
\end{align*}
In the same way, since ${\psi_{\delta}^{\pm}}$ converges to $0$ in $X,$
\begin{align*}
A_{3,\delta,\pm}  &  =\int_{Q}{S^{\prime}(v)A(x,t,\nabla u).\nabla(\varphi
\psi_{\delta}^{\pm})}+\omega(n)=\omega(n,\delta),\\
A_{4,\delta,\pm}  &  =\int_{Q}{S^{\prime\prime}(v)\varphi\psi_{\delta}^{\pm
}A(x,t,\nabla u).\nabla v}+\omega(n)=\omega(n,\delta).
\end{align*}
And $\left\{  {{{g_{n}}}}\right\}  $ converges strongly in $(L^{p^{\prime}%
}(\Omega))^{N},$ thus
\begin{align*}
A_{5}  &  =\int_{Q}{{S^{\prime}({v_{n}})\varphi{f_{n}}}+}\int_{Q}{{S^{\prime
}({v_{n}}){g_{n}.}\nabla\varphi}+}\int_{Q}{{S}}^{\prime\prime}{{({v_{n}%
})\varphi{g_{n}.}\nabla{T_{M}}({v_{n}})}}\\
&  =\int_{Q}{{S^{\prime}(v)\varphi f}+}\int_{Q}{{S^{\prime}(v)g.\nabla\varphi
}+}\int_{Q}{{S}}^{\prime\prime}{{(v)\varphi g.\nabla{T_{M}}(v)}+\omega(n)}\\
&  {=}\int_{Q}{{S^{\prime}(v)\varphi}d{\widehat{\mu_{0}}}+\omega(n).}%
\end{align*}
and $A_{5,\delta,\pm}{=}\int_{Q}{{S^{\prime}(v)\varphi}\psi_{\delta}^{\pm
}d{\widehat{\lambda_{n,0}}}+\omega(n)=}\omega(n,\delta).$ Then $A_{6,\delta
,\pm}+A_{7,\delta,\pm}=\omega(n,\delta)$. From (\ref{12054}) we verify that
$A_{7,\delta,+}=\omega(n,\delta)$ and $A_{6,\delta,-}=\omega(n,\delta)$.
Moreover, from (\ref{muno}) and (\ref{12054}), we find
\[
\left\vert A_{6}-A_{6,\delta,+}\right\vert \leqq\int_{Q}\left\vert {S^{\prime
}({v_{n}})\varphi}\right\vert {(1-\psi_{\delta}^{+})d{\rho_{n,0}}}%
\leqq{\left\Vert S\right\Vert _{{W^{2,\infty}}(\mathbb{R})}}{\left\Vert
\varphi\right\Vert _{{L^{\infty}}({Q})}}\int_{Q}{(1-\psi_{\delta}^{+}%
)d{\rho_{n}}}=\omega(n,\delta).
\]
Similarly we also have $\left\vert A_{7}-A_{7,\delta,-}\right\vert \leqq
\omega(n,\delta)$. Hence $A_{6}=\omega(n)$ and $A_{7}=\omega(n).$ Therefore,
we finally obtain (\ref{renor}):
\begin{equation}
-\int\limits_{\Omega}{\varphi(0)S({u_{0}})}-\int_{Q}{{\varphi_{t}}S(v)+}%
\int_{Q}{S^{\prime}(v)A(x,t,\nabla u).\nabla\varphi+}\int_{Q}{S^{\prime\prime
}(v)\varphi A(x,t,\nabla u).\nabla v=}\int_{Q}{{S^{\prime}(v)\varphi
}d{\widehat{\mu_{0}}.}} \label{renor5}%
\end{equation}
\medskip

(ii) Next, we prove (\ref{renor2}) and (\ref{renor3}). We take $\varphi\in
C_{c}^{\infty}(Q)$ and take $({(1-\psi_{\delta}^{-})\varphi,}\overline{H_{m}%
})$ as test functions in (\ref{renor5}), with $\overline{H_{m}}$ as in
(\ref{Hm}). We can write $D_{1,m}+D_{2,m}=D_{3,m}+D_{4,m}+D_{5,m},$ where
\begin{equation}%
\begin{array}
[c]{l}%
D_{1,m}=-\int\limits_{Q}{{{\left(  {(1-\psi_{\delta}^{-})\varphi}\right)
}_{t}}\overline{H_{m}}(v)},\quad\quad D_{2,m}=\int\limits_{Q}H_{m}%
{(}v{)A(x,t,\nabla u).\nabla\left(  {(1-\psi_{\delta}^{-})\varphi}\right)
},\\
\\
D_{3,m}=\int\limits_{Q}H_{m}{(}v{)(1-\psi_{\delta}^{-})\varphi d{\widehat{\mu
_{0}},}}\quad\quad D_{4,m}=\frac{1}{m}\int\limits_{m\leqq v\leqq2m}%
{(1-\psi_{\delta}^{-})\varphi{A(x,t,\nabla u).\nabla{v},}}\\
\\
D_{5,m}=-\frac{1}{m}\int\limits_{-2m\leqq v\leqq-m}{(1-\psi_{\delta}%
^{-})\varphi{A(x,t,\nabla u)\nabla v.}}%
\end{array}
\label{april281}%
\end{equation}
Taking the same test functions in (\ref{renor}) applied to $u_{n},$ there
holds $D_{1,m}^{n}+D_{2,m}^{n}=D_{3,m}^{n}+D_{4,m}^{n}+D_{5,m}^{n}$, where%
\begin{equation}%
\begin{array}
[c]{l}%
D_{1,m}^{n}=-\int\limits_{Q}{{{\left(  {(1-\psi_{\delta}^{-})\varphi}\right)
}_{t}}\overline{H_{m}}(v}_{n}{)},\quad\quad\quad D_{2,m}^{n}=\int%
\limits_{Q}H_{m}({v}_{n})A(x,t,\nabla u_{n}).\nabla\left(  {(1-\psi_{\delta
}^{-})\varphi}\right)  ,\\
\\
D_{3,m}^{n}=\int\limits_{Q}H_{m}({v}_{n})(1-\psi_{\delta}^{-})\varphi
d(\widehat{\lambda_{n,0}}+\rho_{n,0}-\eta_{n,0}){{,}}\quad D_{4,m}^{n}%
=\frac{1}{m}\int\limits_{m\leqq v\leqq2m}{(1-\psi_{\delta}^{-})\varphi
A(x,t,\nabla u_{n}).{\nabla{v}}}_{n}{{,}}\\
\\
D_{5,m}^{n}=-\frac{1}{m}\int\limits_{-2m\leqq{v}_{n}\leqq-m}(1-\psi_{\delta
}^{-})\varphi A(x,t,\nabla u_{n}).{{\nabla{v}_{n}}}%
\end{array}
\label{281n}%
\end{equation}
In (\ref{281n}), we go to the limit as $m\rightarrow\infty$. Since $\left\{
\overline{H}_{m}(v_{n})\right\}  $ converges to $v_{n}$ and $\left\{
H_{m}(v_{n})\right\}  $ converges to $1,$ $a.e.$ in $Q,$ and $\left\{  \nabla
H_{m}(v_{n})\right\}  $ converges to $0,$ weakly in $(L^{p}(Q))^{N}$ , we
obtain the relation $D_{1}^{n}+D_{2}^{n}=D_{3}^{n}+D^{n},$ where%
\begin{align*}
D_{1}^{n} &  =-\int_{Q}{{\left(  {(1-\psi_{\delta}^{-})\varphi}\right)  }_{t}%
}{v}_{n},\quad D_{2}^{n}=\int_{Q}A(x,t,\nabla u_{n})\nabla\left(
{(1-\psi_{\delta}^{-})\varphi}\right)  ,\quad D_{3}^{n}=\int_{Q}%
{(1-\psi_{\delta}^{-})\varphi d\widehat{\lambda_{n,0}}}\\
D^{n} &  =\int_{Q}{(1-\psi_{\delta}^{-})\varphi d(\rho_{n,0}-\eta_{n,0})+}%
\int_{Q}{(1-\psi_{\delta}^{-})\varphi d(({{\rho}}_{n,s}-\eta_{n,s})}%
^{+}-{({{\rho}}_{n,s}-\eta_{n,s})}^{-})\\
&  =\int_{Q}{(1-\psi_{\delta}^{-})\varphi d({{\rho}}_{n}-\eta_{n}).}%
\end{align*}
Clearly, $D{_{i,m}}-D{_{i}^{n}}=\omega(n,m)$ for $i=1,2,3.$ From  Lemma
(\ref{april261}) and (\ref{12054})-(\ref{12057}), we obtain $D_{5,m}%
=\omega(n,m,\delta),$ and
\[
{\frac{1}{m}\int\limits_{\left\{  m\leqq v<2m\right\}  }\psi_{\delta}%
^{-}\varphi A(x,t,\nabla u).\nabla{v}}=\omega(n,m,\delta),
\]
thus,
\[
D_{4,m}=\frac{1}{m}{\int\limits_{\left\{  m\leqq v<2m\right\}  }\varphi
A(x,t,\nabla u).\nabla v}+\omega(n,m,\delta).
\]
Since $\left\vert \int_{Q}{(1-\psi_{\delta}^{-})\varphi d\eta_{n}}\right\vert
\leqq{\left\Vert \varphi\right\Vert _{{L^{\infty}}}}\int_{Q}{(1-\psi_{\delta
}^{-})d\eta_{n},}$ it follows that $\int_{Q}{(1-\psi_{\delta}^{-})\varphi
d\eta_{n}}=\omega(n,m,\delta)$ from (\ref{12057})$.$ And $\left\vert \int%
_{Q}{{\psi_{\delta}^{-}\varphi d\rho_{n}}}\right\vert \leqq{\left\Vert
\varphi\right\Vert _{{L^{\infty}}}}\int_{Q}{\psi_{\delta}^{-}d\rho_{n},}$
thus, from (\ref{12054}), $\int_{Q}{(1-\psi_{\delta}^{-})\varphi d{{\rho}}%
_{n}}=\int_{Q}{\varphi d\mu_{s}^{+}}+\omega(n,m,\delta).$ Then $D^{n}=\int%
_{Q}{\varphi d\mu_{s}^{+}}+\omega(n,m,\delta).$ Therefore by substraction, we
get
\[
\frac{1}{m}{\int\limits_{\left\{  m\leqq v<2m\right\}  }\varphi A(x,t,\nabla
u).\nabla v}=\int_{Q}{\varphi d\mu_{s}^{+}}+\omega(n,m,\delta),
\]
hence
\begin{equation}
\lim_{m\rightarrow\infty}\frac{1}{m}\int\limits_{\left\{  m\leqq v<2m\right\}
}\varphi A(x,t,\nabla u).\nabla v=\int_{Q}{\varphi d\mu_{s}^{+},}\label{en}%
\end{equation}
which proves (\ref{renor2}) when $\varphi\in C_{c}^{\infty}(Q).$ Next assume
only $\varphi\in C^{\infty}(\overline{Q})$. Then
\[%
\begin{array}
[c]{l}%
\lim_{m\rightarrow\infty}\frac{1}{m}{\int\limits_{\left\{  m\leqq
v<2m\right\}  }\varphi A(x,t,\nabla u).\nabla v}\\
\\
=\lim_{m\rightarrow\infty}\frac{1}{m}{\int\limits_{\left\{  m\leqq
v<2m\right\}  }\varphi\psi_{\delta}^{+}A(x,t,\nabla u)\nabla v}+\lim
_{m\rightarrow\infty}\frac{1}{m}{\int\limits_{\left\{  m\leqq v<2m\right\}
}\varphi(1-\psi_{\delta}^{+})A(x,t,\nabla u).\nabla v}\\
\\
=\int_{Q}{\varphi\psi_{\delta}^{+}d\mu_{s}^{+}}+\lim_{m\rightarrow\infty}%
\frac{1}{m}{\int\limits_{\left\{  m\leqq v<2m\right\}  }\varphi(1-\psi
_{\delta}^{+})A(x,t,\nabla u).\nabla v}=\int_{Q}{\varphi d\mu_{s}^{+}}+D,
\end{array}
\]
where,
\[
D=\int_{Q}{\varphi(1-\psi_{\delta}^{+})d\mu_{s}^{+}}+\lim_{n\rightarrow\infty
}\frac{1}{m}{\int\limits_{\left\{  m\leqq v<2m\right\}  }\varphi
(1-\psi_{\delta}^{+})A(x,t,\nabla u).\nabla v=\;}\omega(\delta).
\]
Therefore, (\ref{en}) still holds for $\varphi\in C^{\infty}(\overline{Q}),$
and we deduce (\ref{renor2}) by density, and similarly, (\ref{renor3}). This
completes the proof of Theorem \ref{sta}.
\end{proof}

As a consequence of Theorem \ref{sta}, we get the following:

\begin{corollary}
\label{051120131} Let $u_{0}\in L^{1}(\Omega)$ and $\mu\in\mathcal{M}_{b}(Q)$.
Then there exists a R-solution $u$ to the problem \ref{pmu} with data
$(\mu,u_{0}).$ Furthermore, if $v_{0}\in L^{1}(\Omega)$ and $\omega
\in\mathcal{M}_{b}(Q)$ such that $u_{0}\leqq v_{0}$ and $\mu\leqq\omega,$ then
one can find R-solution $v$ to the problem \ref{pmu} with data $(\omega
,v_{0})$ such that $u\leqq v$.

In particular, if $a\equiv0$ in (\ref{condi1}), then $u$ satisfies
(\ref{1810131}) and ${\left\Vert v\right\Vert _{{L^{\infty}}((0,T);{L^{1}%
}(\Omega))}}\leqq M$ with $M=||u_{0}||_{1,\Omega}+|\mu|(Q)$.
\end{corollary}

\section{Equations with perturbation terms\label{appli}}

Let $A:Q\times\mathbb{R}^{N}$ $\rightarrow\mathbb{R}^{N}$ satisfying
(\ref{condi1}), (\ref{condi2}) with $a\equiv0$. Let $\mathcal{G}:\Omega
\times(0,T)\times\mathbb{R}\mapsto\mathbb{R}$ be a Caratheodory function. If
$U$ is a function defined in $Q$ we define the function $\mathcal{G}(U)$ in
$Q$ by
\[
\mathcal{G}(U)(x,t)=\mathcal{G}(x,t,U(x,t))\qquad\text{for }a.e.\text{
}(x,t)\in Q.
\]
We consider the problem (\ref{pga}):
\[
\left\{
\begin{array}
[c]{l}%
{u_{t}}-\text{div}(A(x,t,\nabla u))+\mathcal{G}(u)=\mu\qquad\text{in }Q,\\
{u}=0\qquad\text{in }\partial\Omega\times(0,T),\\
u(0)=u_{0}\qquad\text{in }\Omega,
\end{array}
\right.
\]
where $\mu\in\mathcal{M}_{b}(Q)$, $u_{0}\in L^{1}(\Omega)$. We say that $u$ is
a R-solution of problem (\ref{pga}) if $\mathcal{G}(u)\in L^{1}(Q)$ and $u$ is
a R-solution of (\ref{pmu}) with data $(\mu-\mathcal{G}(u),u_{0})$.

\subsection{Subcritical type results}

For proving Theorem \ref{new}, we begin by an integration Lemma:

\begin{lemma}
Let $G$ satisfying (\ref{asg}). If a measurable function $V$ in $Q$ satisfies
\[
\mathrm{meas}\left\{  {|V|}\geqq{t}\right\}  \leqq M{t}^{-{p}_{c}}%
,\qquad\forall t\geqq1,
\]
for some $M>0,$ then for any $L>1$,
\begin{equation}
\int\limits_{\left\{  |V|\geqq L\right\}  }G{(|V|)}\leqq{p}_{c}M\int%
_{L}^{\infty}G{\left(  s\right)  {s^{-1-p_{c}}}ds.} \label{23051}%
\end{equation}

\end{lemma}

\begin{proof}
Indeed, setting $G_{L}(s)={\chi_{\lbrack L,\infty)}}(s)G(s),$ we have
\[
\int\limits_{\left\{  |V|\geqq L\right\}  }G(|V|)dxdt=\int_{Q}{G_{L}%
(|V|)dxdt}\leqq\int_{0}^{\infty}{{G_{L}}(|V|^{\ast}(s))ds}%
\]
where $|V|^{\ast}$ is and the rearrangement of $|V|,$ defined by
\[
|V|^{\ast}(s)=\inf\{a>0:\mathrm{meas}\left\{  |V|>a\right\}  )\leqq
s\},\qquad\forall s\geqq0.
\]
From the assumption, we get $|V|^{\ast}(s)\leqq\sup\left(  (Ms^{-1}%
)^{p_{c}^{-1}},1\right)  $. Thus, for any $L>1,$
\[
\int\limits_{\left\{  |V|\geqq L\right\}  }{G(|V|)dxdt}\leqq\int_{0}^{\infty
}{{G_{L}}\left(  \sup\left(  (Ms^{-1})^{p_{c}^{-1}},1\right)  \right)
ds}=p_{c}M\int_{L}^{\infty}{G\left(  s\right)  {s^{-1-p_{c}}}ds},
\]
which implies (\ref{23051}).
\end{proof}

\medskip

\begin{proof}
[Proof of Theorem \ref{new}]\noindent\textbf{Proof of (i) } Let $\mu=\mu
_{0}+\mu_{s}\in\mathcal{M}_{b}(Q)$, with $\mu_{0}\in\mathcal{M}_{0}(Q),\mu
_{s}\in\mathcal{M}_{s}(Q),$ and $u_{0}\in L^{1}(\Omega).$ Then $\mu_{0}%
^{+},\mu_{0}^{-}$ can be decomposed as $\mu_{0}^{+}=(f_{1},g_{1},h_{1}),$
$\mu_{0}^{-}=(f_{2},g_{2},h_{2})$. Let $\mu_{n,s,i}\in C_{c}^{\infty}(Q),$
$\mu_{n,s,i}\geqq0,$ converging respectively to $\mu_{s}^{+},\mu_{s}^{-}$ in
the narrow topology. By Lemma \ref{att}, we can find $f_{n,i},g_{n,i}%
,h_{n,i}\in C_{c}^{\infty}(Q)$ which strongly converge to $f_{i},g_{i},h_{i}$
in $L^{1}(Q),\left(  L^{p^{\prime}}(Q)\right)  ^{N}$ and $X$ respectively,
$i=1,2,$ such that $\mu_{0}^{+}=(f_{1},g_{1},h_{1}),$ $\mu_{0}^{-}%
=(f_{2},g_{2},h_{2}),$ and $\mu_{n,0,i}=(f_{n,i},g_{n,i},h_{n,i})$, converging
respectively to $\mu_{0}^{+},$ $\mu_{0}^{-}$ in the narrow topology.
Furthermore, if we set
\[
\mu_{n}=\mu_{n,0,1}-\mu_{n,0,2}+\mu_{n,s,1}-\mu_{n,s,2},
\]
then $|\mu_{n}|(Q)\leq|\mu|(Q)$. Consider a sequence $\left\{  u_{0,n}%
\right\}  \subset C_{c}^{\infty}(\Omega)$ which strongly converges to $u_{0}$
in $L^{1}(\Omega)$ and satisfies $||u_{0,n}||_{1,\Omega}\leqq||u_{0}%
||_{1,\Omega}$. \medskip\newline Let $u_{n}$ be a solution of
\[
\left\{
\begin{array}
[c]{l}%
(u_{n})_{t}-\text{div}(A(x,t,\nabla u_{n}))+\mathcal{G}(u_{n})=\mu_{n}%
\qquad\text{in }Q,\\
u_{n}=0\qquad\text{on }\partial\Omega\times(0,T),\\
u_{n}(0)=u_{0,n}\qquad\text{in }\Omega.
\end{array}
\right.
\]
We can choose $\varphi=\varepsilon^{-1}T_{\varepsilon}(u_{n})$ as test
function of above problem. Then we find
\[
\int_{Q}{{{\left(  {{\varepsilon^{-1}\overline{T_{\varepsilon}}}({u_{n}}%
)}\right)  }_{t}}}+\int_{Q}{{\varepsilon^{-1}}A(x,t,\nabla T_{\varepsilon
}(u_{n})).\nabla T_{\varepsilon}(u_{n})}+\int_{Q}\mathcal{G}{(x,t,{u_{n}%
}){\varepsilon^{-1}}{T_{\varepsilon}}({u_{n}})}=\int_{Q}{{\varepsilon^{-1}%
}{T_{\varepsilon}}({u_{n}})d{\mu_{n}}.}%
\]
Since
\[
\int_{Q}{{{\left(  {{\varepsilon^{-1}\overline{T_{\varepsilon}}}({u_{n}}%
)}\right)  }_{t}}}=\int_{\Omega}{{\varepsilon^{-1}\overline{T_{\varepsilon}}%
}({u_{n}}(T))dx}-\int_{\Omega}{{\varepsilon^{-1}\overline{T_{\varepsilon}}%
}({u_{0,n}})dx}\geqq-||{u_{0,n}}|{|_{{L^{1}}(\Omega)}},
\]
there holds
\[
\int_{Q}{\mathcal{G}(x,t,{u_{n}}){\varepsilon^{-1}}{T_{\varepsilon}}({u_{n}}%
)}\leqq|{\mu_{n}}|(Q)+||{u_{0,n}}|{|_{{L^{1}}(\Omega)}}\leqq|{\mu
}|(Q)+||{u_{0}}||_{1,\Omega}.
\]
Letting $\varepsilon\rightarrow0$, we obtain
\begin{equation}
\int_{Q}{\left\vert {\mathcal{G}(x,t,{u_{n}})}\right\vert }\leqq|{\mu
}|(Q)+||{u_{0}}||_{1,\Omega}. \label{esg}%
\end{equation}
Next apply Proposition \ref{estsup} and Remark \ref{h2} to $u_{n}$ with
initial data $u_{0,n}$ and measure data $\mu_{n}-\mathcal{G}(u_{n})\in
L^{1}(Q)$, we get%
\[
\mathrm{meas}\left\{  {|u_{n}|\geqq s}\right\}  \leqq C(|{\mu}|(Q)+||{u_{0}%
}||_{{L^{1}}(\Omega)})^{\frac{p+N}{N}}{s^{-p_{c}}},\qquad\forall s>0,\forall
n\in\mathbb{N},
\]
for some $C=C(N,p,c_{1},c_{2}).$ Since $|\mathcal{G}(x,t,u_{n})|\leqq
G(|u_{n}|)$, we deduce from (\ref{23051}) that $\{|\mathcal{G}(u_{n})|\}$ is
equi-integrable. Then, thanks to Proposition \ref{mun}, up to a subsequence,
$\{u_{n}\}$ converges to some function $u,$ $a.e.$ in $Q,$ and $\left\{
\mathcal{G}(u_{n})\right\}  $ converges to $\mathcal{G}(u)$ in $L^{1}(Q)$.
Therefore, by Theorem \ref{sta}, $u$ is a R-solution of (\ref{pro0}).\medskip

\noindent\textbf{Proof of (ii).} Let $\{u_{n}\}_{n\geqq1}$ be defined by
induction as nonnegative R-solutions of
\[
\left\{
\begin{array}
[c]{l}%
(u_{1})_{t}-\text{div}(A(x,t,\nabla u_{1}))=\mu\qquad\text{in }Q,\\
u_{1}=0\qquad\text{on }\partial\Omega\times(0,T),\\
u_{1}(0)=u_{0}\qquad\text{in }\Omega,
\end{array}
\right.  \qquad\left\{
\begin{array}
[c]{l}%
(u_{n+1})_{t}-\text{div}(A(x,t,\nabla u_{n+1}))=\mu-\lambda\mathcal{G}%
(u_{n})\qquad\text{in }Q,\\
u_{n+1}=0\qquad\text{on }\partial\Omega\times(0,T),\\
u_{n+1}(0)=u_{0}\qquad\text{in }\Omega,
\end{array}
\right.
\]
Thanks to Corollary \ref{051120131} we can assume that $\{u_{n}\}$ is
nondecreasing and satisfies for any $s>0$ and $n\in\mathbb{N}$
\begin{equation}
\mathrm{meas}\left\{  {|u_{n}|\geqq s}\right\}  \leqq C_{1}K_{n}{s^{-p_{c}}},
\label{mnn}%
\end{equation}
where $C_{1}$ does not depend on $s,n,$ and
\begin{align*}
&  K_{1}=\left(  ||u_{0}||_{1,\Omega}+|\mu|(Q)\right)  ^{\frac{p+N}{N}},\\
K_{n+1}  &  =\left(  ||u_{0}||_{1,\Omega}+|\mu|(Q)+\lambda||\mathcal{G}%
(u_{n})||_{1,Q}\right)  ^{\frac{p+N}{N}},
\end{align*}
for any $n\geqq1.$Take $\varepsilon=\lambda+|\mu|(Q)+||u_{0}||_{L^{1}(\Omega
)}\leqq1$. Denoting by $C_{i}$ some constants independent on $n,\varepsilon,$
there holds $K_{1}\leqq C_{2}\varepsilon,$ and for $n\geqq1,$
\[
K{_{n+1}}\leqq C_{3}\varepsilon(||\mathcal{G}({u_{n}})||_{1,Q}^{1+\frac{p}{N}%
}+{1)}.
\]
From (\ref{23051}) and (\ref{mnn}), we find
\[
{\left\Vert {\mathcal{G}({u_{n}})}\right\Vert _{{L^{1}}({Q})}}\leqq\left\vert
{{Q}}\right\vert G(2)+\int\limits_{\left\{  {u_{n}}|\geqq2\right\}
|}{G(|{u_{n}}|)dxdt}\leqq\left\vert {{Q}}\right\vert G(2)+ C_{4}K{_{n}\int%
_{2}^{\infty}G\left(  s\right)  {s^{-1-p_{c}}}ds.}%
\]
Thus, ${K_{n+1}}\leqq{C}_{5}\varepsilon(K_{n}^{1+\frac{p}{N}}+{1)}$.
Therefore, if $\varepsilon$ is small enough, $\left\{  K_{n}\right\}  $ is
bounded. Then, again from (\ref{23051}) and the relation $|\mathcal{G}%
(x,t,u_{n})|\leqq G(|u_{n}|)$ we verify that $\{\mathcal{G}(u_{n})\}$
converges. Then by Theorem \ref{sta}, up to a subsequence, $\{u_{n}\}$
converges to a R-solution $u$ of (\ref{pro1}).\medskip
\end{proof}

\subsection{General case with absorption terms}

In the sequel we assume that $A:\Omega\times\mathbb{R}^{N}\longmapsto
\mathbb{R}^{N}$ does not depend on $t$. We recall a result obtained in
\cite{PhVe1},\cite{BiNQVe} in the elliptic case:

\begin{theorem}
\label{elli} Let $\Omega$ be a bounded domain of $\mathbb{R}^{N}.$
Let$A:\Omega\times\mathbb{R}^{N}\longmapsto\mathbb{R}^{N}$ satisfying
(\ref{condi3}),(\ref{condi4}).Then there exists a constant $\kappa$ depending
on $N,p,c_{3},c_{4}$ such that, if $\omega\in\mathcal{M}_{b}(\Omega)$ and $u$
is a R-solution of problem%
\[
\left\{
\begin{array}
[c]{l}%
-\text{div}(A(x,\nabla u))=\omega\qquad\text{in }\Omega,\\
{u}=0\qquad\text{on }\partial\Omega,
\end{array}
\right.
\]
there holds
\begin{equation}
-\kappa\mathbf{W}_{1,p}^{2\mathrm{diam}\Omega}[\omega^{-}]\leqq u\leqq
\kappa\mathbf{W}_{1,p}^{2\mathrm{diam}\Omega}[\omega^{+}].\label{enca}%
\end{equation}

\end{theorem}

Next we give a general result in case of absorption terms:

\begin{theorem}
\label{main} Let $p<N$ , $A:\Omega\times\mathbb{R}^{N}\longmapsto
\mathbb{R}^{N}$ satisfying (\ref{condi3}),(\ref{condi4}), and $\mathcal{G}%
:Q\times\mathbb{R}\longmapsto\mathbb{R}$ be a Caratheodory function such that
the map $s\mapsto\mathcal{G}(x,t,s)$ is nondecreasing and odd, for $a.e.$
$(x,t)$ in $Q$.

\medskip Let $\mu_{1},\mu_{2}\in\mathcal{M}_{b}^{+}(Q)$ such that there exist
$\omega_{n}\in\mathcal{M}_{b}^{+}(\Omega)$ and nondecreasing sequences
$\left\{  \mu_{1,n}\right\}  ,\left\{  \mu_{2,n}\right\}  $ in $\mathcal{M}%
_{b}^{+}(Q)$ with compact support in $Q$, converging to $\mu_{1},\mu_{2}$,
respectively in the narrow topology, and
\[
\mu_{1,n},\mu_{2,n}\leqq\omega_{n}\otimes\chi_{(0,T)},\qquad\mathcal{G}%
((n+\kappa\mathbf{W}_{1,p}^{2\mathrm{diam}(\Omega)}\left[  {\omega_{n}%
}\right]  ))\in L^{1}(Q),
\]
where the constant $c$ is given at Theorem \ref{elli}. Let $u_{0}\in
L^{1}(\Omega)$, and $\mu=$ $\mu_{1}-\mu_{2}.$ Then there exists a R-solution
$u$ of problem (\ref{pga}).

\medskip Moreover if $u_{0}\in L^{\infty}(\Omega),$ and $\omega_{n}\leq\gamma$
for any $n\in\mathbb{N}$, for some $\gamma\in\mathcal{M}_{b}^{+}(\Omega),$
then $a.e.$ in $Q$,
\begin{equation}
\left\vert u(x,t)\right\vert \leqq\kappa\mathbf{W}_{1,p}^{2\mathrm{diam}%
\Omega}{\gamma}(x)+||u_{0}||_{\infty,\Omega}. \label{abs}%
\end{equation}

\end{theorem}

For proving this result, we need two Lemmas:

\begin{lemma}
\label{T26} Let $\mathcal{G}$ satisfy the assumptions of Theorem \ref{main}
and $\mathcal{G}\in L^{\infty}(Q\times\mathbb{R})$. For $i=1,2,$ let
$u_{0,i}\in L^{\infty}(\Omega)$ be nonnegative, and $\lambda_{i}=\lambda
_{i,0}+\lambda_{i,s}\in\mathcal{M}_{b}^{+}(Q)$ with compact support in $Q$,
$\gamma\in\mathcal{M}_{b}^{+}(\Omega)$ with compact support in $\Omega$ such
that $\lambda_{i}\leqq\gamma\otimes\chi_{(0,T)}$. Let $\lambda_{i,0}%
=(f_{i},g_{i},h_{i})$ be a decomposition of $\lambda_{i,0}$ into functions
with compact support in $Q$. Then, there exist R-solutions $u,u_{1},u_{2},$ to
problems
\begin{equation}
u_{t}-\text{div}(A(x,\nabla u))+\mathcal{G}(u)=\lambda_{1}-\lambda_{2}%
\quad\text{in }Q,\qquad u=0\quad\text{on }\partial\Omega\times(0,T),\qquad
u(0)=u_{0,1}-u_{0,2}, \label{proba}%
\end{equation}%
\begin{equation}
(u_{i})_{t}-\text{div}(A(x,\nabla u_{i}))+\mathcal{G}(u_{i})=\lambda_{i}%
\quad\text{in }Q,\qquad u_{i}=0\quad\text{on }\partial\Omega\times
(0,T),~\qquad u_{i}(0)=u_{0,i}, \label{probb}%
\end{equation}
relative to decompositions $(f_{1,n}-f_{2,n}-\mathcal{G}(u_{n}),g_{1,n}%
-g_{2,n},h_{1,n}-h_{2,n})$, $(f_{i,n}-\mathcal{G}(u_{i,n}),g_{i,n},h_{i,n}),$
such that $a.e.$ in $Q,$
\begin{equation}
-||u_{0,2}||_{\infty,\Omega}-\kappa\mathbf{W}_{1,p}^{2\mathrm{diam}\Omega
}{\gamma}(x)\leqq-u_{2}(x,t)\leqq u(x,t)\leqq u_{1}(x,t)\leqq\kappa
\mathbf{W}_{1,p}^{2\mathrm{diam}\Omega}{\gamma}(x)+||u_{0,1}||_{\infty,\Omega
}, \label{17062}%
\end{equation}
and
\begin{equation}
\int_{Q}\left\vert \mathcal{G}(u)\right\vert \leqq\sum_{i=1,2}\left(
\lambda_{i}(Q)+||u_{0,i}||_{L^{1}(\Omega)}\right)  ,\quad\text{and\quad}%
\int_{Q}\mathcal{G}(u_{i})\leqq\lambda_{i}(Q)+||u_{0,i}||_{1,\Omega},\quad
i=1,2. \label{18061}%
\end{equation}
\medskip Furthermore, assume that $\mathcal{H},\mathcal{K}$ have the same
properties as $\mathcal{G},$ and $\mathcal{H}(x,t,s)\leqq\mathcal{G}%
(x,t,s)\leqq\mathcal{K}(x,t,s)$ for any $s\in(0,+\infty)$ and $a.e.$ in $Q.$
Then, one can find solutions $u_{i}(\mathcal{H}),u_{i}(\mathcal{K})$,
corresponding to $\mathcal{H},\mathcal{K}$ with data $\lambda_{i}$, such that
$u_{i}(\mathcal{H})\geqq u_{i}\geqq u_{i}(\mathcal{K})$, $i=1,2$.

Assume that $\omega_{i},\theta_{i}$ have the same properties as $\lambda_{i}$
and $\omega_{i}\leqq\lambda_{i}\leqq\theta_{i}$, $u_{0,i,1},u_{0,i,2}\in
L^{\infty+}(\Omega)$, $u_{0,i,2}\leqq u_{0,i}\leqq u_{0,i,1}.$ Then one can
find solutions $u_{i}(\omega_{i}),u_{i}(\theta_{i})$, corresponding to
$(\omega_{i},u_{0,i,2}),(\theta_{i},u_{0,i,1})$, such that $u_{i}(\omega
_{i},u_{0,i,2})\leqq u_{i}\leqq u_{i}(\theta_{i},u_{0,i,1}).$
\end{lemma}

\begin{proof}
Let $\left\{  \varphi_{1,n}\right\}  ,\left\{  \varphi_{2,n}\right\}  $ be
sequences of mollifiers in $\mathbb{R}$ and $\mathbb{R}^{N}$, and $\varphi
_{n}=\varphi_{1,n}\varphi_{2,n}$. Set $\gamma_{n}=\varphi_{2,n}\ast\gamma,$
and for $i=1,2,$ $u_{0,i,n}=\varphi_{2,n}\ast u_{0,i},$
\[
\lambda_{i,n}=\varphi_{n}\ast\lambda_{i}=f_{i,n}-\text{div}(g_{i,n}%
)+(h_{i,n})_{t}+\lambda_{i,s,n},
\]
where $f_{i,n}=\varphi_{n}\ast f_{i},\;g_{i,n}=\varphi_{n}\ast g_{i}%
,\;h_{i,n}=\varphi_{n}\ast h_{i},\;\lambda_{i,s,n}=\varphi_{n}\ast
\lambda_{i,s},$ and
\[
\lambda_{n}=\lambda_{1,n}-\lambda_{2,n}=f_{n}-\text{div}(g_{n})+(h_{n}%
)_{t}+\lambda_{s,n},
\]
where $f_{n}=f_{1,n}-f_{2,n},\;g_{n}=g_{1,n}-g_{2,n},\;h_{n}=h_{1,n}%
-h_{2,n},\;\lambda_{s,n}=\lambda_{1,s,n}-\lambda_{2,s,n}$. Then for $n$ large
enough, $\lambda_{1,n},\lambda_{2,n},\lambda_{n}\in C_{c}^{\infty}(Q)$,
$\gamma_{n}\in C_{c}^{\infty}(\Omega).$ Thus there exist unique solutions
$u_{n},u_{i,n},v_{i,n},$ $i=1,2,$ of problems
\[
(u_{n})_{t}-\text{div}(A(x,\nabla u_{n}))+\mathcal{G}(u_{n})=\lambda
_{1,n}-\lambda_{2,n}\quad\text{in }Q,\quad u_{n}=0\quad\text{on }%
\partial\Omega\times(0,T),\quad u_{n}(0)=u_{0,1,n}-u_{0,2,n}\quad\text{in
}\Omega,
\]%
\[
(u_{i,n})_{t}-\text{div}(A(x,\nabla u_{i,n}))+\mathcal{G}(u_{i,n}%
)=\lambda_{i,n}\quad\text{in }Q,\qquad u_{i,n}=0\quad\text{on }\partial
\Omega\times(0,T),\qquad u_{i,n}(0)=u_{0,i,n}\quad\text{in }\Omega,
\]%
\[
-\text{div}(A(x,\nabla w_{n}))=\gamma_{n}\quad\text{in }\Omega,\qquad
w_{n}=0\quad\text{on }\partial\Omega,
\]
such that
\[
-||u_{0,2}||_{\infty,\Omega}-w_{n}(x)\leqq-u_{2,n}(x,t)\leqq u_{n}(x,t)\leqq
u_{1,n}(x,t)\leqq w_{n}(x)+||u_{0,1}||_{\infty,\Omega},\quad a.e.\text{ in
}Q.
\]
Moreover, as in the Proof of Theorem \ref{new}, (i), there holds
\[
\int_{Q}|\mathcal{G}(u_{n})|\;\leqq\sum_{i=1,2}\left(  \lambda_{i}%
(Q)+||u_{0,i,n}||_{1,\Omega}\right)  ,\quad\text{and\quad}\int_{Q}%
\mathcal{G}(u_{i,n})\leqq\lambda_{i}(Q)+||u_{0,i,n}||_{1,\Omega},\quad i=1,2.
\]
By Proposition \ref{mun}, up to a common subsequence, $\left\{  u_{n}%
,u_{1,n},u_{2,n}\right\}  $ converge to some $(u,u_{1},u_{2}),$ $a.e.$ in $Q$.
Since $\mathcal{G}$ is bounded, in particular, $\left\{  \mathcal{G}%
(u_{n})\right\}  $ converges to $\mathcal{G}(u)$ and $\left\{  \mathcal{G}%
(u_{i,n})\right\}  $ converges to $\mathcal{G}(u_{i})$ in $L^{1}(Q)$. Thus,
(\ref{18061}) is satisfied. Morover $\left\{  \lambda_{i,n}-\mathcal{G}%
(u_{i,n}),f_{i,n}-\mathcal{G}(u_{i,n}),g_{i,n},h_{i,n},\lambda_{i,s,n}%
,u_{0,i,n}\right\}  $ and $\left\{  \lambda_{n}-\mathcal{G}(u_{n}%
),f_{n}-\mathcal{G}(u_{n}),g_{n},h_{n},\lambda_{s,n},u_{0,1,n}-u_{0,2,n}%
\right\}  $ are approximations of $(\lambda_{i}-\mathcal{G}(u_{i}%
),f_{i}-\mathcal{G}(u_{i}),g_{i},h_{i},\lambda_{i,s},u_{0,i})$ and
$(\lambda-\mathcal{G}(u),f-\mathcal{G}(u),g,h,\lambda_{s},u_{0,1}-u_{0,2})$,
in the sense of Theorem \ref{sta}. Thus, we can find (different) subsequences
converging $a.e.$ to $u,u_{1},u_{2},$ R-solutions of (\ref{proba}) and
(\ref{probb}). Furthermore, from \cite[Corollary 3.4]{NQVe}, up to a
subsequence, $\left\{  w_{n}\right\}  $ converges $a.e.$ in $Q$ to a
R-solution
\[
-\text{div}(A(x,\nabla w))=\gamma\quad\text{in }\Omega,\qquad w=0\quad\text{on
}\partial\Omega,
\]
such that $w\leqq c\mathbf{W}_{1,p}^{2\mathrm{diam}\Omega}{\gamma}$ $a.e.$ in
$\Omega$. Hence, we get the inequality (\ref{17062}). The other conclusions
follow in the same way.\medskip
\end{proof}

\begin{lemma}
\label{T27} Let $\mathcal{G}$ satisfy the assumptions of Theorem \ref{main}.
For $i=1,2,$ let $u_{0,i}\in L^{\infty}(\Omega)$ be nonnegative, $\lambda
_{i}\in\mathcal{M}_{b}^{+}(Q)$ with compact support in $Q$, and $\gamma
\in\mathcal{M}_{b}^{+}(\Omega)$ with compact support in $\Omega$, such that
\begin{equation}
\lambda_{i}\leqq\gamma\otimes\chi_{(0,T)},\qquad\mathcal{G}((||u_{0,i}%
||_{\infty,\Omega}+\kappa\mathbf{W}_{1,p}^{2\mathrm{diam}(\Omega)}{\gamma
}))\in L^{1}(Q). \label{hypg}%
\end{equation}
Then, there exist R-solutions $u,u_{1},u_{2}$ of the problems (\ref{proba})
and (\ref{probb}), respectively relative to the decompositions $(f_{1}%
-f_{2}-\mathcal{G}(u),g_{1}-g_{2},h_{1}-h_{2})$, $(f_{i}-\mathcal{G}%
(u_{i}),g_{i},h_{i}),$ satifying (\ref{17062}) and (\ref{18061}).\medskip

Moreover, assume that $\omega_{i},\theta_{i}$ have the same properties as
$\lambda_{i}$ and $\omega_{i}\leqq\lambda_{i}\leqq\theta_{i}$, $u_{0,i,1}%
,u_{0,i,2}\in L^{\infty+}(\Omega)$, $u_{0,i,2}\leqq u_{0,i}\leqq u_{0,i,1}.$
Then, one can find solutions $u_{i}(\omega_{i},u_{0,i,2}),$ $u_{i}(\theta
_{i},u_{0,i,1})$, corresponding with $(\omega_{i},u_{0,i,2}),$ $(\theta
_{i},u_{0,i,1})$, such that $u_{i}(\omega_{i},u_{0,i,2})\leqq u_{i}\leqq
u_{i}(\theta_{i},u_{0,i,1}).\medskip$
\end{lemma}

\begin{proof}
From Lemma \ref{T26} there exist R-solutions $u_{n}$, $u_{i,n}$ to problems
\[
(u_{n})_{t}-\text{div}(A(x,\nabla u_{n}))+T_{n}(\mathcal{G}(u_{n}%
))=\lambda_{1}-\lambda_{2}\quad\text{in }Q,\qquad u_{n}=0\quad\text{on
}\partial\Omega\times(0,T),\qquad u_{n}(0)=u_{0,1}-u_{0,2}%
\]%
\[
(u_{i,n})_{t}-\text{div}(A(x,\nabla u_{i,n}))+T_{n}(\mathcal{G}(u_{i,n}%
))=\lambda_{i}\quad\text{in }Q,\qquad u_{i,n}=0\quad\text{on }\partial
\Omega\times(0,T),\qquad u_{i,n}(0)=u_{0,i},
\]
relative to the decompositions $(f_{1}-f_{2}-T_{n}(\mathcal{G}(u_{n}%
),g_{1}-g_{2},h_{1}-h_{2})$, $(f_{i}-T_{n}(\mathcal{G}(u_{i,n}),g_{i},h_{i});$
and they satisfy
\begin{align}
-||u_{0,2}||_{\infty,\Omega}-\kappa\mathbf{W}_{1,p}^{2\mathrm{diam}\Omega
}{\gamma}(x)  &  \leqq-u_{2,n}(x,t)\leqq u_{n}(x,t)\nonumber\\
&  \leqq u_{1,n}(x,t)\leqq\kappa\mathbf{W}_{1,p}^{2\mathrm{diam}\Omega}%
{\gamma}(x)+||u_{0,1}||_{\infty,\Omega}, \label{18062}%
\end{align}%
\[
\int_{Q}|T_{n}\left(  \mathcal{G}(u_{n})\right)  |\leqq\sum_{i=1,2}%
(\lambda_{i}(Q)+||u_{0,i}||_{1,\Omega}),\quad\text{and\quad}\int_{Q}%
T_{n}\left(  \mathcal{G}(u_{i,n})\right)  \leqq\lambda_{i}(Q)+||u_{0,i}%
||_{1,\Omega}.
\]
As in Lemma \ref{T26}, up to a common subsequence, $\{u_{n},u_{1,n},u_{2,n}\}$
converges $a.e.$ in $Q$ to $\{u,u_{1},u_{2}\}$ for which (\ref{17062}) is
satisfied $a.e.$ in $Q$. From (\ref{hypg}), (\ref{18062}) and the dominated
convergence Theorem, we deduce that $\left\{  T_{n}(\mathcal{G}(u_{n}%
))\right\}  $ converges to $\mathcal{G}(u)$ and $\left\{  T_{n}(\mathcal{G}%
(u_{i,n}))\right\}  $ converges to $\mathcal{G}(u_{i})$ in $L^{1}(Q)$. Thus,
from Theorem \ref{sta}, $u$ and $u_{i}$ are respective R-solutions of
(\ref{proba}) and (\ref{probb}) relative to the decompositions $(f_{1}%
-f_{2}-\mathcal{G}(u),g_{1}-g_{2},h_{1}-h_{2})$, $(f_{i}-\mathcal{G}%
(u_{i}),g_{i},h_{i}),$ and (\ref{17062}) and (\ref{18061} hold. The last
statement follows from the same assertion in Lemma \ref{T26}.\medskip
\end{proof}

\begin{proof}
[Proof of Theorem \ref{main}]By Proposition \ref{P5}, for $i=1,2,$ there exist
$f_{i,n},f_{i}\in L^{1}(Q)$, $g_{i,n},g_{i}\in(L^{p^{\prime}}(Q))^{N}$ and
$h_{i,n},h_{i}\in X,$ $\mu_{i,n,s},\mu_{i,s}\in\mathcal{M}_{s}^{+}(Q)$ such
that
\[
\mu_{i}=f_{i}-\operatorname{div}g_{i}+(h_{i})_{t}+\mu_{i,s},\qquad\mu
_{i,n}=f_{i,n}-\operatorname{div}g_{i,n}+(h_{i,n})_{t}+\mu_{i,n,s},
\]
and $\left\{  f_{i,n}\right\}  ,\left\{  g_{i,n}\right\}  ,\left\{
h_{i,n}\right\}  $ strongly converge to $f_{i},g_{i},h_{i}$ in $L^{1}%
(Q),\;(L^{p^{\prime}}(Q))^{N}$ and $X$ respectively, and $\left\{  \mu
_{i,n}\right\}  ,\left\{  \mu_{i,n,s}\right\}  $ converge to $\mu_{i}%
,\mu_{i,s}$ (strongly) in $\mathcal{M}_{b}(Q),$ and
\[
||f_{i,n}||_{1,Q}+||g_{i,n}||_{p^{\prime},Q}+||h_{i,n}||_{X}+\mu
_{i,n,s}(\Omega)\leqq2\mu(Q).
\]
By Lemma \ref{T27}, there exist R-solutions $u_{n}$, $u_{i,n}$ to problems
\begin{equation}
(u_{n})_{t}-\text{div}(A(x,\nabla u_{n}))+\mathcal{G}(u_{n})=\mu_{1,n}%
-\mu_{2,n}\quad\text{in }Q,\qquad u_{n}=0\quad\text{on }\partial\Omega
\times(0,T),~~u_{n}(0)=T_{n}(u_{0}) \label{110411}%
\end{equation}%
\begin{equation}
(u_{i,n})_{t}-\text{div}(A(x,\nabla u_{i,n}))+\mathcal{G}(u_{i,n})=\mu
_{i,n}\quad\text{in }\Omega,\qquad u_{i}=0\quad\text{on }\partial\Omega
\times(0,T),\quad u_{i,n}(0)=T_{n}(u_{0}^{\pm}), \label{110412}%
\end{equation}
for $i=1,2,$ relative to the decompositions $(f_{1,n}-f_{2,n}-\mathcal{G}%
(u_{n}),g_{1,n}-g_{2,n},h_{1,n}-h_{2,n})$, $(f_{i,n}-\mathcal{G}%
(u_{i,n}),g_{i,n},h_{i,n}),$ such that $\{u_{i,n}\}$ is nonnegative and
nondecreasing, and $-u_{2,n}\leqq u_{n}\leqq u_{1,n}$; and
\begin{equation}
\int_{Q}|\mathcal{G}(u_{n})|\leqq\mu_{1}(Q)+\mu_{2}(Q)+||u_{0}||_{1,\Omega
,}\quad\text{and\quad}\int_{Q}\mathcal{G}(u_{i,n})\leqq\mu_{i}(Q)+||u_{0}%
||_{1,\Omega},\quad i=1,2. \label{18063}%
\end{equation}
\newline As in the proof of Lemma \ref{T27}, up to a common subsequence
$\{u_{n},u_{1,n},u_{2,n}\}$ converge $a.e.$ in $Q$ to $\{u,u_{1},u_{2}\}$.
Since $\left\{  \mathcal{G}(u_{i,n})\right\}  $ is nondecreasing, and
nonnegative, from the monotone convergence Theorem and (\ref{18063}), we
obtain that $\left\{  \mathcal{G}(u_{i,n})\right\}  $ converges to
$\mathcal{G}(u_{i})$ in $L^{1}(Q)$, $i=1,2$. Finally, $\left\{  \mathcal{G}%
(u_{n})\right\}  $ converges to $\mathcal{G}(u)$ in $L^{1}(Q),$ since
$|\mathcal{G}(u_{n})|\leqq\mathcal{G}(u_{1,n})+\mathcal{G}(u_{2,n})$. Thus, we
can see that
\[
\left\{  \mu_{1,n}-\mu_{2,n}-\mathcal{G}(u_{n}),f_{1,n}-f_{2,n}-\mathcal{G}%
(u_{n}),g_{1,n}-g_{2,n},h_{1,n}-h_{2,n},\mu_{1,s,n}-\mu_{2,s,n},T_{n}%
(u_{0}^{+})-T_{n}(u_{0}^{-})\right\}
\]
is an approximation of $(\mu_{1}-\mu_{2}-\mathcal{G}(u),f_{1}-f_{2}%
-\mathcal{G}(u),g_{1}-g_{2},h_{1}-h_{2},\mu_{1,s}-\mu_{2,s},u_{0})$, in the
sense of Theorem \ref{sta}; and
\[
\left\{  \mu_{i,n}-\mathcal{G}(u_{i,n}),f_{i,n}-\mathcal{G}(u_{i,n}%
),g_{i,n},h_{i,n},\mu_{i,s,n},T_{n}(u_{0}^{\pm})\right\}
\]
is an approximation of $(\mu_{i}-\mathcal{G}(u_{i}),f_{i}-\mathcal{G}%
(u_{i}),g_{i},h_{i},\mu_{i,s},u_{0}^{\pm})$. Therefore, $u$ is a R-solution of
(\ref{pga}), and (\ref{abs}) holds if $u_{0}\in L^{\infty}(\Omega)$ and
$\omega_{n}\leq\gamma$ for any $n\in\mathbb{N}$ and some $\gamma\in
\mathcal{M}_{b}^{+}(\Omega).\medskip$
\end{proof}

As a consequence we prove Theorem \ref{main1}. We use the following result of
\cite{BiNQVe}:

\begin{proposition}
[ see \cite{BiNQVe}]\label{110413} Let $q>p-1,$ $\alpha\in\left(
0,\frac{N(q+1-p)}{pq}\right)  $, $r>0$ and $\nu\in\mathcal{M}_{b}^{+}(\Omega
)$. If $\nu$ does not charge the sets of $C_{\alpha p,\frac{q}{q+1-p}}%
$-capacity zero, there exists a nondecreasing sequence $\{\nu_{n}%
\}\subset\mathcal{M}_{b}^{+}(\Omega)$ with compact support in $\Omega$ which
converges to $\nu$ strongly in $\mathcal{M}_{b}(\Omega)$ and such that
$\mathbf{W}_{\alpha,p}^{r}[\nu_{n}]\in L^{q}(\mathbb{R}^{N})$, for any
$n\in\mathbb{N}$.
\end{proposition}

\begin{proof}
[Proof of Theorem \ref{main1}]Let $f\in L^{1}(Q)$, $u_{0}\in L^{1}(\Omega),$
and $\mu\in\mathcal{M}_{b}(Q)$ such that $\left\vert \mu\right\vert
\leqq\omega\otimes F,$ where $F\in L^{1}((0,T))$ and $\omega$ does not charge
the sets of $C_{p,\frac{q}{q+1-p}}$-capacity zero. From Proposition
\ref{110413}, there exists a nondecreasing sequence $\{\omega_{n}%
\}\subset\mathcal{M}_{b}^{+}(\Omega)$ with compact support in $\Omega$ which
converges to $\omega,$ strongly in $\mathcal{M}_{b}(\Omega),$ such that
$\mathbf{W}_{1,p}^{2diam\Omega}[\omega_{n}]\in L^{q}(\mathbb{R}^{N})$. We can
write
\begin{equation}
f+\mu=\mu_{1}-\mu_{2},\qquad\mu_{1}=f^{+}+\mu^{+},\qquad\mu_{2}=f^{-}+\mu^{-},
\label{fmu}%
\end{equation}
and $\mu^{+},\mu^{-}\leqq\omega\otimes F.$ We set
\begin{equation}
Q_{n}=\{(x,t)\in\Omega\times(\frac{1}{n},T-\frac{1}{n}):d(x,\partial
\Omega)>\frac{1}{n}\},\qquad F_{n}=T_{n}(\chi_{(\frac{1}{n}T-\frac{1}{n})}F),
\label{qn}%
\end{equation}%
\begin{equation}
\mu_{1,n}=T_{n}(\chi_{Q_{n}}f^{+})+\inf\{\mu^{+},\omega_{n}\otimes
F_{n}\},\qquad\mu_{2,n}=T_{n}(\chi_{Q_{n}}f^{-})+\inf\{\mu^{-},\omega
_{n}\otimes F_{n}\}. \label{fntn}%
\end{equation}
Then $\left\{  \mu_{1,n}\right\}  ,\left\{  \mu_{2,n}\right\}  $ are
nondecreasing sequences with compact support in $Q,$ and $\mu_{1,n},\mu
_{2,n}\leqq\tilde{\omega}_{n}\otimes\chi_{(0,T)},$ with $\tilde{\omega}%
_{n}=n(\chi_{\Omega}+\omega_{n})$ and $(n+\kappa\mathbf{W}_{1,p}^{2diam\Omega
}[\omega_{n}])^{q}\in L^{1}(Q)$. Besides, $\omega_{n}\otimes F_{n}$ converges
to $\omega\otimes F$ strongly in $\mathcal{M}_{b}(Q):$ indeed we easily check
that
\[
||\omega_{n}\otimes F_{n}-\omega\otimes F||_{\mathcal{M}_{b}(Q)}\leqq
||F_{n}||_{L^{1}((0,T))}||\omega_{n}-\omega||_{\mathcal{M}_{b}(\Omega
)}+||\omega||_{\mathcal{M}_{b}(\Omega)}||F_{n}-F||_{L^{1}((0,T))}%
\]
Observe that for any measures $\nu,\theta,\eta\in\mathcal{M}_{b}(Q),$ there
holds
\[
\left\vert \inf\{\nu,\theta\}-\inf\{\nu,\eta\}\right\vert \leqq\left\vert
\theta-\eta\right\vert ,
\]
hence $\left\{  \mu_{1,n}\right\}  ,\left\{  \mu_{2,n}\right\}  $ converge to
$\mu_{1},\mu_{2}$ respectively in $\mathcal{M}_{b}(Q)$. Therefore, the result
follows from Theorem \ref{main}.
\end{proof}

\begin{remark}
\label{pari} Our result improves the existence results of \cite{PePoPor},
where $\mu\in\mathcal{M}_{0}(Q).$ Indeed, let $p_{e}=N(p-1)/(N-p)$ be the
critical exponent for the elliptic problem%
\[
-\Delta_{p}w+\left\vert w\right\vert ^{q-1}w=\omega\quad\text{in }%
\Omega,\qquad w=0\quad\text{on }\partial\Omega.
\]
Notice that $p_{c}<p_{e},$ since $p>p_{1}.$ If $q\geqq p_{e},$ there exist
measures $\omega\in\mathcal{M}_{b}^{+}(\Omega)$ which do not charge the sets
of $C_{p,\frac{q}{q+1-p}}$-capacity zero, such that $\omega\not \in
\mathcal{M}_{0,e}(\Omega).$ Then for any $F\in L^{1}((0,T)),$ $F\geqq
0,F\not \equiv 0,$ we have $\omega\otimes F\not \in \mathcal{M}_{0}(Q).$
\end{remark}

\begin{remark}
Let $A:\Omega\times\mathbb{R}^{N}\longmapsto\mathbb{R}^{N}$ satisfying
(\ref{condi3}),(\ref{condi4}). Let $\mathcal{G}:Q\times\mathbb{R}%
\rightarrow\mathbb{R}$ be a Caratheodory function such that the map
$s\mapsto\mathcal{G}(x,t,s)$ is nondecreasing and odd, for $a.e.$ $(x,t)$ in
$Q$. Assume that $\omega\in\mathcal{M}_{0,e}(\Omega)$. Thus, we have
$\omega(\{x:W_{1,p}^{2diam(\Omega)}[\omega](x)=\infty\})=0$. As in the proof
of Theorem \ref{main1} with $\omega_{n}=\chi_{W_{1,p}^{2diam\Omega}%
[\omega]\leqq n}\omega$, we get that (\ref{pga}) has a R-solution.
\end{remark}

\begin{remark}
\label{exten}As in \cite{BiNQVe}, from Theorem \ref{main}, we can extend
Theorem \ref{main1} given for $\mathcal{G}(u)=\left\vert u\right\vert ^{q-1}%
u$, to the case of a function $\mathcal{G}(x,t,.),$ odd for $a.e.$ $(x,t)\in
Q,$ such that
\[
|\mathcal{G}(x,t,u)|\leqq G(|u|),\qquad\int_{1}^{\infty}G(s)s^{-q-1}ds<\infty,
\]
where $G$ is a nondecreasing continuous, under the condition that $\omega$
does not charge the sets of zero $C_{_{p},\frac{q}{q-p+1},1}$-capacity, where
for any Borel set $E\subset\mathbb{R}^{N},$
\[
C_{_{p},\frac{q}{q-p+1},1}(E)=\inf\{||\varphi||_{L^{\frac{q}{q-p+1}%
,1}(\mathbb{R}^{N})}:\varphi\in L^{\frac{q}{q-p+1},1}(\mathbb{R}^{N}),\quad
G_{p}\ast\varphi\geqq\chi_{E}\}
\]
where $L^{\frac{q}{q-p+1},1}(\mathbb{R}^{N})$ is the Lorentz space of order
$(q/(q-p+1),1)$.
\end{remark}

In case $\mathcal{G}$ is of exponential type, we introduce the notion of
maximal fractional operator, defined for any $\eta\geqq0,$ $R>0,$ $x_{0}%
\in\mathbb{R}^{N}$ by
\[
\mathbf{M}_{p,R}^{\eta}[\omega](x_{0})=\sup_{t\in\left(  0,R\right)  }%
\frac{\omega(B(x_{0},t))}{t^{N-p}h_{\eta}(t)},\qquad\text{where }h_{\eta
}(t)=\inf((-\ln t)^{-\eta},(\ln2)^{-\eta})).
\]
We obtain the following:

\begin{theorem}
\label{expo} Let $A:\Omega\times\mathbb{R}^{N}\longmapsto\mathbb{R}^{N}$
satisfying (\ref{condi3}),(\ref{condi4}). Let $p<N$ and $\tau>0,\beta>1,\mu
\in\mathcal{M}_{b}(Q)$ and $u_{0}\in L^{1}(\Omega)$. Assume that $\left\vert
\mu\right\vert \leqq\omega\otimes F,$ with $\omega\in\mathcal{M}_{b}%
^{+}(\Omega)$, $F\in L^{1}((0,T))$ be nonnegative. Assume that one of the
following assumptions is satisfied:

\noindent(i) $||F||_{L^{\infty}((0,T))}\leqq1$ and for some $M_{0}%
=M_{0}(N,p,\beta,\tau,c_{3},c_{4},\mathrm{diam}\Omega),$
\begin{equation}
||\mathbf{M}_{p,2\mathrm{diam}_{\Omega}}^{\frac{{p-1}}{\beta^{\prime}}}%
[\omega]|{|_{{L^{\infty}}({\mathbb{R}^{N}})}}<M_{0},\label{plou}%
\end{equation}
(ii) there exists $\beta_{0}>\beta$ such that $\mathbf{M}_{p,2diam\Omega
}^{\frac{{p-1}}{\beta_{0}^{\prime}}}[\omega]\in L^{\infty}(\mathbb{R}^{N}).$

\medskip Then there exists a R-solution to the problem
\[
\left\{
\begin{array}
[c]{l}%
u_{t}-\text{div}(A(x,\nabla u))+(e^{\tau|u|^{\beta}}-1)\mathrm{sign}%
u=F+\mu\qquad\text{in }Q,\\
u=0\;\qquad\text{on }\partial\Omega\times(0,T),\\
u(0)=u_{0}\qquad\text{in }\Omega.
\end{array}
\right.
\]

\end{theorem}

For the proof we use the following result of \cite{BiNQVe}:

\begin{proposition}
[see \cite{BiNQVe}, Theorem 2.4]\label{majex}Suppose $1<p<N.$ Let $\nu
\in\mathcal{M}_{b}^{+}(\Omega),$ $\beta>1,$ and $\delta_{0}=((12\beta
)^{-1})^{\beta}p\ln2.$ There exists $C=C(N,p,\beta,\mathrm{diam}\Omega)$ such
that, for any $\delta\in\left(  0,\delta_{0}\right)  $,%
\[
\int_{\Omega}\exp\left(  \delta\frac{(\mathbf{W}_{1,p}^{2\mathrm{diam}\Omega
}[\nu])^{\beta}}{||\mathbf{M}_{p,2\mathrm{diam}_{\Omega}}^{\frac{{p-1}}%
{\beta^{\prime}}}[\nu]|{|_{{L^{\infty}}({\mathbb{R}^{N}})}^{\frac{\beta}{p-1}%
}}}\right)  \leqq\frac{C}{\delta_{0}-\delta}.
\]

\end{proposition}

\begin{proof}
[Proof of Theorem \ref{expo}]Let $Q_{n}$ be defined at (\ref{qn}), and
$\omega_{n}=\omega\chi_{\Omega_{n}},$ where $\Omega_{n}=\{x\in\Omega
:d(x,\partial\Omega)>1/n\}.$ We still consider $\mu_{1},\mu_{2},F_{n}%
,\mu_{1,n},\mu_{2,n}$ as in (\ref{fmu}), (\ref{fntn}).\newline Case 1: Assume
that $||F||_{L^{\infty}((0,T))}\leqq1$ and (\ref{plou}) holds. We have
$\mu_{1,n},\mu_{2,n}\leqq n\chi_{\Omega}+\omega$. For any $\varepsilon>0,$
there exists $c_{\varepsilon}=c_{\varepsilon}(\varepsilon,N,p,\beta
,\kappa,\mathrm{diam}\Omega)$ $>0$ such that
\[
(n+\kappa\mathbf{W}_{1,p}^{2\mathrm{diam}\Omega}[n\chi_{\Omega}+\omega
])^{\beta}\leqq c_{\varepsilon}n^{\frac{\beta p}{p-1}}+(1+\varepsilon
)\kappa^{\beta}(\mathbf{W}_{1,p}^{2\mathrm{diam}\Omega}[\omega])^{\beta}%
\]
$a.e.$ in $\Omega$. Thus,
\[
\exp\left(  \tau(n+\kappa\mathbf{W}_{1,p}^{2\mathrm{diam}\Omega}[n\chi
_{\Omega}+\omega])^{\beta}\right)  \leqq\exp\left(  \tau c_{\varepsilon
}n^{\frac{\beta p}{p-1}}\right)  \exp\left(  \tau(1+\varepsilon)\kappa^{\beta
}(\mathbf{W}_{1,p}^{2\mathrm{diam}\Omega}[\omega])^{\beta}\right)
\]
If (\ref{plou}) holds with $M_{0}=\left(  \delta_{0}/\tau\kappa^{\beta
}\right)  ^{(p-1)/\beta}$ then we can chose $\varepsilon$ such that
\[
\tau(1+\varepsilon)\kappa^{\beta}||\mathbf{M}_{p,2\mathrm{diam}_{\Omega}%
}^{\frac{{p-1}}{\beta^{\prime}}}[\nu]|{|_{{L^{\infty}}({\mathbb{R}^{N}}%
)}^{\frac{\beta}{p-1}}<}\delta_{0}.
\]
From Proposition \ref{majex}, we get $\exp(\tau(1+\varepsilon)\kappa^{\beta
}\mathbf{W}_{1,p}^{2\mathrm{diam}\Omega}[\omega])^{\beta})\in L^{1}(\Omega),$
which implies $\exp(\tau(n+\kappa^{\beta}\mathbf{W}_{1,p}^{2\mathrm{diam}%
\Omega}[n\chi_{\Omega}+\omega])^{\beta})\in L^{1}(\Omega)$ for all $n$. We
conclude from Theorem \ref{main}.

\noindent Case 2: Assume that there exists $\varepsilon>0$ such that
$\mathbf{M}_{p,2diam\Omega}^{(p-1)/(\beta+\varepsilon)^{\prime}}[\omega]\in
L^{\infty}(\mathbb{R}^{N})$. Now we use the inequality $\mu_{1,n},\mu
_{2,n}\leqq n(\chi_{\Omega}+\omega)$. For any $\varepsilon>0$ and
$n\in\mathbb{N}$ there exists $c_{\varepsilon,n}>0$ such that
\[
(n+\kappa^{\beta}\mathbf{W}_{1,p}^{2\mathrm{diam}\Omega}[n(\chi_{\Omega
}+\omega)])^{\beta}\leqq c_{\varepsilon,n}+\varepsilon(\mathbf{W}%
_{1,p}^{2\mathrm{diam}\Omega}[\omega])^{\beta_{0}}%
\]
Thus, from Proposition \ref{majex} we get $\exp(\tau(n+\kappa^{\beta
}\mathbf{W}_{1,p}^{2\mathrm{diam}\Omega}[n(\chi_{\Omega}+\omega)])^{\beta})\in
L^{1}(\Omega)$ for all $n$. We conclude from Theorem \ref{main}.
\end{proof}

\subsection{Equations with source term}

As a consequence of Theorem \ref{main}, we get a first result for problem
(\ref{pmu}):

\begin{corollary}
\label{TH5} Let $A:\Omega\times\mathbb{R}^{N}\longmapsto\mathbb{R}^{N}$
satisfying (\ref{condi3})(\ref{condi4}). Let $u_{0}\in L^{\infty}(\Omega),$
and $\mu\in\mathcal{M}_{b}(${$Q$}$)$ such that $\left\vert \mu\right\vert
\leqq\omega\otimes\chi_{(0,T)}$ for some $\omega\in\mathcal{M}_{b}^{+}%
(\Omega)$. Then there exist a R-solution u of (\ref{pmu}), such that
\begin{equation}
\left\vert u(x,t)\right\vert \leqq\kappa\mathbf{W}_{1,p}^{2\mathrm{diam}%
(\Omega)}[\omega](x)+||u_{0}||_{\infty,\Omega},\qquad\text{for }a.e.~(x,t)\in
Q,\label{09041}%
\end{equation}
where $\kappa$ is defined at Theorem \ref{elli}.
\end{corollary}

\begin{proof}
Let $\left\{  \phi_{n}\right\}  $ be a nonnegative, nondecreasing sequence in
$C_{c}^{\infty}(Q)$ which converges to $1,$ $a.e.$ in $Q.$ Since $\{\phi
_{n}\mu^{+}\},\{\phi_{n}\mu^{-}\}$ are nondecreasing sequences, the result
follows from Theorem \ref{main}.
\end{proof}

Our proof of Theorem \ref{120410} is based on a property of W\"{o}lf potentials:

\begin{theorem}
[see \cite{PhVe1}]\label{12042}Let $q>p-1$, $0<p<N$, $\omega\in\mathcal{M}%
_{b}^{+}(\Omega)$. If for some $\lambda>0,$
\begin{equation}
\omega(E)\leqq\lambda C_{_{p},\frac{q}{p-q+1}}(E)\quad\quad\text{for any
compact set }E\subset\mathbb{R}^{N}, \label{cont}%
\end{equation}
then $(\mathbf{W}_{1,p}^{2\mathrm{diam}\Omega}[\omega])^{q}\in L^{1}(\Omega),$
and there exists $M=M(N,p,q,\mathrm{diam}(\Omega))$ such that, $a.e.$ in
$\Omega,$
\begin{equation}
\mathbf{W}_{1,p}^{2\mathrm{diam}\Omega}\left[  \mathbf{W}_{1,p}%
^{2\mathrm{diam}\Omega}[\omega]\right]  ^{q}\leqq M\lambda^{\frac
{q-p+1}{(p-1)^{2}}}\mathbf{W}_{1,p}^{2\mathrm{diam}\Omega}[\omega]<\infty.
\label{12044}%
\end{equation}

\end{theorem}

We deduce the following:

\begin{lemma}
\label{12047}Let $\omega\in$ $\mathcal{M}_{b}^{+}(\Omega)$, and $b\geqq0$ and
$K>0$. Suppose that $\{u_{m}\}_{m\geqq1}$ is a sequence of nonnegative
functions in $\Omega$ that satisfies
\[
{u_{1}}\leqq K\mathbf{W}_{1,p}^{2\mathrm{diam}\Omega}[\omega]+b,\qquad
{u_{m+1}}\leqq K\mathbf{W}_{1,p}^{2\mathrm{diam}\Omega}[u_{m}^{q}%
+\omega]+b\qquad\forall m\geqq1.
\]
Assume that $\omega$ satisfies (\ref{cont}) for some $\lambda>0.$ Then there
exist $\lambda_{0}$ and $b_{0},$ depending on $N,p,q,K,$ and ${\mathrm{diam}%
\Omega,}$ such that, if $\lambda\leqq\lambda_{0}$ and $b\leqq b_{0}$, then
$\mathbf{W}_{1,p}^{2\mathrm{diam}\Omega}[\mu]\in L^{q}(\Omega)$ and for any
$m\geqq1,$
\begin{equation}
{u_{m}}\leqq2\beta_{p}K\mathbf{W}_{1,p}^{2\mathrm{diam}\Omega}[\omega
]+2b,\qquad\beta_{p}=\max({1,{3^{\frac{{2-p}}{{p-1}}})}}. \label{12043}%
\end{equation}

\end{lemma}

\begin{proof}
Clearly, (\ref{12043}) holds for $m=1$. Now, assume that it holds at the order
$m$. Then
\[
u_{m}^{q}\leqq2^{q-1}{\left(  {2}\beta_{p}\right)  ^{q}}{K^{q}(}%
{\mathbf{W}{_{1,p}^{2\mathrm{diam}\Omega}[\omega]})^{q}}+2^{q-1}{\left(
{2b}\right)  ^{q}}%
\]
Using (\ref{12044}) we get
\begin{align*}
u_{m+1}  &  \leq K\mathbf{W}_{1,p}^{2\mathrm{diam}\Omega}\left[
{2^{q-1}{{\left(  {2}\beta_{p}\right)  }^{q}}K^{q}({{{W_{1,p}^{2\mathrm{diam}%
\Omega}[\omega])}}^{q}}+2^{q-1}{{\left(  {2b}\right)  }^{q}}+\omega}\right]
+b\\
&  \leq\beta_{p}K\left(  {{A_{1}}\mathbf{W}_{1,p}^{2\mathrm{diam}\Omega
}\left[  ({{{{W_{1,p}^{2\mathrm{diam}\Omega}[\omega])}}^{q}}}\right]
+\mathbf{W}_{1,p}^{2\mathrm{diam}\Omega}\left[  {{{\left(  {2b}\right)  }^{q}%
}}\right]  +W_{1,p}^{2\mathrm{diam}\Omega}[\omega]}\right)  +b\\
&  \leq\beta_{p}K(A_{1}M\lambda^{\frac{{q-p+1}}{{{{(p-1)}^{2}}}}}%
+1)\mathbf{W}_{1,p}^{2\mathrm{diam}\Omega}[\omega]+\beta_{p}K\mathbf{W}%
_{1,p}^{2\mathrm{diam}\Omega}\left[  {{{\left(  {2b}\right)  }^{q}}}\right]
+b\\
&  =\beta_{p}K(A_{1}M\lambda^{\frac{{q-p+1}}{{{{(p-1)}^{2}}}}}+1)\mathbf{W}%
_{1,p}^{2\mathrm{diam}\Omega}[\omega]+A_{2}b^{\frac{q}{p-1}}+b,
\end{align*}
where $M$ is as in (\ref{12044}) and $A_{1}=\left(  2^{q-1}{{\left(  {2}%
\beta_{p}\right)  }^{q}}K^{q}\right)  ^{1/(p-1)}$, $A_{2}=\beta_{p}%
K2^{q/(p-1)}|{B_{1}}{|}^{1/(p-1)}{(}p^{\prime})^{-1}{\left(  {2\mathrm{diam}%
\Omega}\right)  ^{p^{\prime}}}$.\newline Thus, (\ref{12043}) holds for $m=n+1$
if we prove that
\[
A_{1}M\lambda^{\frac{{q-p+1}}{{{{(p-1)}^{2}}}}}\leq1\text{ and }A_{2}%
b^{\frac{q}{p-1}}\leq b,
\]
which is equivalent to
\[
\lambda\leq(A_{1}M)^{-\frac{(p-1)^{2}}{q-p+1}}\text{ and }b\leq A_{2}%
^{-\frac{p-1}{q-p+1}}.
\]
Therefore, we obtain the result with $\lambda_{0}=(A_{1}M)^{-(p-1)^{2}%
/(q-p+1)}$ and $b_{0}=A_{2}^{-(p-1)/(q-p+1)}.$
\end{proof}

\medskip

\begin{proof}
[Proof of Theorem \ref{120410}]From Corollary \ref{051120131} and \ref{TH5},
we can construct a sequence of nonnegative nondecreasing R-solutions
$\{u_{m}\}_{m\geqq1}$ defined in the following way: $u_{1}$ is a R-solution of
(\ref{pmu}), and $u_{m+1}$ is a nonnegative R-solution of
\[
\left\{
\begin{array}
[c]{l}%
(u_{m+1})_{t}-\text{div}(A(x,\nabla u_{m+1}))=u_{m}^{q}+\mu\qquad\text{in
}Q,\\
u_{m+1}=0\qquad\text{on }\partial\Omega\times(0,T),\\
u_{m+1}(0)=u_{0}\qquad\text{in }\Omega.
\end{array}
\right.
\]
Setting $\overline{u}_{m}=\sup_{t\in(0,T)}u_{m}(t)$ for all $m\geqq1,$ there
holds
\[
{\overline{u}_{1}}\leqq\kappa\mathbf{W}_{1,p}^{2\mathrm{diam}\Omega}%
[\omega]+||u_{0}||_{\infty,\Omega},\qquad{\overline{u}_{m+1}}\leqq
\kappa\mathbf{W}_{1,p}^{2\mathrm{diam}\Omega}[\overline{u}_{m}^{q}%
+\omega]+||u_{0}||_{\infty,\Omega}\quad\quad\forall m\geqq1.
\]
From Lemma \ref{12047}, we can find $\lambda_{0}=\lambda_{0}(N,p,q,$%
\textrm{$diam$}$\Omega)$ and $b_{0}=b_{0}(N,p,q,\mathrm{diam}\Omega)$ such
that if (\ref{051120132}) is satisfied with $\lambda_{0}$ and $b_{0}$, then
\begin{equation}
u_{m}\leqq{\overline{u}_{m}}\leqq2\beta_{p}\kappa\mathbf{W}_{1,p}%
^{2\mathrm{diam}\Omega}[\omega]+2||u_{0}||_{\infty,\Omega}\quad\quad\forall
m\geqq1. \label{12049}%
\end{equation}
Thus $\left\{  u_{m}\right\}  $ converges $a.e.$ in $Q$ and in $L^{1}(Q)$ to
some function $u,$ for which (\ref{maw}) is satisfied in $\Omega$ with
$c=2\beta_{p}\kappa.$ Finally, one can apply Theorem \ref{sta} to the sequence
of measures $\left\{  u_{m}^{q}+\mu\right\}  ,$ and obtain that $u$ is a
R-solution of (\ref{pro3}).\medskip
\end{proof}

Next we consider the exponential case.

\begin{theorem}
\label{MTH1} Let $A:\Omega\times\mathbb{R}^{N}\longmapsto\mathbb{R}^{N}$
satisfying (\ref{condi3}),(\ref{condi4}). Let $\tau>0,l\in\mathbb{N}$ and
$\beta\geqq1$ such that $l\beta>p-1.$ Set
\begin{equation}
\mathcal{E}(s)=e^{s}-\sum\limits_{j=0}^{l-1}{\frac{{{s^{j}}}}{{j!}}}%
,\qquad\forall s\in\mathbb{R}.\label{ess}%
\end{equation}
Let $\mu\in\mathcal{M}_{b}^{+}(Q)$, $\omega\in\mathcal{M}_{b}^{+}(\Omega)$
such that $\mu\leqq\chi_{(0,T)}\otimes\omega$. Then, there exist $b_{0}$ and
$M_{0}$ depending on $N,p,\beta,\tau,l$ and $\mathrm{diam}\Omega,$ such that
if
\[
||\mathbf{M}_{p,2\mathrm{diam}\Omega}^{\frac{{(p-1)(\beta-1)}}{\beta}}%
[\omega]|{|_{{L^{\infty}}({\mathbb{R}^{N}})}}\leqq M_{0},\qquad||u_{0}%
||_{\infty,\Omega}\leqq b_{0},
\]
the problem
\begin{equation}
\left\{
\begin{array}
[c]{l}%
u_{t}-\text{div}(A(x,\nabla u))=\mathcal{E}(\tau u^{\beta})+\mu\qquad\text{in
}Q,\\
u=0\qquad\text{on }\partial\Omega\times(0,T),\\
u(0)=u_{0}\qquad\text{in }\Omega
\end{array}
\right.  \label{pro2}%
\end{equation}
admits nonnegative R- solution $u$, which satisfies, $a.e.$ in $Q,$ for some
$c,$ depending on $N,p,c_{3},c_{4}$
\begin{equation}
{u(x,t)}\leqq c\mathbf{W}_{1,p}^{2\mathrm{diam}\Omega}[\omega](x)+2{b_{0}%
}.\label{1334}%
\end{equation}

\end{theorem}

For the proof we first recall an approximation property, which is a
consequence of \cite[Theorem 2.5]{NQVe}:

\begin{theorem}
\label{TH3} Let $\tau>0$, $b\geqq0$, $K>0$, $l\in\mathbb{N}$ and $\beta\geqq1$
such that $l\beta>p-1$. Let $\mathcal{E}$ be defined by (\ref{ess}). Let
$\{v_{m}\}$ be a sequence of nonnegative functions in $\Omega$ such that, for
some $K>0,$
\[
v{_{1}}\leqq K\mathbf{W}_{1,p}^{2\mathrm{diam}\Omega}[\mu]+b,\qquad v{_{m+1}%
}\leqq K\mathbf{W}_{1,p}^{2\mathrm{diam}\Omega}[\mathcal{E}(\tau u_{m}^{\beta
})+\mu]+b,\quad\forall m\geqq1.
\]
Then, there exist $b_{0}$ and $M_{0},$ depending on $N,p,\beta,\tau,l,K$ and
$\mathrm{diam}\Omega$ such that if $b\leqq b_{0}$ and
\begin{equation}
||\mathbf{M}_{p,{2\mathrm{diam}\Omega}}^{\frac{{(p-1)(\beta-1)}}{\beta}}%
[\mu]|{|_{\infty,{\mathbb{R}^{N}}}}\leqq M_{0}, \label{mai}%
\end{equation}
then, setting $c_{p}=2{{{{{\max({1,}2{{^{\frac{{2-p}}{{p-1}}}),}}}}}}}$%
\[
{\exp(\tau{{{({Kc_{p}\mathbf{W}_{1,p}^{2\mathrm{diam}\Omega}[\mu]+2{b_{0})}}%
}^{\beta}})}\in L^{1}(\Omega),}%
\]%
\begin{equation}
{v_{m}}\leqq Kc_{p}W_{1,p}^{2\mathrm{diam}\Omega}[\mu]+2{b_{0}},\quad\forall
m\geqq1.\, \label{1232}%
\end{equation}

\end{theorem}

\begin{proof}
[Proof of Theorem \ref{MTH1}]From Corollary \ref{051120131} and \ref{TH5} we
can construct a sequence of nonnegative nondecreasing R-solutions
$\{u_{m}\}_{m\geqq1}$ defined in the following way: $u_{1}$ is a R-solution of
problem (\ref{pmu}), and by induction, $u_{m+1}$ is a R-solution of
\[
\left\{
\begin{array}
[c]{l}%
(u_{m+1})_{t}-\text{div}(A(x,\nabla u_{m+1}))=\mathcal{E}(\tau u_{m}^{\beta
})+\mu\qquad\text{in }Q,\\
u_{m+1}=0\qquad\text{on }\partial\Omega\times(0,T),\\
u_{m+1}(0)=u_{0}\qquad\text{in }\Omega.
\end{array}
\right.
\]
And, setting $\overline{u}_{m}=\sup_{t\in(0,T)}u_{m}(t),$ there holds
\[
{\overline{u}_{1}}\leqq{\kappa}W_{1,p}^{2\mathrm{diam}\Omega}[\omega
]+||u_{0}||_{\infty,\Omega},\qquad{\overline{u}_{m+1}}\leqq\kappa
W_{1,p}^{2\mathrm{diam}\Omega}[\mathcal{E}(\tau\overline{u}_{m}^{\beta
})+\omega]+||u_{0}||_{\infty,\Omega},\quad\quad\forall m\geqq1.
\]
Thus, from Theorem \ref{TH3}, there exist $b_{0}\in(0,1]$ and $M_{0}>0$
depending on $N,p,\beta,\tau,l$ and $\mathrm{diam}\Omega$ such that, if
(\ref{mai}) holds, then (\ref{1232}) is satisfied with $v_{m}=\overline{u}%
_{m}$. As a consequence, $u_{m}$ is well defined. Thus, $\left\{
u_{m}\right\}  $ converges $a.e.$ in $Q$ to some function $u,$ for which
(\ref{1334}) is satisfied in $\Omega$. Furthermore, $\left\{  \mathcal{E}(\tau
u_{m}^{\beta})\right\}  $ converges to $\mathcal{E}(\tau u^{\beta})$ in
$L^{1}(Q)$. Finally, one can apply Theorem \ref{sta} to the sequence of
measures $\left\{  \mathcal{E}(\tau u_{m}^{\beta})+\mu\right\}  ,$ and obtain
that $u$ is a R-solution of (\ref{pro2}).
\end{proof}

\section{Appendix}

\begin{proof}
[Proof of Lemma \ref{integ}]Let $\mathcal{J}{\ }$be defined by (\ref{lam}).
Let $\zeta\in C_{c}^{1}([0,T))$ with values in $[0,1],$ such that $\zeta
_{t}\leqq0$, and $\varphi=\zeta\xi\lbrack j{(S(v))}]_{l}$. Clearly,
$\varphi\in X\cap L^{\infty}(Q)$; we choose the pair of functions
$(\varphi,S)$ as test function in (\ref{renor}). Thanks to convergence
properties of Steklov time-averages, we easily will obtain (\ref{parti}) if we
prove that
\[
\lim_{\overline{l\rightarrow0,\zeta\rightarrow1}}(-\int_{Q}{{{\left(
{\zeta\xi{{\left[  {j(S(v))}\right]  }_{l}}}\right)  }_{t}}S(v))}\geqq
-\int_{Q}{{\xi_{t}}J(S(v)).}%
\]
We can write $-\int_{Q}{{{\left(  {\zeta\xi{{\left[  {j(S(v))}\right]  }_{l}}%
}\right)  }_{t}}S(v)}=F+G,$ with
\[
F=-\int_{Q}{(\zeta\xi)_{t}}{{\left[  {j(S(v))}\right]  }_{l}}S(v),\qquad
G=-\int_{Q}{\zeta\xi S(v)\frac{1}{l}\left(  {j(S(v))(x,t+l)-j(S(v))(x,t)}%
\right)  .}%
\]
\newline Using (\ref{222}) and integrating by parts we have
\begin{align*}
G  &  \geqq-\int_{Q}{\zeta\xi\frac{1}{l}\left(  \mathcal{J}{(S(v))(x,t+l)-}%
\mathcal{J}{(S(v))(x,t)}\right)  }\\
&  =-\int_{Q}{\zeta\xi\frac{\partial}{{\partial t}}\left(  {{{\left[
\mathcal{J}{(S(v))}\right]  }_{l}}}\right)  }=\int_{Q}{{(\zeta\xi)_{t}%
}{{\left[  \mathcal{J}{(S(v))}\right]  }_{l}}}+\int_{\Omega}{\zeta(0)\xi
(0)}{\left[  \mathcal{J}{(S(v))}\right]  _{l}}(0)\\
&  \geqq\int_{Q}{{(\zeta\xi)_{t}}{{\left[  \mathcal{J}{(S(v))}\right]  }_{l}}%
},
\end{align*}
since $\mathcal{J}{(S(v))}$ $\geqq0.$ Hence,
\[
-\int_{Q}{{{\left(  {\zeta\xi{{\left[  {j(S(v))}\right]  }_{l}}}\right)  }%
_{t}}S(v)}\geqq\int_{Q}{{(\zeta\xi)_{t}}{{\left[  \mathcal{J}{(S(v))}\right]
}_{l}}}+F=\int_{Q}{{(\zeta\xi)_{t}}\left(  {{{\left[  \mathcal{J}%
{(S(v))}\right]  }_{l}}-{{\left[  J{(S(v))}\right]  }_{l}}S(v)}\right)  }%
\]
Otherwise, $\mathcal{J}(S(v))$ and $J(S(v)\in C(\left[  0,T\right]  {;L}%
^{1}{(\Omega))}$, thus $\left\{  {(\zeta\xi)_{t}}\left(  {{{\left[
\mathcal{J}{(S(u))}\right]  }_{l}}-{{\left[  J{(S(u))}\right]  }_{l}}%
S(u)}\right)  \right\}  $ converges to $-{(\zeta\xi)_{t}}J(S(u))$ in $L^{1}%
(${$Q$}$)$ as $l\rightarrow0$. Therefore,
\begin{align*}
\lim_{\overline{l\rightarrow0,\zeta\rightarrow1}}({-}\int_{Q}{{{{\left(
{\zeta\xi{{[J(S(v))]}_{l}}}\right)  }_{t}}S(v))}}  &  \geqq\lim_{\overline
{\zeta\rightarrow1}}\left(  {-}\int_{Q}{{{{\left(  {\zeta\xi}\right)  }_{t}%
}J(S(v))}}\right) \\
&  \geqq\lim_{\overline{\zeta\rightarrow1}}\left(  {-}\int_{Q}{{\zeta\xi
_{t}J(S(v))}}\right)  =-\int_{Q}{{\xi_{t}J}(S(v))},
\end{align*}
which achieves the proof.
\end{proof}


\begin{thebibliography}{99}                                                                                               %


\bibitem {AdPo}Adams D. and Polking J., \textit{The equivalence of two
definitions of the capacity}, Proc. A.M.S., 37 (1973), 529-534.

\bibitem {AMST}Andreu F., Mazon J.M., Segura de Leon S. and Toledo J.
\textit{\ Existence and uniqueness for a degenerate parabolic equation with }
$L^{1}$\textit{\ data}, Trans. Amer. Math. Soc. 351(1999), 285-306 .

\bibitem {An}Andreu F., Mazon J., Segura de Leon S. and Toledo J. ,Quasilinear
elliptic and parabolic equations in $L^{1}$ with nonlinear boundary
conditions, Adv. Math. Sci. Appl. 7 (1997), 183-213.

\bibitem {AndSbWi}Andreianov B., Sbihi K., and Wittbold P., On uniqueness and
existence of entropy solutions for a nonlinear parabolic problem with
absorption, J. Evol. Equ. 8 (2008), 449-490.

\bibitem {BaPi}Baras P. and Pierre M., \textit{Singularit\'{e}s
\'{e}liminables pour des \'{e}quations semi-lin\'{e}aires}, Ann. Inst.
Fourier, 34 (1984), 185-206.

\bibitem {BaPi1}Baras P. and Pierre M., \textit{Crit\`{e}re d'existence de
solutions positives pour des \'{e}quations semi-lin\'{e}aires non monotones},
Ann. I.H.P. 2 (1985), 185-212.

\bibitem {BaPi2}Baras P. and Pierre M., \textit{Probl\`{e}mes paraboliques
semi-lin\'{e}aires avec donn\'{e}es mesures}, Applicable Anal. 18 (1984), 111-149.

\bibitem {Be0}Di Benedetto E., \textit{\ On the local behaviour of solutions
of degenerate parabolic equations with measurable coefficients,} Ann. Sc.
Norm. Sup. Pisa, 13 (1996), 487-535.

\bibitem {Be}Di Benedetto E., \textit{\ Degenerate Parabolic Equations},
Springer-Verlag (1993).

\bibitem {BH1}Di Benedetto E. and Herrero M.A.,\textit{\ On the Cauchy problem
and initial trace for a degenerate parabolic equation}, Trans. Amer. Math.
Soc. 314 (1989), 187-223.

\bibitem {BH2}Di Benedetto E. and Herrero M.A.,\textit{\ Nonnegative solutions
of the evolution p-Laplacian equation. Initial trace and Cauchy problem when
}$1<p<2$, Arch. Rat. Mech. Anal. 111 (1990), 225-290.

\bibitem {CdB}Di Benedetto E. and Chen Y. Z.,\textit{\ On the local behaviour
of solutions of singular parabolic equation}, Arch. Rat. Mech. Anal. 107
(1989), 293-324.

\bibitem {BBGGPV}Benilan P., Boccardo L., Gallouet T., Gariepy R., Pierre M.
and V\'{a}zquez J., \textit{An }$L1$\textit{-theory of existence and
uniqueness of solutions of nonlinear elliptic equations}, Ann. Scuola Norm.
Sup. Pisa Cl. Sci. (4) 22 (1995), no. 2, 241--273.

\bibitem {Bi1}Bidaut-V\'{e}ron M.F., \textit{Removable singularities and
existence for a quasilinear equation}, Adv. Nonlinear Studies 3 (2003), 25-63.

\bibitem {Bi2}Bidaut-V\'{e}ron M.F., \textit{Necessary conditions of existence
for an elliptic equation with source term and measure data involving the }%
$p$\textit{-Laplacian, } E.J.D.E., Conference 08 (2002), 23--34.

\bibitem {BiChVe}Bidaut-V\'{e}ron M.F., Chasseigne E. and V\'{e}ron L.
\textit{Initial trace of solutions odf some quasilinear parabolic equations
with absorption,} J. Funct. Anal. 193 (2002), 140-205.

\bibitem {BiNQVe}Bidaut-V\'{e}ron M.F., Nguyen Quoc H. and V\'{e}ron L.,
\textit{Quasilinear Emden-Fowler equations with absorption terms and measure
data}, Arxiv 1212.6314, to appear in J. Math. Pures Appl.

\bibitem {Bl}Blanchard D.,\textit{\ Truncations and monotonicity methods for
parabolic equations}, Nonlinear Anal. T., M. \& A. 21 (1993), 725-743.

\bibitem {BlMu}Blanchard D. and Murat F.,\textit{\ Renormalized solutions of
nonlinear parabolic equation with }$L^{1}$\textit{\ data: existence and
uniqueness}, Proc. Roy. Soc. Edinburgh 127A (1997), 1153-1179.

\bibitem {BlPeRe}Blanchard D., Petitta F. and Redwane H., \textit{Renormalized
solutions of nonlinear parabolic equations with diffuse measure data,
Manuscripta Math. }141 (2013), 601-635$.$

\bibitem {BlPo1}Blanchard D. and Porretta A., \textit{Nonlinear parabolic
equations with natural growth terms and measure initial data,} Ann.Scuola
Norm. Su. Pisa Cl Sci. 30 (2001), 583-622.

\bibitem {BlPo}Blanchard D. and Porretta A., \textit{Stefan problems with
nonlinear diffusion and convection,} J. Diff. Equ. 210 (2005), 383-428.

\bibitem {BoGa89}Boccardo L. and Gallouet T., \textit{Nonlinear elliptic and
parabolic equations involving measure data,} J Funct. Anal. 87 (1989), 149-169.

\bibitem {BoGa92}Boccardo L. and Gallouet T., \textit{Nonlinear elliptic
equations with right-hand side measures,} Comm. Partial Diff. Equ. 17 (1992), 641--655.

\bibitem {BoGaOr96}Boccardo .L, Gallouet T., and Orsina L., \textit{Existence
and uniqueness of entropy solutions for nonlinear elliptic equations with
measure data, }Ann. Inst. H. Poincar\'{e} Anal. non Lin. 13 (1996), 539-555.

\bibitem {BDGO97}Boccardo L., Dall'Aglio A., Gallouet T. and Orsina L.,
\textit{Nonlinear parabolic equations with measure data, J. Funct. Anal. }147
(1997), 237-258\textit{. }

\bibitem {BoMuPu84}Boccardo L., Murat F., Puel J., \textit{R\'{e}sultats
d'existence pour certains probl\`{e}mes elliptiques quasilin\'{e}aires,} Ann.
Scuola. Norm. Sup. Pisa 11 (1984), 213--235.

\bibitem {BrFr}Brezis H. and Friedman A., \textit{Nonlinear parabolic
equations involving measures as initial conditions,} J.Math.Pures Appl. 62
(1983), 73-97.

\bibitem {ChQiWa}Chen X., Qi Y. and Wang M., \textit{Singular solutions of
parabolic p-Laplacian with absorption}, T.A.M.S. 3589 (2007), 5653-5668.

\bibitem {DAOr}Dall'Aglio A. and Orsina L.,\textit{\ Existence results for
some nonlinear parabolic equations with nonregular data}, Diff. Int. Equ. 5
(1992), 1335-1354.

\bibitem {DAOr2}Dall'Aglio A. and Orsina L., \textit{Nonlinear parabolic
equations with natural growth conditions and }$L^{1}$data, Nonlinear Anal. 27
(1996), 59-73.

\bibitem {DMOP}Dal Maso G., Murat F., Orsina L., and Prignet A.,
\textit{Renormalized solutions of elliptic equations with general measure
data,} Ann. Scuola Norm. Sup. Pisa, 28 (1999), 741-808.

\bibitem {DrPoPr}Droniou J., Porretta A. and Prignet A., \textit{Parabolic
capacity and soft measures for nonlinear equations,} Potential Anal. 19
(2003), 99-161.

\bibitem {DrPr}Droniou J. and Prignet A., \textit{Equivalence between entropy
and renormalized solutions for parabolic equations with smooth data,}
Nonlinear Diff Eq. Appl. 14 (2007), 181-205.

\bibitem {Fe}Fefferman C., \textit{Strong differentiation with respect to
measure,} Amer. J. Math. 103 (1981), 33-40.

\bibitem {Gm}Gmira A\textit{., On quasilinear parabolic equations involving
measure data}, Asymptotic Anal. 3 (1990), 43-56.

\bibitem {HeKilMa}Heinonen J., Kilpelainen T. and Martio O., Nonlinear
potential theory of degenerate elliptic equations, Oxford Science
Publications, 1993.

\bibitem {KP2}Kamin S. and Peletier L. A.\textit{,Source-type solutions of
degenerate diffusion equations with absorption}, Isr. J. Math. 50 (1985), 219-230.

\bibitem {KPV}Kamin S., Peletier L. A. and Vazquez
J.L.,\textit{\ Classification of singular solutions of a nonlinear heat
equation}, Duke Math. J. 58 (1989), 601-615.

\bibitem {KaVa}Kamin S. and Vazquez J. L., \textit{Singular solutions of some
nonlinear parabolic equations,} J. Analyse Math. 59 (1992), 51-74.

\bibitem {KiXi}Kilpelainen T. and Xu X., \textit{0n the uniqueness problem for
quasilinear elliptic equations involving measure,} Revista Matematica
Iberoamericana. 12 (1996), 461-475.

\bibitem {La}Landes, R., \textit{On the existence of weak solutions for
quasilinear parabolic initial boundary-value problems,} Proc. Royal Soc.
Edinburg Sect A, 89(1981), 217-237.

\bibitem {LePe}Leonori T. and Petitta F., \textit{Local estimates for
parabolic equations with nonlinear gradient terms,} Calc. Var. Partial Diff.
Equ. 42 (2011), 153--187.

\bibitem {LeoPel}Leoni F. and Pellacci B., \textit{Local estimates and global
existence for strongly nonlinear parabolic equations with locally integrable
data}, J.\ Evol. Equ. 6 (2006), 113-144.

\bibitem {Li}Lions J.L., Quelques m\'{e}thodes de r\'{e}solution des
probl\`{e}mes aux limites non lin\'{e}aires, Dunod et Gauthiers-Villars (1969).

\bibitem {MaVe}Marcus M. and V\'{e}ron L., \textit{Initial trace of positive
solutions of some nonlinear parabolic equations,} Comm. Partial Diff. Equ. 24
(1999), 1445-1499.

\bibitem {NQVe}Nguyen Quoc H. and V\'{e}ron L., \textit{Quasilinear and
Hessian equations with exponential reaction and measure data, }Arxiv\textit{
1305-4332. }

\bibitem {Pe07}Petitta F., Asymptotic behavior of solutions for linear
parabolic equations with general measure data, C. R. Acad. Sci. Paris, Ser. I
344 (2007) 571--576.

\bibitem {Pe08}Petitta F., \textit{Renormalized solutions of nonlinear
parabolic equations with general measure data,} Ann. Math. Pura Appl. 187
(2008), 563-604.

\bibitem {PePoPor}Petitta F., Ponce A. and Porretta A., \textit{Diffuse
measures and nonlinear parabolic equations}, J. Evol. Equ. 11 (2011), 861-905.

\bibitem {Por99}Porretta A., \textit{Existence results for nonlinear parabolic
equations via strong convergence of truncations,} Ann. Mat. Pura Apll. 177
(1999), 143-172.

\bibitem {Por02}Porretta A.,\textit{\ Nonlinear equations with natural growth
terms and measure data, }E.J.D.E., Conference 9 (2002), 183-202.

\bibitem {PhVe1}Phuc N. and Verbitsky I., \textit{Quasilinear and Hessian
equations of Lane-Emden type}. Ann. of Math. 168 (2008), 859--914.

\bibitem {Pr97}Prignet A., \textit{Existence and uniqueness of "entropy"
solutions of parabolic problems with }$L^{1}$\textit{ data,} Nonlinear Anal.
TMA 28 (1997), 1943-1954.

\bibitem {Ra}Rakotoson J., \textit{\ Some quasilinear parabolic equations},
Nonlinear Anal. 17 (1991), 1163-1175.

\bibitem {Ra0}Rakotoson J., \textit{A compactness Lemma for quasilinear
problems: application to parabolic equations}, J. Funct. Anal. 106 (1992), 358-374.

\bibitem {Ra1}Rakotoson J., \textit{Generalized solutions in a new type of
sets for problems with measures as data,} Diff. Int. Equ. 6 (1993), 27-36.

\bibitem {St}Stampacchia,G., \textit{Le probl\`{e}me de Dirichlet pour les
\'{e}quations elliptiques du second ordre \`{a} coefficients discontinus},
Ann. Inst. Fourier, 15 (1965), 189-258.

\bibitem {Xu}Xu, X., \textit{On the initial boundary-value-problem for }%
$u_{t}-$div$(\left\vert \nabla u\right\vert ^{p-2}\left\vert \nabla
u\right\vert )=0$, Arch. Rat. Mech. Anal. 127 (1994), 319-335.

\bibitem {Zh}Zhao J., \textit{Source-type solutions of a quasilinear
degenerate parabolic equation with absorption,} Chin. Ann. of MathB,1 (1994), 89-104.

\bibitem {Zh2}Zhao J., \textit{Source-type solutions of a degenerate
quasilinear parabolic equations,} J. Diff. Equ. 92 (1991), 179-198.
\end{thebibliography}
\end{document}